\documentclass[a4paper,11pt]{article}
\usepackage{graphics,color}
\usepackage{amsmath}
\usepackage{bbm}
\usepackage{amssymb}
\usepackage{mathrsfs}
\usepackage{esint}
\usepackage{cancel}
\usepackage{cite}
\usepackage{verbatim}
\usepackage{float}
\usepackage{graphicx}
\usepackage{amsthm}
\usepackage{enumitem}
\usepackage{textcomp}
\usepackage{subfig}
\usepackage{wrapfig}
\usepackage{hyperref}

\newcommand{\dx}{\, {\rm d}x}
\newcommand{\dy}{\, {\rm d}y}
\newcommand{\dz}{\, {\rm d}z}

\newcommand{\dt}{\, {\rm d}t}

\newcommand{\di}{\, {\rm d}}
\newcommand{\strokedint}{\fint}
\newif\ifdraft
\draftfalse 

\drafttrue

\def\eps{\varepsilon}

\newcommand{\F}{\mathcal{F}}

\def\L{\mathcal L}

\newcommand{\M}{\mathcal{M}}

\newcommand{\nada}[1]{}

\newcommand{\R}{\mathbb{R}}

\newcommand{\res}{\mathop{\hbox{\vrule height 7pt width 0.5pt depth 0pt
\vrule height 0.5pt width 6pt depth 0pt}}\nolimits}

\def\Z{\mathcal Z}

\numberwithin{equation}{section}
\mathchardef\emptyset="001F

\newtheorem{theorem}{Theorem}[section]
\newtheorem{definition}[theorem]{Definition}
\newtheorem{prop}[theorem]{Proposition}
\newtheorem{cor}[theorem]{Corollary}
\newtheorem{lemma}[theorem]{Lemma}

\theoremstyle{definition}
\newtheorem{remark}[theorem]{Remark}

\title{$\Gamma$-Convergence {for} plane to wrinkles transition problem }
\author{
Peter Bella\footnote{
Technische Universit\"at Dortmund
Fakult\"at f\"ur Mathematik, 
44227 Dortmund, Germany. E-mail: peter.bella@udo.edu
                      }
\and
Roberta Marziani\footnote{Dipartimento di Ingegneria e Scienze dell'Informazione e Matematica, 67100 L'Aquila, Italy.
	E-mail: roberta.marziani1@univaq.it}
}

\begin{document}

\maketitle
\begin{abstract}
We consider a variational problem modeling transition between flat and wrinkled region in a thin elastic sheet, and identify the $\Gamma$-limit as the sheet thickness goes to $0$, thus extending the previous work of the first author [Bella, ARMA 2015]. The limiting problem is scalar and convex, but constrained and posed for measures. For the $\Gamma-\liminf$ inequality we first pass to quadratic variables so that the constraint becomes linear, and then obtain the lower bound using Reshetnyak's theorem. The construction of the recovery sequence for the $\Gamma-\limsup$ inequality relies on  mollification of quadratic variables, and careful choice of multiple construction parameters. Eventually for the limiting problem we show existence of a minimizer and equipartition of the energy for each frequency. 

\end{abstract}

\section{Introduction}\label{sec:introduction}

This paper is about fine analysis of minimizers of a nonconvex variational problem 
which describes wrinkling of thin elastic sheets. 

Motivated by some physical experiments with thin elastic sheets~\cite{DaShCe12,DaSh++11,GeBeMe04}, the first author, in his PhD thesis~\cite{Be12} (see also~\cite{BeKo14+}), considered a specific variational problem describing deformations of a thin elastic sheet $\Omega_h\subset\R^3$ of thickness $h$ and cross section of annular shape. The elastic energy, corresponding to a deformation $v\colon\Omega_h\to\R^3$, consists of a membrane term, measuring stretching and compression of the sheet, and of a bending term, which penalizes curvature. As a proxy for the energy one can think  of
\begin{equation}\label{proto_en}
 E_h(v) = \int_{\Omega_h} \mathrm{dist}^2(\nabla v, { \mathrm{SO}}(3)) \dx+ h^2 \int_{\Omega_h} |\nabla^2 v|^2\dx. 
\end{equation}
The membrane part is non-convex, possibly giving rise to oscillations. In contrast, the latter bending part is convex and of higher-order, thus regularizing the problem. Since the bending resistance is related to the sheet thickness $h$, the magnitude of this contribution asymptotically vanishes in the limit $h \searrow 0$. 

The physics approach to tackle these problems consists of a specific choice of an ansatz (guess) for the form (shape) of a minimizer. In other words, one restricts the analysis to a class of competitors having specific characteristics, and look for a minimizer of the energy within that class.
On the other hand, the rigorous analytical approach does not make any assumptions on the form of a minimizer, i.e., the energy is minimized over all possible deformations. The problem in \eqref{proto_en} being non-convex, hence possibly possessing many (local) minimizers or critical points, the first step is to understand the minimal value of the energy, with possibly learning some clues by which deformations is this minimal value, at least approximately, achieved. 

\begin{figure}
 	\centering
    \includegraphics[width=0.5\textwidth]{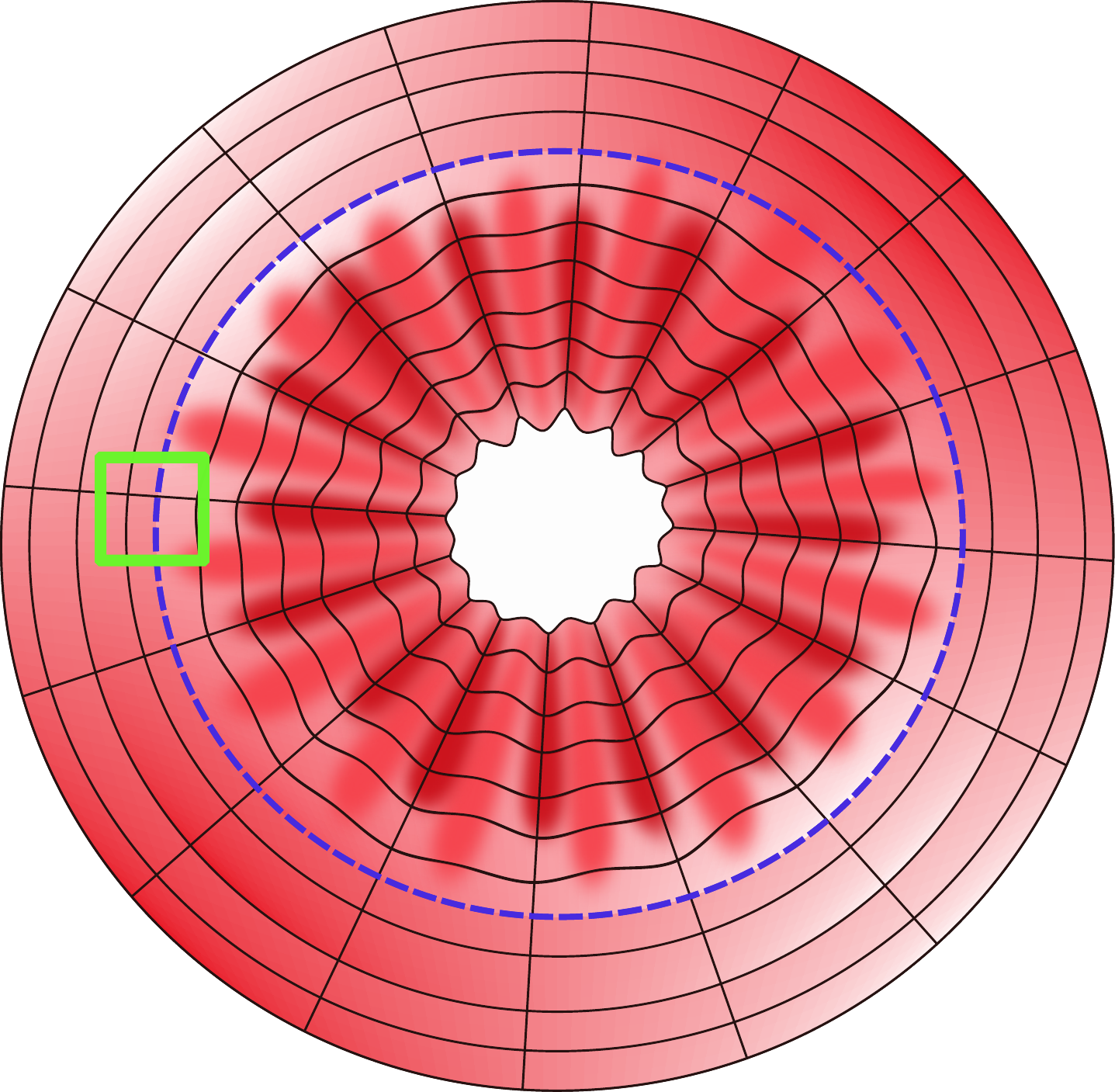} 
	\caption{Elastic annular membrane {stretched}  in the radial direction. The blue dotted curve represents the free boundary that separates the outer stretched region from the inner wrinkled one.}\label{fig1}
\end{figure}

Hence, we first try to identify the minimal value of the energy. Precisely, in the present situation, the goal is to understand its dependence on the (small) sheets thickness $h$. It turns out that the minimal value $\min_v E_h(v)$ consists of a leading zeroth-order term $\mathcal E_0$ (coming from the stretching of the sheet) plus a linear correction in $h$, which corresponds to the cost of \emph{wrinkling} of the sheet~\cite{Be12,BeKo14+}. More precisely, there exist two constants $0 < C_0 < C_1 < \infty$ such that for any thickness $0 < h < 1$ there holds 
\begin{equation}\label{eq01}
 \mathcal E_0 + C_0 h \le \min_{v} E_h(v) \le \mathcal E_0 + C_1 h.
\end{equation}
The wrinkling serves as a mechanism to relieve compressive stresses, which are caused by specific geometrical effects. 
An alternative to wrinkling would be simply compression, which contributes to the membrane part at the order $O(1)$. 
Hence, in the case of small thickness (present situation) compression is much less energetically favorable ($O(1)$ vs $O(h^2)$), and thus not expected. 

The identified linear scaling law~\eqref{eq01} in $h$ for the minimal value of the energy raised a lot of discussion among the physics community, having improved their ansatz-based prediction by a factor of $\log h$ (i.e. $h$ in~\cite{BeKo14+} vs $h {(|\log h|+1)}$ in~\cite{DaShCe12}). It turns out that this discrepancy is related to a suboptimal choice of the ansatz close to the interface between the wrinkled and the flat region. Moreover, 
{the upper bound in~\eqref{eq01} is achieved through a complex construction involving branching effects, a pattern which was not observed experimentally.}

To shed light on this discrepancy, the first author considered a variational problem modelling the transition region~\cite{BellaARMA} with the aim of better understanding the behavior of the minimizer in that region.
Working at the level of the energy, this means to consider the quantity $\frac{\min_u E_h(u) - \mathcal{E}_0}{h}$,{ which not only is bounded away} from $0$ and $\infty$ (see~\eqref{eq01}), 
but as $h \searrow 0$ it actually converges to some value $\sigma$ (as proven in~\cite{BellaARMA}). Even though the value $\sigma$ is characterized as a limit of minima of simpler scalar and convex variational problems, it \emph{does not} provide any information on the form of sequence of minimizers. 

In that respect, the goal of this paper is to overcome this shortcoming by showing $\Gamma$-convergence of the functionals
$\frac{E_h - \mathcal{E}_0}{h}$ as $h \searrow 0$. As usual, a consequence of $\Gamma$-convergence is convergence of minimizers $u_h$ of $E_h$ to a minimizer of the limiting problem, hence providing information on $u_h$, at least for $0 < h \ll 1$. Denoting by $\mathcal F_\infty$ the $\Gamma$-limit functional (see~\eqref{def:F_infty} below), {it turns out that  as expected from~\cite{BellaARMA}, $\mathcal F_\infty$ is scalar and convex, thus}
possibly much easier to analyze than the original $E_h$. Nevertheless, the study of minimizers of $\mathcal F_\infty$ is still far from trivial and we postpone it to a future work -- except for some preliminary results collected in Section~\ref{sec:regularity}.

\medskip
There are many areas of material science, most of them falling within a class of energy-driven pattern formation~\cite{KohnICM}, where the idea to study energy scaling laws for variational problems turned out to be very fruitful. The common features of these problems is the presence of a nonconvex term in the energy, which is regularized by a higher-order term with a small prefactor. This small parameter (for now denoted $\eps$) has different meanings: thickness in the case of elastic films, inverse Ginzburg-Landau parameter in the theory of superconductors, strength of the interfacial energy for models of shape-memory alloys or micromagnetics, to name just few. As $\eps \searrow 0$, the oscillations caused by the nonconvexity are less penalized, giving the energy more freedom to form patterns/microstructure. 

The first paper in this direction, in the context of shape memory alloys, is a seminal work of Kohn and Müller~\cite{KoMu94}, where they studied a toy problem to model the interface in the austenite-martensite phase transformation. They showed that the energy minimum scales like $\epsilon^{2/3}$, which was in contrast with the scaling $\eps^{1/2}$, widely accepted in the physics community. More precisely, the physics arguments were based on an ansatz of ``one-dimensional'' structure of minimizers, whereas Kohn and Müller used a branching construction to achieve lower energy. While they did not show the form of minimizers, they provided localized (in one direction) estimates on the energy distribution for the minimizer -- thus providing hints on scales used for branching. Subsequently, Conti~\cite{Co00} used an intricate upper bound construction to show localized energy bounds (in both directions), which in particular implies asymptotical self-similarity of the minimizer close to the interface. The analysis of the toy model was later generalized in several directions, for example analysis based on energy scalings laws for the cubic-to-tetragonal phase transformation - e.g. rigidity of the microstructure~\cite{capellaotto1,capellaotto2} or study of the energy barrier for the nucleation in the bulk~\cite{KnuepferKohnOtto} and at the boundary~\cite{BeGo15}. 
In that respect it is worth to also mention recent works of Rüland and Tribuzio~\cite{RuTr23,RuTr23b}, where a novel use of Fourier Analysis allows to obtain sharp lower bounds on the energy on a more advanced model. 

The work of Kohn and Müller initiated many developments in other areas of material science to study pattern formation driven by the energy minization, for example in micromagnetics~\cite{Ot02,BrKn23,RiRo23}, island growth on epitaxially strained films~\cite{BellaGoldmanZwicknagl}, diblock copolymers~\cite{Ch01}, optimal design~\cite{KoWi16,KoWi14}, or superconductors~\cite{Se15}. Picking one of them as an example, the Ginzburg-Landau model describes behavior of superconductors in different regimes of the applied magnetics field. While for extreme values of the magnetic field (very small or very large) there is only one (normal or superconducting) phase, for intermediate values of the field the mixed states consisting of many vortices are observed. There the leading order energy characterizes the number of vortices, and the next order in the energy describes interaction between them (see~\cite{Se15} for a survey, \cite{RoSaSe23} for analysis in three spatial dimensions, and~\cite{PeSe17} for {a} similar work in the context of $2d$ Coulomb gases). 

\medskip
The models for wrinkling of thin elastic films have similar feature, with the leading order term  in the energy {expansion} encoding the wrinkled regions while the next term in the energy expansion being related to the form (e.g. lengthscale) of wrinkling. The relevant physical object being a two dimensional (thickened) surface in $\mathbb R^3$, the local energy expense of a deformation $u$ is encoded using two principal values of a $3\times2$ matrix $\nabla u$ -- heuristically, singular value greater or smaller than $1$ corresponds to a tension or a compression, respectively. Wrinkling being an energetically efficient alternative to a compression, we expect it to appear in the case of (at least) one singular value being less than one. 

{A compressed elastic} sheet can feel the compression in one (``tensile wrinkling'') or both directions (``compressive wrinkling''). A class of problems falling into the latter category for which the energy scaling laws were identified include for instance blistering/delamination problem (with~\cite{KoNg13,BeKo15,BoCoMu16,PG20} or without~\cite{JiSt01,BBCoDeMu02,BBCoDeMu00} substrate effects), crumpling of elastic sheets~\cite{CoMa08,Ve04}, or analysis of conical singularities in elastic sheets~\cite{Ol19,Ol18}. A common feature of this problem is degeneracy of the relaxed energy: the minimum of the relaxed energy $\mathcal E_0$ equals zero, and more importantly it is achieved by many different minimizers, making the analysis of the next order {expansion} of the energy often difficult. 

In contrast, tensile wrinkling problems usually have relaxed problem with unique minimizer, making the analysis of the next order term (which describes wrinkling) more accessible. The need for compression usually comes from the prescribed boundary conditions (as for example in the raft problem~\cite{BrKoNg13,benny-raft}, twisted ribbon~\cite{KoOB18}, hanging drapes~\cite{romandrapes, BeKo17drapes}, or compressed cylinder~\cite{To18}), through prescribed incompatible strain~\cite{BeKo14,MaSh19} or curvature effects~\cite{BeKo17,BeRo23,To21}. 

\medskip

The model we consider here is a mixture of the first and the second case, i.e., it is driven both by the boundary conditions as well as prescribed nontrivial {metric (prestrain)}. The latter should mimic the need to ``waste the length'' in one direction, this need coming from geometric effects in our original motivation~\cite{BeKo14+}. More precisely, in~\cite{BeKo14+} an elastic annulus is stretched radially with stronger inner loads, forcing some of the concentric circles of material to move closer to the center.  Pushing some circles into less space naturally force compression or wrinkling out of plane, while the circles towards outer boundary stay planar (and are actually stretched in the azimuthal direction). As we will see, it is crucial that the amount of arclength grows linearly in the distance from the free boundary (between the wrinkled and planar region) and not slower (e.g. quadratically) -- the latter case {is expected to} be quite boring with the minimizer using only \emph{one} frequency. In contrast, the present problem requires infinitely many frequencies, in particular near the transition higher and higher frequencies are needed. 


The rest of this section will provide an overview of our results and organization of the paper. As in~\cite{BellaARMA}, we consider a specific thickness dependent energy $E_h$ (see \eqref{energy1} for its precise definition), a model problem describing transition between planar and wrinkled region in thin elastic sheet, and are interested to understand structure of minimizers of the energies as $h \searrow 0$. We consider a thin elastic sheet of thickness $h$ and cross section of rectangular shape $[-1,1]\times [-1,1] \subset \R^2$, which represents a piece of the elastic annulus {depicted above in { Figure}~\ref{fig1} by a green region}, and assume the sheet is i) stretched in the $x$-direction, and ii) stretched/compressed in the $y$-direction proportional to $x$ (i.e., it is {unstrained} for $x=0$, stretched in the $y$-direction in the left half and compressed in the right half of the domain). The streching/compression in the $y$-direction is modeled via prescribed metric together with periodic boundary conditions {at the top and bottom boundary.}

To relax the compression in the region $\{ x > 0 \}$ we expect the sheet to wrinkle, with the lengthscale of wrinkles of order $h^{1/2}$~\cite{BeKo14+}. In order to analyze the limit of $\frac{E_h-\mathcal{E}_0}{h}$ as $h \searrow 0$, we rescale the $y$-variable by $h^{-1/2}$ so that the wrinkles lengthscales stay of order $1$, and the {out-of-plane displacement has} chance to converge to some limiting shape. Though the rescaling cause changes of the domain to $[-1,1] \times [-L,L]$ for $L := h^{-1/2} \to \infty$, in particular the $\Gamma$-limit of the functionals is not clear due to changing domain, we pass to the Fourier space to avoid these complications. More precisely, we rewrite the energy using Fourier expansion in $y$, with $L$ {appearing through} the summation set $\frac{\pi \mathbb{Z}}{L}$. Heuristically, as $L \to 0$ the Fourier sum will turn into an integral, hence there is a hope for a limiting functional to make sense. 

This strategy was successfully pursued by the first author in~\cite{BellaARMA}, by observing i) the out-of-plane displacement $u$ being the only relevant quantity to be monitored in this limit, and ii) for fixed (large) $L = h^{-1/2}$ the minimum of the excess energy $\frac{ E_h - \mathcal{E}_0}{h}$ is well approximated by minimum of a \emph{scalar, convex, and constrained} variational problem for $u$ of the form
\begin{equation}\label{eq:SL}
 \mathcal S_L(u) := \int_0^1 \fint_{-L}^L u_{,x}^2 + u_{,yy}^2 \dx \dy \quad \textrm{ subject to }\quad \fint_{-L}^L u_{,y}^2(x,y) \dy = 2x + o(1) \quad \textrm{for a.e. } x \in (0,1)
\end{equation}
Denoting by $a_k(x)$ the Fourier coefficients in $y$ of $u(x,\cdot)$, we can rewrite
\begin{equation*}
\mathcal{S}_L(u)=\int_0^1 \sum_{k \in \frac{\pi \mathbb Z}{L}}( \dot a_k^2(x) + a_k^2(x)k^4)\dx,\quad \text{ and }\quad \fint_{-L}^L u_{,y}^2(x,y) \dy= \sum_{k \in \frac{\pi \mathbb Z}{L}}a_k^2(x)k^2\,, 
\end{equation*}
where ``dot'' denotes the derivative.
%
The main achievements of~\cite{BellaARMA} was to show that minima of $\mathcal S_L$ converge, and then to construct a recovery sequence for the original energy $E_h$, including construction of the in-plane displacement. Since the elastic energy $E_h$ includes all second derivatives of $u$, and not only $u_{,yy}$ which appears in $\mathcal S_L$, regularity statement for the minimizers of $\mathcal S_L$'s played a crucial role for the construction of the recovery sequence. 

The analysis of minima of $\mathcal S_L$ from \cite{BellaARMA} completely avoided the {notion } of convergence of minimizers, which needs to be an integral part of a $\Gamma$-convergence which we study here. To avoid the issue of nonlinear constraint we use quadratic variables (i.e. monitoring $b_k := a_k^2$ instead of $a_k$), which turns the constraint into a linear one. The second term in the energy {$\mathcal{S}_L$} becomes also linear, while the first term can be rewritten as $\frac{(\dot b_k)^2}{4b_k}$. One disadvantage of this approach is the $L^1$-framework,  which naturally leads the limit functional to be defined on the space of measures.
However, the constraint provides a good pointwise control in $x$, in particular the limiting measure can be written as a product of $\dx$ and $x$-dependent measures in $k$. The lower bound argument (Proposition~\ref{prop:lower_bound}) is obtain using Reshetnyak Theorem. 

The upper bound (construction of a recovery sequence) is much more tricky since it needs to be done for any ``limiting'' measure with finite energy, in contrast with~\cite{BellaARMA}, where it was done just for one (more regular) minimizer. The proof of the upper bound (Proposition~\ref{prop:upb}) consists of several steps: 
\begin{enumerate}
 \item Given a limiting measure, to obtain $a_k$'s we will ``discretize'' the measure in the $k$-variable (Lemma~\ref{lem:discretisation}). Moreover, using smoothing of $a_k$'s (more precisely {of} $a_k^2$), for which we need to extend the coefficients $a_k$ from $[0,1]$ via dilation into larger interval $[0,1+]$, we get a good starting point for the construction.
 \item Careful choice of the smoothing scale $\eps(L)$ together with few other parameters allow for definition of the out-of-plane displacement (see Lemma~\ref{lem:moll}), which is then basis for the construction of the in-plane displacement as well as estimates on the excess energy (Proposition~\ref{prop:upb}). 
\end{enumerate}

The paper is organized as follows: in the next Chapter we provide a derivation of the energy, including the functional-analytical framework in form of measures with well-behaved distributional derivatives in $x$, as well as rewriting the energy to a form compatible also with this framework. {Finally,} we state the main $\Gamma$-convergence result of this paper Theorem~\ref{theo:main_theo}. In Chapter~\ref{sec:preliminaries} we show how to disintegrate the limiting measures, while in the {subsequent} Chapter we show the compactness for a sequence with excess energy (Proposition~\ref{prop:compactness}). Subsequently we also show the lower bound (Proposition~\ref{prop:lower_bound}). The upper bound construction is content of Chapter~\ref{sec:upb}. Eventually, in the last Chapter we state and prove existence of minimizer (as a measure) for the limiting energy as well as pointwise (in $k$) equipartion of the energy for this minimizer {(see Theorem~\ref{thm:minimizers})}. A finer analysis of this minimizer will be pursued in a future work of the authors. 

\section{Setting of the problem and main results}\label{sec:setting}
We start by collecting some notation we will use throughout the paper.\\

\noindent
\textbf{Notation.} 
\begin{enumerate}[label=$(\alph*)$]
	\item $\chi_A$ denotes the characteristic function of the set $A$;
	\item $\mathcal{L}^1$ denotes the 1-dimensional Lebesgue measure;
	\item $\delta_k$ denotes the Dirac measure on $k\in\R$;
\item $\mathcal{M}_b(A)$ denotes the space of bounded Radon measures on $A$ with $A\subset\R^2$ Borel measurable;
\item $\mathcal{M}^+_b(A)$ denotes the subspace of  $\mathcal{M}_b(A)$ of positive bounded Radon measures;
\item For a function $f\colon A\subset\R\to\R$ we denote by $\dot f(x)$  and $\ddot{f}(x)$ the first and the second derivative, respectively;
\item For a function $u\colon A\subset\R^2\to\R$ we denote by $u_{,\small\underbrace{x\dots x}_{i \text{ times}}\small\underbrace{y\dots y}_{j \text{ times}}}$ its partial derivative 
$$D^{i+j}u(x,y)= \frac{\di^j}{\dy^j}\frac{\di^i}{\dx^i}u(x,y)\,,\quad i,j\in\mathbb{N},\, 1\le i+j\le3\,;$$
\item For a measure $\mu\in \mathcal{M}_b(A)$ we denote by $\mu_{,x}$ its distributional derivative with respect to the first variable;
\item For a measure $\mu\in \mathcal{M}_b(A)$ we denote by $|\mu|\in \mathcal{M}^+_b(A)$ its total variation;
\item For $\tilde\mu=(\mu_1,\mu_2)\in ( \mathcal{M}_b(A))^2$ we analogously denote by $|\tilde\mu|\in \mathcal{M}^+_b(A)$ its total variation;
\item For $\mu_1\in \mathcal{M}_b(A)$, $\mu_2\in \mathcal{M}^+_b(A)$ we write $\mu_1\ll\mu_2$ if $\mu_1$ is absolute continuous with respect to $\mu_2$ and we indicate by $\frac{\di\mu_1}{\di\mu_2}\in L^1(A,\mu_2)$ the associated density {(Radon-Nikod\'ym derivative)}. 
\end{enumerate}

\noindent
\textbf{The Model.}
Let us now describe the model (energy) for the transition between the flat and wrinkled region, which {the first author} already analyzed in~\cite{BellaARMA}. Instead of considering the annular elastic sheet as in~\cite{BeKo14+}, we consider only a rectangular piece (cut off from the sheet) near the transition region, in particular simplifying the problem by avoiding the need to work in the radial geometry. The annular sheet in~\cite{BeKo14+} is stretched in the radial direction and the concentric circles close to the transition region are stretched/compressed proportional to the distance from the free boundary. We will model the radial stretching by dead tension loads in the horizontal direction, while the stretching/compression in the vertical direction will be modeled by prescribing a non-euclidean metric of the form $\dx + (1+x)\dy$. Moreover, the rectangle being part of the annulus, we prescribe periodic boundary conditions in the vertical direction. 

It is physically natural~\cite{DaShCe12} and mathematically convenient to use ``small-slope'' geometrically linear F\"oppl-von K\'arm\'an theory. {In the membrane part of the energy the in-plane displacement is represented via the \emph{linear} strain while the out-of-plane displacement is kept non-linear (quadratic). The bending part is modeled by simply $L^2$ norm of the Hessian of the out-of-plane displacement instead of $L^2$ norm of the second fundamental form.}
Denoting by $w=(w_1,w_2)$ and $u$ the in-plane and out-of-plane displacement respectively, the elastic energy $E_h$ (normalized per unit thickness) has the form
\begin{equation}\label{energy1}
	\begin{split}
	 E_h(w,u) :=& \frac 12 \int_{-1}^1 \int_{-1}^1 |e(w) + \frac 12 \nabla u \otimes \nabla u - x e_2 \otimes e_2|^2 \dx\dy
	 \\&+\frac12
 \int_{-1}^1 \int_{-1}^1	  h^2 |\nabla^2 u|^2 \dx \dy 
	 - \int_{-1}^1 (w_1(1,y) - w_1(-1,y) )\dy.
	\end{split}
\end{equation}
Here $e(w) := (\nabla w + \nabla^T w)/2$ denotes the symmetric gradient of $w$ and $x e_2 \otimes e_2$ is the deviation of the prescribed metric from the euclidean one. The third integral models the applied tensile dead loads in the horizontal direction. The factor $1/2$ in front of the elastic energy is chosen for convenience, and can be changed to any factor using simple rescaling of $w$ and $u$. Finally, we assume the displacement $(w,u)$ is $2$-periodic in the second variable. 

The behavior of $E_h$ as $h \to 0$ at the leading order is well understood using relaxation techniques~\cite{pipkin1} (also called tension-field theory in the mechanics community). Applied to $E_h$ from~\eqref{energy1}, in the limit $h \to0$ the bending term simply disappears, and the integrand in the membrane term changes to $\min_{A \ge 0} |e(w) - x e_2 \otimes e_2+A|^2$. Hence, one can explicitly compute the (unique) minimizer of the relaxed energy ($w_2=0$ and $w_1 = x$) and its minimum $-2 + \frac 13 = -\frac 53 =: \mathcal{E}_0$. 

From~\cite{BeKo14+} we know that the next term in the energy $E_h$ scales linearly in $h$, hence the right quantity to look at is the rescaled \emph{excess energy} $\frac{E_h - \mathcal{E}_0}{h}$. For $x > 0$ one expects that the sheet wrinkles out-of-plane in the $y$-direction, in order to offset $-xe_2 \otimes e_2$ with $u_{,y}^2$. The linear scaling in $h$ predicts $h^2|\nabla^2 u|^2 \sim h$, in particular $u_{,yy}$ (its largest component) to be of order $h^{-1/2}$. In particular, the scale of wrinkles in the bulk should be reciprocal of this value, i.e., $h^{1/2}$. Not surprisingly, this is also the scale used in the upper bound construction in~\cite{BeKo14+}. 

In order to analyze limiting form of the wrinkles as $h \to 0$ we rescale the $y$-variable by a factor $L := h^{-1/2}$, so that the characteristic lengthscale of wrinkles becomes $1$. Precisely, after performing the change of variables
\begin{equation*}
	\hat w_1(x,y):=w_1(x,L^{-1}y), \quad \hat w_2(x,y):=Lw_2(x,L^{-1}y),\quad \hat u(x,y):=Lu(x,L^{-1}y),
\end{equation*}
the energy $E_h$ becomes (see~\cite[page 630]{BellaARMA} for a straightforward {algebraic manipulation})

\begin{align}\label{eqn:RL}
\mathcal E_L(w,u):=& \strokedint_{-L}^L \int_{-1}^1
 \biggl(\biggl( w_{1,x}+\frac{u_{,x}^2}{2L^2}-1 \biggr)^2 -1\biggr)\dx\dy
+ \strokedint_{-L}^L \int_{-1}^1 \biggl( w_{2,y}+\frac{u_{,y}^2}{2} -x \biggr)^2 \dx\dy \nonumber \\
&
+  \frac{1}{L^2}\strokedint_{-L}^L \int_{-1}^1  {\Big(L^2 w_{1,y}+w_{2,x}+u_{,x}u_{,y} \Big)^2}  \dx\dy
+   \frac{1}{L^2}\strokedint_{-L}^L \int_{-1}^1  \big(u^2_{,x}+u^2_{,yy}\big)\dx \dy\\
&+  \frac{1}{L^4}\strokedint_{-L}^L \int_{-1}^1  \Big(2u^2_{,xy}+\frac{1}{L^2}u^2_{,xx}\Big)\dx\dy\,.
\end{align}
Thus, the functional is defined as $\mathcal E_L\colon \mathcal{A}_L^{\rm in}\times	\mathcal{A}_L^{\rm out}\to[0,+\infty]$, 
where the function spaces describing admissible deformations have the form 
\begin{equation}\label{def:in-plane}
	\mathcal{A}_L^{\rm in}:=\Big\{w=(w_1,w_2)\in W^{1,2}((-1,1)\times\R;\R^2)\colon w(x,\cdot) \text{ is $2L$-periodic $\forall x\in(-1,1)$}\Big\},
\end{equation}
\begin{equation}\label{def:out-of-plane}
		\mathcal{A}_L^{\rm out}:=\Big\{u\in W^{2,2}((-1,1)\times\R)\colon u(x,\cdot) \text{ is $2L$-periodic $\forall x\in(-1,1)$}\Big\}.
\end{equation}
Furthermore, the $\frac{E_h - \mathcal{E}_0}{h}$ turns into 
$\mathcal{F}_L\colon	\mathcal{A}_L^{\rm in}\times	\mathcal{A}_L^{\rm out}\to\R$ defined as
\begin{equation}\label{def:F_L}
	\mathcal{F}_L(w,u):=L^2(\mathcal E_L(w,u)-\mathcal E_0)\,,
\end{equation}
where $\mathcal E_0=-\frac53$ is as above the minimum of the relaxed energy, so that 
\begin{equation}\label{def:F_L_explicit}
\begin{split}
\mathcal{F}_L(w,u)=L^2& \strokedint_{-L}^L \int_{-1}^1
\biggl( w_{1,x}+\frac{u_{,x}^2}{2L^2}-1 \biggr)^2 \dx\dy-\frac{L^2}3
+L^2 \strokedint_{-L}^L \int_{-1}^1 \biggl (w_{2,y}+\frac{u_{,y}^2}{2} -x \biggr)^2 \dx\dy \nonumber \\
&
+  \strokedint_{-L}^L \int_{-1}^1  \Big( L^2 w_{1,y}+w_{2,x}+u_{,x}u_{,y} \Big)^2  \dx\dy
+   \strokedint_{-L}^L \int_{-1}^1  \big(u^2_{,x}+u^2_{,yy}\big)\dx \dy\\
&+  \frac{1}{L^2}\strokedint_{-L}^L \int_{-1}^1  \Big(2u^2_{,xy}+\frac{1}{L^2}u^2_{,xx}\Big)\dx\dy\,.
\end{split}
\end{equation}
%

Before we rigorously proceed further, let us discuss heuristically the form of functional $\mathcal F_L$ and its implications. Most of the terms in the energy are of quadratic nature, and since in addition we are dealing with oscillatory objects defined on longer and longer intervals, it is natural to look at the problem in the Fourier space. 

Expecting the limit of $\mathcal F_L$ to exists (in particular having the {minimizing} sequence bounded as $L \to \infty$), both integrals on the first line needs to (quickly) converge to $0$. The first integral can easily achieve that by simply choosing $w_1 \sim x + o(L^{-1})$ and $u_{,x}$ not too big, the smallness of the {second} integral (after integration in $y$ and using periodicity of $w$) implies the constrain $\fint_{-L}^L u_{,y}^2 \dy = 2x + o(L^{-1})$. 

In order to have continuity of the constraint in the limit $L \to \infty$, and also for other reasons which will be apparent later, we will work with squares of the Fourier coefficients and suitably defined measures as primary objects of studies.
In the following we denote by $k\in\R$ the variable corresponding to the Fourier transform in the $y$-variable. {Moreover,} we use the same notation to denote the second variable when {working} with measures.
\begin{definition}[Measures $\mu^L$ and $\mu^L_{,x}$]\label{def:muL}
	Let $u\in \mathcal{A}_L^{\rm out}$.  We denote by $\mu^L(u)\in\mathcal{M}_b^+((-1,1)\times\R)$ the measure given by 
	\begin{equation*}
		\mu^L(u):= \sum_{
			k\in \frac{\pi \mathbb{Z}}{L}} a(x,k)  \L^1\res{(-1,1)}\times \delta_k\,,
	\end{equation*}
with $$	 {a}(x,k):=k^2a_k^2(x)\,,$$and $a_k\in W^{2,2}(-1,1)$ being the $k$-th Fourier coefficient of $u(x,\cdot)$ for all $x\in (-1,1)$, that is
\begin{equation}\label{def:coef_a}
	a_k(x):=
	 \begin{cases}\displaystyle
	 \sqrt 2	\strokedint_{-L}^{L} u(x,y) \sin(ky) \dy&k\in \dfrac{\pi \mathbb{Z}}{L}, k>0 \,,\\[1em]
	 	\displaystyle
	 \sqrt 2 \strokedint_{-L}^{L} u(x,y) \cos(ky) \dy&k\in \dfrac{\pi \mathbb{Z}}{L}, k< 0\,,\\[1em]
	 	\displaystyle
	 \strokedint_{-L}^Lu(x,y)\dy&k=0\,.
	 \end{cases}
\end{equation}
Moreover we denote by $\mu^L_{,x}(u)$ the distributional $x$-derivative of $\mu^L(u)$.
	\end{definition}
\begin{remark}\label{rem:int-byparts}

\begin{enumerate}[label=$(\roman*)$]
		\item\label{rem-i} The distributional $x$-derivative of a measure $\mu\in\mathcal{M}_b^+((-1,1)\times\R)$ is defined as follows:
		for all $\varphi\in C^\infty_c((-1,1)\times\R)$ we have
		\begin{equation*}
			\langle\mu_{,x},\varphi\rangle:=-\int_{(-1,1)\times\R}\varphi_{,x}\di\mu\,.
		\end{equation*}
		Moreover by a density argument $\mu_{,x}$ can be extended to functions $\varphi(x,k)=\phi(x) \mathbbm{1}_A(k)$ with $\phi\in C^\infty_c(-1,1)$ and $A\subset\R$ bounded and measurable as 
		\begin{equation*}
			\langle\mu_{,x},\phi(x) \mathbbm{1}_A(k)\rangle:=-\int_{(-1,1)\times A}\dot\phi(x)\di\mu\,;
		\end{equation*}
	\item\label{rem-ii} Let $\mu\in\mathcal{M}_b^+((-1,1)\times\R)$ be of the form
	\begin{equation*}
		\mu=\sum_{k\in K}a(x,k)\mathcal{L}^1\res(-1,1)\times\delta_k\,,
	\end{equation*}
	with $K\subset\R$ countable and $a(\cdot,k)\in W^{1,1}(-1,1)$ for all $k\in K$. Then 
	\begin{equation*}
	\mu_{,x}= \sum_{k\in K}a_{,x}(x,k)\mathcal{L}^1\res(-1,1)\times\delta_k\,.
	\end{equation*}
	Moreover as
	$a_{,x}(\cdot,k)=0$ a.e. in $\{x\in(-1,1)\colon {a}(x,k)=0\}$ it follows $\mu_{,x}\in \mathcal{M}((-1,1)\times\R)$ and $\mu_{,x}\ll\mu$.
	\end{enumerate}
\end{remark}
\begin{definition}[Convergence] \label{def:convergence}
	For $L>0$ let $(w^L,u^L)\in\mathcal{A}_L^{\rm in}\times \mathcal{A}_L^{\rm out}$.
	We say a sequence $(w^L,u^L)$ converges as $L\to\infty$ to $\mu\in\M_{b}^+((-1,1)\times\R)$, if 
	$(\mu^L(u^L),\mu^L_{,x}(u^L))$ weakly-* converge  to $(\mu,\mu_{,x})$. 
\end{definition}
\noindent 
	We introduce the class of measures
	\begin{equation}\label{def:limit_measures}
		\begin{split}
			\mathcal{M}_\infty:=\bigg\{\mu\in &\M_b^+((-1,1)\times\R) \colon\mu( (-1,0]\times\R)=0,\ \mu_{,x}\in\M_b((-1,1)\times\R)\,, \\& \mu_{,x}\ll \mu\,,
			\int_{(0,1)\times\R}\phi(x)\di\mu(x,k)=	\int_{0}^1	2x\phi(x)\dx\quad \forall\phi\in C_c^{\infty}(0,1)
			\bigg\}\,,
		\end{split}
	\end{equation}
and the functional $\mathcal F_\infty\colon \M_\infty \to[0,+\infty]$
	\begin{equation}\label{def:F_infty}
	\mathcal F_\infty(\mu)= \int_{(0,1)\times\R} \biggl[
	k^2 +  
	\frac{1}{4k^2}\Bigl(\frac{\di \mu_{,x}}{\di\mu}\Bigr)^2 \biggr]\di\mu\,.
\end{equation}
Here $\frac{\di \mu_{,x}}{\di\mu}$ denotes the Radon-Nikodym derivative, existence of which follows from absolute continuity of $\mu_{,x}$ w.r.t. $\mu$. Moreover, if $\mu_L$ from  definition~\ref{def:muL} would be supported in $(0,1]\times \R$, then $\mathcal F_\infty(\mu_L)$ is simply equal to $\sum_{k \in \frac{\pi \mathbb Z}{L}} \int_0^1 \dot a_k^2(x) + a_k^2(x) k^4 \dx$, i.e. via Plancherel equality {(see equation~\eqref{plancherel-derivatives})} it equals $\mathcal S_L$ from~\eqref{eq:SL}. 
\begin{remark}\label{rem:limit}
	\begin{enumerate}[label=$(\roman*)$]
		\item\label{rem:limit(i)} When convenient we will identify the class $\mathcal M_\infty$ with the class of measures
		\begin{equation} 
			\begin{split}
				\bigg\{\mu\in \M_b^+((0,1)\times\R) &\colon \mu_{,x}\in\M_b((0,1)\times\R)\,, \ \mu_{,x}\ll \mu\,,
				\\& 	\int_{(0,1)\times\R}\phi(x)\di\mu(x,k)=	\int_{0}^1	2x\phi(x)\dx\quad \forall\phi\in C_c^{\infty}(0,1)
				\bigg\}\,;
			\end{split}
		\end{equation}
	\item\label{rem:limit(ii)} For later convenience we observe that $\mathcal{F}_\infty$ can be rewritten as follows
	\begin{equation}\label{def:F_infty2}
	\mathcal F_\infty(\mu)= \int_{(0,1)\times\R} 
k^2 \di\mu+   \int_{(0,1)\times\R} 
\frac{1}{4k^2}\Bigl(\frac{\di \mu}{\di|\tilde\mu|}\Bigr)^{-1}\Bigl(\frac{\di \mu_{,x}}{\di|\tilde\mu|}\Bigr)^2
 \di|\tilde \mu|\,,
	\end{equation}
where $\tilde \mu=(\mu,\mu_{,x})$ and $|\tilde\mu|$ denote its total variation. Indeed, since $\mu_{,x}\ll\mu\ll|\tilde\mu|$, we have 
\begin{equation*}
	\frac{\di\mu_{,x}}{\di|\tilde\mu|}= \frac{\di\mu_{,x}}{\di\mu}\frac{\di\mu}{\di|\tilde\mu|}\,,
\end{equation*}
from which we deduce 
\begin{equation*}
\begin{split}
\int_{(0,1)\times\R} 
\frac{1}{4k^2}\Bigl(\frac{\di \mu_{,x}}{\di\mu}\Bigr)^2 \di\mu&=  \int_{(0,1)\times\R} 
\frac{1}{4k^2}\Bigl(\frac{\di \mu_{,x}}{\di\mu}\Bigr)^2 \frac{\di\mu}{\di|\tilde\mu|}\di|\tilde\mu|\\
& =\int_{(0,1)\times\R} 
\frac{1}{4k^2}\Bigl(\frac{\di \mu}{\di|\tilde\mu|}\Bigr)^{-2} \Bigl(\frac{\di \mu_{,x}}{\di|\tilde\mu|}\Bigr)^2\frac{\di\mu}{\di|\tilde\mu|}\di|\tilde\mu|\\
&=\int_{(0,1)\times\R} 
\frac{1}{4k^2}\Bigl(\frac{\di \mu}{\di|\tilde\mu|}\Bigr)^{-1} \Bigl(\frac{\di \mu_{,x}}{\di|\tilde\mu|}\Bigr)^2\di|\tilde\mu|\,.
\end{split}
\end{equation*}
	\end{enumerate}
\end{remark}
We are now ready to state our main result.
\begin{theorem}[$\Gamma$-convergence]\label{theo:main_theo}
	Let $\mathcal F_L$ and $\mathcal{F}_\infty$ be as in \eqref{def:F_L} and \eqref{def:F_infty} respectively.	Then the following holds:
	\begin{itemize}
			\item[$a)$] $($Compactness$)$. For $L>0$ let $(w^L,u^L)\in \mathcal{A}_L^{\rm in}\times \mathcal{A}_L^{\rm out}$ be such that $\sup_L\mathcal F_L( w^L,u^L)<+\infty$. Then there exists a subsequence (not relabeled) and $\mu\in\mathcal M_\infty$ such that $(w^L,u^L)$ converges as $L\to +\infty$ in the sense of Definition~\ref{def:convergence} to $\mu$;
		\item[$b)$] $($$\Gamma$-convergence$)$. As $L\to+\infty$ the functionals $\mathcal F_L$ $\Gamma$-converge, with respect to the convergence in Definition~\ref{def:convergence}, to the functional $\mathcal F_\infty$.
	\end{itemize}
\end{theorem}
%
\section{Preliminaries}\label{sec:preliminaries}
 Let $u\in \mathcal{A}_L^{\rm out}$ and let  $a_k(x)$ be defined as in \eqref{def:coef_a}. Then we have
	\begin{equation}\label{eq:fourier_repr}
		\begin{aligned}
	u(x,y)&=a_0(x)+ \sum_ {k\in \frac{\pi \mathbb{Z}}{L}, k>0 } a_k(x)\sqrt2\sin(ky) + \sum_ {k\in \frac{\pi \mathbb{Z}}{L}, k < 0 } a_k(x)\sqrt2\cos(ky) \\
&=a_0(x)+ \sum_ {k\in \frac{\pi \mathbb{Z}}{L}, k>0 } {\rm sign}(a_k(x))\frac{\sqrt{a(x,k)}}{k} \sqrt{2}\sin(ky) + \sum_ {k\in \frac{\pi \mathbb{Z}}{L}, k < 0 } {\rm sign}(a_k(x))\frac{\sqrt{a(x,k)}}{-k} \sqrt{2}\cos(ky)
	\,.
		\end{aligned}
	\end{equation}  
Then  Plancherel equality yields 
\begin{equation}\label{eq:plancherel-u}
	\fint_{-L}^Lu^2\dy= a^2_0(x)+\sum_ {k\in \frac{\pi \mathbb{Z}}{L},k\ne0}a^2_k(x)=a^2_0(x)+
	\sum_ {k\in \frac{\pi \mathbb{Z}}{L},k\ne0}\frac{{a}(x,k)}{k^2}\,.
\end{equation}
The same holds for partial derivatives of $u$, that is
\begin{equation}\label{plancherel-derivatives}
	\begin{split}
	\fint_{-L}^L(D^\alpha u)^2\dy&=
	(D^\alpha a_0(x))^2+
	\sum_ {k\in \frac{\pi \mathbb{Z}}{L},k\ne0}\Big(\frac{\di^{\alpha_1}}{\di x^{\alpha_1}} a_k(x)k^{\alpha_2}\Big)^2 \\
&	= (D^\alpha a_0(x))^2+
	\sum_ {k\in \frac{\pi \mathbb{Z}}{L},k\ne0}\Big(\frac{\partial^{\alpha_1}}{\partial x^{\alpha_1}} \big(\sqrt{a(x,k)}\big)k^{\alpha_2-1}\Big)^2 
	\,, 
	\end{split}
\end{equation}
with $\alpha=(\alpha_1,\alpha_2)$ multi-index with $|\alpha|\le2$.  In case $u$ has higher regularity, i.e., $u\in W^{k,2}((-1,1)\times\R)$  with $k>2$, then the same applies for the higher derivatives, i.e., for $|\alpha|\le k$.
For later convenience we also note that 
\begin{equation}\label{derivatives-a}
\frac{\partial}{\partial x} \big(\sqrt{a(x,k)}\big)= \frac{a_{,x}(x,k)}{2\sqrt{a(x,k)}}\,, \quad\frac{\partial^2}{\partial x^2} \big(\sqrt{a(x,k)}\big)=\frac{a_{,xx}(x,k)}{2\sqrt{a(x,k)}}- \frac{(a_{,x}(x,k))^2}{4\sqrt{a^3(x,k)}}\,.
\end{equation}
We now recall the definition of disintegration of measures only in a specific case that we will be used throughout the paper, and we refer to \cite{AmFuPa:00} for a complete treatment of the subject.
\begin{definition}[Disintegration of measures in the $x$-variable]\label{def:disintegration}
	Let $I\subset\R$ be an interval  and let $\mu\in\mathcal{M}_b(I\times\R)$. We say that the family 
	$$(\nu_x,g(x))_{x\in I} \subset \mathcal{M}_b(\R)\times \R$$ is a disintegration of $\mu$ $($in the $x$-variable$)$ if $x\mapsto\nu_x$ is Lebesgue measurable, $|\nu_x|(\R)=1$ for every $x\in I$,  $g\in L^1(I)$, and 
		\begin{equation}\label{eq:disint}
		\int_{I\times\R}f(x,k)\di\mu=\int_I\int_{\R}f(x,k)\di\nu_x(k)g(x)\dx\,,
	\end{equation}
for every $f\in L^1(I\times\R;|\mu|)$. 
\end{definition}
Formally it simply means $\di\mu(x,k) = \di\nu_x (k) g(x) \dx$.
%
%
%
\begin{lemma}\label{lem:disint=pushfwd}
	Let $I\subset\R$ be an interval  and let $\mu\in\mathcal{M}^+_b(I\times\R)$. Then 
	\begin{equation}\label{eq:push-forward}
		\int_{I\times\R}\phi(x)\di\mu=	\int_Ig(x)\phi(x)\dx \quad \forall\phi\in C_c^{\infty}(I)\,,
	\end{equation}
	for some non-negative $g\in L^1(I)$,	if and only if there exists $x\mapsto\nu_x\in\mathcal{M}^+_b(\R)$ Lebesgue measurable such that $(\nu_x,g(x))_{x\in I} $ is a disintegration of $\mu$.
\end{lemma}
\begin{proof}
	\textit Let  $(\nu_x,g(x))_{x\in I} \subset \mathcal{M}^+_b(\R)\times \R^+$ be a disintegration of $\mu$. Then \eqref{eq:disint} holds with $f(x,k)=\phi(x)\in C_c^{\infty}(I)$ and since $|\nu_x|(\R)=\nu_x(\R)=1$ we readily deduce \eqref{eq:push-forward}.\\
Assume instead that \eqref{eq:push-forward} holds true. 
	Let $\pi_1\colon I\times \R \to I$ be the canonical projection and let $(\pi_1)_\sharp\mu\in\mathcal{M}^+_b(I)$ be the push-forward of $\mu$ with respect to $\pi_1$.  By the Disintegration Theorem (cf. \cite[Theorem 2.28]{AmFuPa:00}) there exists $x\mapsto\nu_x\in\mathcal{M}^+_b(\R)$ measurable with $\nu_x(\R)=1$ such that
	\begin{equation*}
		\int_{I\times\R}f(x,k)\di\mu(x,k)=\int_I\int_{\R}f(x,k)\di\nu_x(k)\di(\pi_1)_\sharp\mu(x)\,,
	\end{equation*}
	for all $f\in L^1(I\times \R;\mu)$. On the other hand \eqref{eq:push-forward} implies that $(\pi_1)_\sharp\mu(x)=g(x)\L^1\res I$ and therefore $(\nu_x,g(x))_{x\in I}$ is a disintegration of $\mu$.		\\
\end{proof}
\begin{cor}[Disintegration of $\mu\in\mathcal{M}_\infty$ in the $x$-variable]\label{cor:disint=pushfwd}
		Let $\mu\in\mathcal{M}_\infty$. Then there exists $x\mapsto\nu_x\in\mathcal{M}^+_b(\R)$ measurable such that $(\nu_x,2x)_{x\in (0,1)} $ is a disintegration of $\mu$.
\end{cor}
\begin{proof}
The proof follows by Lemma \ref{lem:disint=pushfwd} and from the fact that 
\begin{equation}\label{constraint1}
\int_{(0,1)\times\R}\phi(x)\di\mu=	\int_{0}^1	2x\phi(x)\dx\quad \forall\phi\in C_c^{\infty}(0,1)\,.
\end{equation}
\end{proof}
\section{Compactness and lower bound}\label{sec:compactness_lwb}
In this section we prove compactness and the $\Gamma-\liminf$ inequality. 
\begin{prop}[Compactness]\label{prop:compactness}
Let for $L>0$ be $(w^L,u^L)\in \mathcal{A}_L^{\rm in}\times \mathcal{A}_L^{\rm out}$ such that $\sup_L\mathcal F_L( w^L,u^L)<+\infty$. Then there exist a, not relabeled, subsequence  and $\mu\in\mathcal M_\infty$ such that $(w^L,u^L)$ converges to $\mu$, as $L\to+\infty$, in the sense of Definition~\ref{def:convergence}.
\end{prop}
\begin{proof}
Let $(w^L,u^L)$ be as in the statement. Let 
$\mu^L:=\mu^L(u^L)$
and $\mu^L_{,x}:=\mu^L_{,x}(u^L)$ be defined accordingly to Definition \ref{def:muL}, i.e., there exist ${a}^L(x,k)$ such that $a(\cdot,k)\in W^{1,1}(-1,1)$ and
\begin{equation*}
	\mu^L= \sum_{
		k\in \frac{\pi \mathbb{Z}}{L}}  a^L(x,k) \L^1\res{(-1,1)}\times \delta_k\,,
	\quad
	\mu^L_{,x}= \sum_{
		k\in \frac{\pi \mathbb{Z}}{L}}  a_{,x}^L(x,k)  \L^1\res{(-1,1)}\times\delta_k\,.
\end{equation*} 
\textit{Step 1:} we show that there exists $\mu\in \mathcal{M}_b^+((-1,1)\times\R)$ with $\mu_{,x}\in \mathcal{M}_b((-1,1)\times\R)$ and such that $(\mu^L,\mu^L_{,x})\stackrel{*}{\rightharpoonup}(\mu,\mu_{,x})$.
To this aim we observe that by taking $0<C_0:=\sup_L\mathcal{F}_L(w^L,u^L)<+\infty$ we have
\begin{equation*}
\mathcal F_L(w^L,u^L)
\le C_0\,,
\end{equation*}
so that in particular
\begin{equation}\label{eq:estimate}
	\begin{split}
	C_0\ge	\mathcal F_L(w^L,u^L)\ge 
		{L^2}\strokedint_{-L}^L\int_{-1}^1 &\bigg ( w^L_{2,y}+\frac{(u_{,y}^L)^2}{2} -x \bigg)^2 \dx\dy
	-\frac{L^2}3\\
	& +\strokedint_{-L}^L\int_{-1}^1 (u^L_{,x})^2+(u^L_{,yy})^2\dx\dy\,.
	\end{split}
\end{equation}
By Fubini's theorem, Jensen's inequality  and the fact that $w^L(x,\cdot)$ is $2L$-periodic we get
\begin{equation}\label{eq:estimate1}
	\begin{split}
	{L^2}\strokedint_{-L}^L\int_{-1}^1& \bigg( w^L_{2,y}+\frac{(u_{,y}^L)^2}{2} -x \bigg)^2 \dx\dy\ge L^2\int_{-1}^1\bigg(\fint_{-L}^L
\Big(	w^L_{2,y}+\frac{(u_{,y}^L)^2}{2} -x\Big)\dy
\bigg)^2\dx\\
& =L^2
\int_{0}^1\bigg(\fint_{-L}^L
\frac{(u_{,y}^L)^2}{2} \dy-x
\bigg)^2\dx+ L^2\int_{-1}^0\bigg(\fint_{-L}^L
\frac{(u_{,y}^L)^2}{2} \dy-x
\bigg)^2\dx\\
&
\ge L^2 \int_{0}^1\bigg(\fint_{-L}^L
\frac{(u_{,y}^L)^2}{2} \dy-x
\bigg)^2\dx + L^2
\int_{-1}^0 \bigg(\fint_{-L}^L
\frac{(u_{,y}^L)^2}{2} \dy
\bigg)^2\dx +L^2\int_{-1}^0x^2\dx\,,
	\end{split}
\end{equation}
where  the last inequality follows by using that $(a+b)^2\ge a^2+b^2$ provided that $ab>0$ with $a=\frac12\fint_{-L}^L(u_{,x}^2)\dy$ and $b=-x$ for $x\in (-1,0)$. 
Combining \eqref{eq:estimate} with \eqref{eq:estimate1} and using that $\int_{-1}^0x^2\dx=\frac13$ we find
\begin{equation}\label{eq:est0}
	\begin{split}
	C_0\ge \mathcal{F}_L(w^L,u^L)\ge & L^2 \int_{0}^1\bigg(\fint_{-L}^L
	\frac{(u_{,y}^L)^2}{2} \dy-x
	\bigg)^2\dx \\
&	+ L^2
	\int_{-1}^0 \bigg(\fint_{-L}^L
	\frac{(u_{,y}^L)^2}{2} \dy
	\bigg)^2\dx +\strokedint_{-L}^L\int_{-1}^1 (u^L_{,x})^2+(u^L_{,yy})^2\dx\dy\,.
	\end{split}
\end{equation}
Thus from \eqref{plancherel-derivatives} it follows
\begin{equation}\label{eq:est1}
	\begin{split}
	\frac C{L^2}\ge 	\int_0^1\bigg(\strokedint_{-L}^L\frac{(u^L_{,y})^2}2  \dy-x\bigg)^2 \dx&= \int_0^1\bigg(\frac12\sum_{k\in \frac{\pi \mathbb{Z}}{L}}a^L(x,k)-x\bigg)^2\dx\\
	&\ge C \Bigg( \int_0^1\bigg(\frac12\sum_{k\in \frac{\pi \mathbb{Z}}{L}}a^L(x,k)\bigg)^2\dx- \frac13\Bigg)\,,
	\end{split}
\end{equation}
and
\begin{equation}\label{eq:est2}
\frac C{L^2}\ge \int_{-1}^0\bigg( \strokedint_{-L}^L\frac{(u^L_{,y})^2}2 \dy\bigg)^2 \dx=\int_{-1}^0\bigg(\frac12\sum_{k\in \frac{\pi \mathbb{Z}}{L}} a^L(x,k)\bigg)^2\dx\,.
\end{equation}
Hence we obtain
\begin{equation*}
|\mu^L|((-1,1)\times\R)=\mu^L((-1,1)\times\R)
=
\int_{-1}^1\sum_{k\in \frac{\pi \mathbb{Z}}{L}} a^L(k,x)\dx\le C\,,
\end{equation*}
from which we deduce the existence of a (not relabeled) subsequence  and $\mu\in\mathcal{M}^+_b((-1,1)\times\R)$ such that 
$\mu^L\stackrel{	*}{\rightharpoonup}\mu$.
In addition \eqref{eq:est0} together with \eqref{plancherel-derivatives} and \eqref{derivatives-a} yield
\begin{equation}\label{eq:compactness}
	\begin{split}
		C&\ge  \int_{-1}^1 \strokedint_{-L}^L  (u^L_{,x})^2+(u^L_{,yy})^2\dy\dx\\
		& \ge \int_{-1}^1 \Bigg( \sum_ {k\in \frac{\pi \mathbb{Z}}{L}} a^L(x,k)k^2+\sum_ {k\in \frac{\pi \mathbb{Z}}{L},k\ne0}\frac{1}{4k^2}\frac{({a}^L_{,x}(x,k))^2}{a^L(x,k)}\Bigg)\dx\\
		&\ge  \int_{-1}^1   \sum_ {k\in \frac{\pi \mathbb{Z}}{L}}|{a}^L_{,x}(x,k) |
		\dx=|\mu^L_{,x}|((-1,1)\times\R)\,,
		 	\end{split}
\end{equation}
where the last inequality follows by Young's inequality. 
 Hence, up to subsequence, we may deduce that there exists $\tilde\mu\in\mathcal{M}_b(\times(-1,1)\times\R)$ such that $\mu^L_{,x}\stackrel{*}{\rightharpoonup}\tilde\mu$. Moreover given any $\varphi\in C^\infty_c((-1,1)\times\R)$, it holds
\begin{equation*}
	\int_{(-1,1)\times\R}\varphi \di\tilde\mu=\lim_{L\to+\infty}
\int_{(-1,1)\times\R}\varphi \di\mu^L_{,x}= - \lim_{L\to+\infty} \int_{(-1,1)\times\R} \varphi_{,x} \di\mu^L=-\int_{(-1,1)\times\R}\varphi_{,x}\di\mu\,,
\end{equation*}
which in turn implies $\tilde\mu=\mu_{,x}$. \\

\noindent
\textit{Step 2:} we show that $\mu_{,x}\ll\mu$.  By Remark \ref{rem:int-byparts} \ref{rem-ii} we have that $\mu^L_{,x}\ll\mu^L$. Now let $N\in\mathbb N$ be fixed and let $\mu^L_N:=\mu^L\res(-1,1)\times(-N,N)$ and $\mu_N:=\mu\res(-1,1)\times(-N,N)$. Then the following properties hold:
\begin{equation*}
\mu^L_{N,x}:=\mu_{,x}^L\res(-1,1)\times(-N,N)\,, \quad \mu_{N,x}:=\mu_{,x}\res(-1,1)\times(-N,N)\,,
\end{equation*}
\begin{equation}\label{eq:abs}
\mu^L_{N,x}\ll\mu^L_{N}\,,\quad (\mu^L_N,\mu^L_{N,x} )  \stackrel{*}{\rightharpoonup}(\mu_N,\mu_{N,x})\,,
\end{equation}
and 
\begin{equation*}
\frac{\di\mu^L_{N,x}}{\di\mu^L_N}(x,k)=\frac{\di\mu^L_{,x}}{\di\mu^L}(x,k)\res(-1,1)\times(-N,N)
\,.
\end{equation*}
Moreover recalling the definition of $\mu^L$ and \eqref{eq:compactness} we have 
\begin{equation}\label{eq:bound}
	\begin{split}
	\int_{(-1,1)\times(-N,N)}\frac1{4N^2} \Big(\frac{\di\mu^L_{N,x}}{\di\mu^L_N}(x,k)\Big)^2\di\mu^L_N&\le 
\int_{(-1,1)\times(-N,N)}\frac1{4k^2} \Big(\frac{\di\mu^L_{,x}}{\di\mu^L}(x,k)\Big)^2\di\mu^L\\
&\le \int_{-1}^{1} \sum_ {k\in \frac{\pi \mathbb{Z}}{L},k\ne0} \frac{1}{4k^2}\frac{({a}^L_{,x}(x,k))^2}{ a^L(x,k)}\dx\le C\,.
	\end{split}
\end{equation}
From \eqref{eq:abs}, \eqref{eq:bound} and \cite[Example 2.36 pg. 67, and discussion at pg. 66]{AmFuPa:00} we deduce that $\mu_{N,x}\ll\mu_N$ for every $N\in\mathbb N$ and hence $\mu_{,x}\ll\mu$.\\

\noindent
\textit{Step 3:} we show that $\mu\in\mathcal{M}^+_b((0,1)\times\R)$, that is, $\mu ((-1,0]\times\R)=0$, and that 
\begin{equation}\label{eq:constr}
	\int_{(0,1)\times\R}\phi(x)\di\mu=\int_0^12x\phi(x)\dx,
\end{equation}
for all $\phi\in C_c^{\infty}((0,1))$.
To this purpose for fixed $\delta\in(0,1)$ by \eqref{eq:estimate} we have 
\begin{equation*}
	\begin{split}
		 \mu^L((-1,\delta)\times\R)&=	\int_{-1}^\delta\sum_{k\in \frac{\pi \mathbb{Z}}{L}} a^L(x,k)\dx\\
&	\le C\int_{-1}^0\bigg(\sum_{k\in \frac{\pi \mathbb{Z}}{L}} a^L(x,k)\bigg)^2\dx + C\int_{0}^\delta\bigg(\sum_{k\in \frac{\pi \mathbb{Z}}{L}}a^L(x,k)-x\bigg)^2\dx + C\int_0^\delta x^2\dx\\
&\le \frac{C}{L^2}+C\delta^3.
	\end{split}
\end{equation*}
This together with the lower semicontinuity with respect to the weak* convergence  give
\begin{equation*}
	 \mu((-1,0]\times\R)\le  \mu((-1,\delta)\times\R)\le 
	\liminf_{L\to \infty}\mu^L((-1,\delta)\times\R)\le C\delta^3.
\end{equation*}
By sending $\delta\to0$ we deduce
$
 \mu((-1,0]\times\R)=0.
$
It remains to show \eqref{eq:constr}. Given $\phi\in C_c^\infty(0,1)$ it holds
\begin{equation*}
\int_{(0,1)\times\R}\phi(x)\di\mu^L= \int_0^1\phi(x)\Big(\sum_{k\in \frac{\pi \mathbb{Z}}{L}} a^L(x,k)-2x\Big)\dx+ \int_0^12x\phi(x)\dx.
\end{equation*}
From \eqref{eq:est1} it follows that
\begin{equation*}
	\int_0^1\Big|\phi(x)\sum_{k\in \frac{\pi \mathbb{Z}}{L}} a^L(x,k)-2x\Big|\dx
	\le C\|\phi\|_\infty\bigg( \int_0^1\Big(\frac12\sum_{k\in \frac{\pi \mathbb{Z}}{L}} a^L(x,k)-x\Big)^2\dx\bigg)^{1/2}\le\frac{C}{L}\to0, 
\end{equation*}
as $L\to+\infty$, so that 
\begin{equation}\label{eq:constr1}
\lim_{L\to +\infty}\int_{(0,1)\times\R}\phi(x)\di\mu^L=  \int_0^12x\phi(x)\dx.
\end{equation}
Next we fix $R\ge1$ and take $\psi_R\in C^\infty_c(\R)$ such that  $0\le \psi_R\le1$, $\psi_R(k)\equiv 1$ if $|k|<R$ and $\psi_R(k)\equiv 0$ if $|k|>R+1$. We have
\begin{equation}\label{eq:cutoff}
	\int_{(0,1)\times\R}\phi(x)\di\mu^L=
	\int_{(0,1)\times\R}\phi(x)\psi_R(k)\di\mu^L+ 
		\int_{(0,1)\times\R}\phi(x)(1-\psi_R(k))\di\mu^L\,.
\end{equation}
The weak* convergence yields
\begin{equation*}
\lim_{L\to +\infty}	\int_{(0,1)\times\R}\phi(x)\psi_R(k)\di\mu^L=	\int_{(0,1)\times\R}\phi(x)\psi_R(k)\di\mu\,,
\end{equation*}
whereas for the second term on the right hand-side of \eqref{eq:cutoff} we get 
\begin{equation*}
	\begin{split}
		\int_{(0,1)\times\R}|\phi(x)(1-\psi_R(k))|\di\mu^L
		&\le\int_{0}^{1}|\phi(x)|\Big(\sum_{k\in \frac{\pi \mathbb{Z}}{L},|k|>R} a^L(x,k)
	\Big)\dx\\
	&\le\frac{\|\phi\|_\infty}{R^2}
	\int_{0}^{1}\Big(\sum_{k\in \frac{\pi \mathbb{Z}}{L},|k|>R} a^L(x,k)k^2
	\Big)\dx\\
	&\le \frac{\|\phi\|_\infty}{R^2}
	\int_{0}^{1}\fint_{-L}^L(u_{,yy})^2\dy\dx\le \frac{C}{R^2}\,,
	\end{split}
\end{equation*}
 where the last to inequalities follow from \eqref{plancherel-derivatives} and \eqref{eq:estimate}.
Thus passing to the limit as $L\to+\infty$ in \eqref{eq:cutoff} we obtain
\begin{equation*}
	\int_{(0,1)\times\R}\phi(x)\psi_R(k)\di\mu-\frac{C}{R^2}\le
\lim_{L\to+\infty}	\int_{(0,1)\times\R}\phi(x)\di\mu^L\le \int_{(0,1)\times\R}\phi(x)\psi_R(k)\di\mu+ \frac{C}{R^2}\,.
\end{equation*}
Eventually by letting $R\to+\infty$ we deduce 
\begin{equation*}
\lim_{L\to+\infty}	\int_{(0,1)\times\R}\phi(x)\di\mu^L=\int_{(0,1)\times\R}\phi(x)\di\mu\,,
\end{equation*}
which together with \eqref{eq:constr1} yield \eqref{eq:constr}.
\end{proof}
\begin{prop}[Lower bound]\label{prop:lower_bound}
	Let $\mathcal F_L$ and $\mathcal F_\infty$ be as in \eqref{def:F_L} and \eqref{def:F_infty} respectively. Let for $L>0$ be $(w^L,u^L)\subset \mathcal{A}_L^{\rm in}\times \mathcal{A}_L^{\rm out}$ a sequence converging to $\mu\in\mathcal{M}_\infty$ in the sense of Definition \ref{def:convergence}. Then there holds
	\begin{equation}\label{eq:lowe_bound}
		\liminf_{L\to \infty}\mathcal F_L(w^L,u^L)\ge \mathcal F_\infty(\mu).
	\end{equation}
\end{prop}
\begin{proof}
%
Let $(w^L,u^L)$ be as in the statement and let
$\mu^L:=\mu^L(u^L)$
and $\mu^L_{,x}:=\mu^L_{,x}(u^L)$ be defined accordingly to Definition \ref{def:muL}, that is,
\begin{equation*}
	\mu^L= \sum_{
		k\in \frac{\pi \mathbb{Z}}{L}} a^L(x,k)  \L^1\res{(-1,1)}\times \delta_k\,,
	\quad
	\mu^L_{,x}= \sum_{
		k\in \frac{\pi \mathbb{Z}}{L}} a_{,x}^L(x,k)  \L^1\res{(-1,1)}\times\delta_k\,.
\end{equation*} 
Recalling \eqref{eq:est0}, \eqref{plancherel-derivatives} and \eqref{derivatives-a} we have that
\begin{equation}\label{eq:lower_bound}
	\begin{split}
	\mathcal F_L(w^L,u^L)&\ge  \int_{-1}^1\fint_{-L}^L  (u^L_{,x})^2+(u^L_{,yy})^2\dy\dx\\
	&\ge\int_0^1 \Big(
	\sum_ {k\in \frac{\pi \mathbb{Z}}{L},k\ne0}\frac{1}{4k^2}\frac{({a}^L_{,x}(x,k))^2}{ a^L(x,k)}+\sum_ {k\in \frac{\pi \mathbb{Z}}{L}} a^L(x,k)k^2
	\Big)\dx
	\\
	&=\int_{(0,1)\times\R}\bigg( k^2+ \frac{1}{4k^2}\Big(\frac{\di\mu_{,x}^L}{\di\mu^L}(x,k)\Big)^2\bigg)\di\mu^L
	\\
	&=
	 \int_{(0,1)\times\R} k^2 \di\mu^L
	 + \int_{(0,1)\times\R} 
	\frac{1}{4k^2}\Bigl(\frac{\di \mu^L}{\di|\tilde\mu^L|}(x,k)\Bigr)^{-1}\Bigl(\frac{\di \mu^L_{,x}}{\di|\tilde\mu^L|}(x,k)\Bigr)^2 \di|\tilde\mu^L|\,,
	\end{split}
\end{equation}
where  $\tilde\mu^L:=(\mu^L,\mu^L_{,x})$, and the last equality follows from Remark \ref{rem:limit} \ref{rem:limit(ii)}.
 By Reshetnyak Theorem (cf. \cite[Theorem 2.38]{AmFuPa:00}) there hold
\begin{equation}\label{eq:Resh1}
\liminf_{L\to+\infty}
\int_{(0,1)\times\R} k^2 \di\mu^L\ge \int_{(0,1)\times\R} k^2 \di\mu \,,
\end{equation}
and 
\begin{equation}\label{eq:Resh2}
	\begin{split}
\liminf_{L\to+\infty}	&
\int_{(0,1)\times\R} 	\frac{1}{4k^2}
\Bigl(\frac{\di \mu^L}{\di|\tilde\mu^L|}(x,k)\Bigr)^{-1}\Bigl(\frac{\di \mu^L_{,x}}{\di|\tilde\mu^L|}(x,k)\Bigr)^2 \di|\tilde\mu^L|
\\ 
&\ge\int_{(0,1)\times\R} \frac{1}{4k^2}
\Bigl(\frac{\di \mu}{\di|\tilde\mu|}(x,k)\Bigr)^{-1}\Bigl(\frac{\di \mu_{,x}}{\di|\tilde\mu|}(x,k)\Bigr)^2 \di|\tilde\mu|\,,
	\end{split}
\end{equation}
with $\tilde\mu:=(\mu,\mu_{,x})$.
Gathering together \eqref{eq:lower_bound}, \eqref{eq:Resh1} and \eqref{eq:Resh2} we find
\begin{equation*}
	\begin{split}
			\liminf_{L\to \infty}\mathcal F_L(w^L,u^L)\ge &\int_{(0,1)\times\R} k^2 \di\mu +\int_{(0,1)\times\R} \frac{1}{4k^2}
			\Bigl(\frac{\di \mu}{\di|\tilde\mu|}(x,k)\Bigr)^{-1}\Bigl(\frac{\di \mu_{,x}}{\di|\tilde\mu|}(x,k)\Bigr)^2 \di|\tilde\mu|
			 =\mathcal{F}_\infty(\mu)\,.
	\end{split}
\end{equation*}
\end{proof}
\section{Upper bound}\label{sec:upb}
In this section we prove the $\Gamma-\limsup$ inequality.
\begin{prop}[Upper bound]\label{prop:upb}
	Let $\mu\in \mathcal{M}_\infty$. Then for $L>0$ there exists a sequence $(w^L,u^L)\in \mathcal{A}_L^{\rm in}\times \mathcal{A}_L^{\rm out}$ that converges
 to $\mu\in\mathcal{M}_\infty$ in the sense of Definition \ref{def:convergence} and such that
 \begin{equation*}
 	\limsup_{L\to\infty}\F_L(w^L,u^L)\le \F_\infty(\mu)\,,
 \end{equation*}
with $\F_L$ and $\F_\infty$ defined as in \eqref{def:F_L} and \eqref{def:F_infty} respectively.
\end{prop}
\noindent
We 
divide the proof of Proposition \ref{prop:upb} into a number of steps.
For any $\lambda\ge1$ we define the following class of measures
\begin{equation}\label{def:limit_measures_lambda}
	\begin{split}
		\mathcal{M}_\infty^\lambda:=\bigg\{\mu\in &\M_b^+((0,\lambda)\times\R) \colon\mu_{,x}\in\M_b((0,\lambda)\times\R)\,,\quad \mu_{,x}\ll\mu\,,\\ 
		&	\int_{0}^\lambda	2x\phi(x)dx=\int_{(0,\lambda)\times\R}\phi(x)\di\mu(k,x)\quad \forall\phi\in C_c^{\infty}(0,\lambda)
		\bigg\}\,,
	\end{split}
\end{equation}
and the functional $\mathcal F_\infty^\lambda\colon \M_\infty^\lambda \to[0,+\infty]$
\begin{equation}\label{def:F_infty-lambda}
	\mathcal F_\infty^\lambda(\mu)= \int_{(0,\lambda)\times\R} \biggl[
	k^2 +  
	\frac{1}{4k^2} \Bigl(\frac{\di \mu_{,x}}{\di\mu}\Bigr)^2 \biggr]\di\mu\,.
\end{equation} Then we have $\mathcal{M}_\infty=\mathcal M_\infty^1$ and $\mathcal F_\infty=\mathcal F_\infty^1$.
\begin{lemma}[Dilation of $\mu$]\label{lem:extension}
	Let $\mu\in \mathcal{M}_\infty$.  Then for each $\lambda\in(1,2)$ there exists $\mu_\lambda\in	\mathcal{M}_\infty^\lambda$ such that  
		\begin{equation}\label{eq:lim_lambda}
		(\mu_\lambda,\mu_{\lambda,x})\res((0,1)\times\R) \stackrel{*}{\rightharpoonup}(\mu,\mu_{,x})\quad\text{as }\lambda\searrow1\,,
	\end{equation}
	and 
	\begin{equation}\label{eq:est_F_lambda}
		\int_{(0,\lambda)\times\R}k^2\di\mu_\lambda= \lambda^2 	\int_{(0,1)\times\R}k^2\di\mu\,,\quad 
			\int_{(0,\lambda)\times\R}\frac{1}{4k^2}\Bigl(\frac{\di \mu_{\lambda,x}}{\di\mu_\lambda}\Bigr)^2 \di\mu_\lambda= \int_{(0,1)\times\R}\Bigl(\frac{\di \mu_{,x}}{\di\mu}\Bigr)^2 \di\mu
	\end{equation}
	so that, in particular
	\begin{equation}\label{eq:lim_F_lambda}
	\lim_{\lambda\searrow1}	\F_\infty^\lambda\big(\mu_\lambda\big)=  \F_\infty(\mu)\,.
	\end{equation}
Moreover $(\nu_{\frac{x}{\lambda}},2x)_{x\in(0,\lambda)}$ is a disintegration of $\mu_\lambda$ where $\nu_x\in\mathcal{M}_b^+(\R)$ for $x\in(0,1)$ is the measure given by Corollary \ref{cor:disint=pushfwd}.
\end{lemma}
\begin{proof}
	Let $\mu_\lambda\in\mathcal{M}_b^+((0,\lambda)\times\R)$ be defined via duality as follows
\begin{equation}\label{eq:ext_mu}
	\int_{(0,\lambda)\times\R}\psi(x,k)\di\mu_\lambda=\lambda^2\int_{(0,1)\times\R}\psi(\lambda x,k)\di\mu\,,
\end{equation}
for every $\psi\in C((0,\lambda)\times\R)$.    Notice that $\mu_{\lambda,x}\in\mathcal{M}_b((0,\lambda)\times\R)$ and is given by
\begin{equation}\label{eq:ext_mudot}
	\int_{(0,\lambda)\times\R}\psi(x,k)\di \mu_{\lambda,x}=\lambda\int_{(0,1)\times\R}\psi(\lambda x,k)\di\mu_{,x}\,.
\end{equation}
Hence $\mu_{\lambda,x}\ll\mu_\lambda$ and 
\begin{equation*}
	\frac{\di\mu_{\lambda,x}}{\di\mu_\lambda}(x,k)=\frac1\lambda\frac{\di\mu_{,x}}{\di\mu}\Big(\frac x{\lambda},k\Big)\,.
\end{equation*}
Moreover \eqref{eq:ext_mu} together with the change of variable $ x=\lambda \hat x$  imply 
\begin{equation*}
		\int_{0}^\lambda	2x\phi(x)\dx=\int_{(0,\lambda)\times\R}\phi(x)\di\mu_\lambda\quad \forall\phi\in C_c^{\infty}(0,\lambda)\,.
\end{equation*}
It follows that $\mu_\lambda\in	\mathcal{M}_\infty^\lambda$.
Moreover from \eqref{eq:ext_mu} and \eqref{eq:ext_mudot} we deduce that 
\begin{equation*}
	(\mu_\lambda,\mu_{\lambda,x})\res(\R\times(0,1)) \stackrel{*}{\rightharpoonup}(\mu,\mu_{,x})\quad\text{as }\lambda\to1\,,
\end{equation*}
and \eqref{eq:est_F_lambda} which implies \eqref{eq:lim_F_lambda}.
The fact that
$(\nu_{\frac{x}{\lambda}},2x)_{x\in(0,\lambda)}$ is a disintegration of $\mu_\lambda$ follows again by a change of variable.
\end{proof}
\begin{lemma}[Discretisation of $\mu$]\label{lem:discretisation}
	Let $\mu\in \mathcal{M}_\infty$ with $\F_\infty(\mu)<+\infty$ and let $\lambda=\lambda(L)\searrow 1$ as $L\to\infty$. Then there exists   $(\mu^L)\subset \mathcal{M}_\infty^\lambda$ with the following properties: 
	\begin{enumerate}[label=$(\roman*)$]
		\item \label{(i)discr}$	\mu^L=\sum_{
			k\in \frac{\pi \mathbb{Z}}{L}} \overline b^L(x,k)  \L^1\res{(0,\lambda)}\times\delta_k$ 
		with 
		\begin{equation}\label{eq:constraint}
		\overline 	b^L(\cdot,k)\in W^{1,1}(0,\lambda)\quad \text{and}\quad
			\sum_{
				k\in \frac{\pi \mathbb{Z}}{L}}\overline  b^L(x,k)=2x\,,\quad \forall x\in (0,\lambda)\,,
		\end{equation}
		\begin{equation}\label{eq:energy}
			\F_\infty^\lambda(\mu^L)= \int_0^\lambda	\sum_{
				k\in \frac{\pi \mathbb{Z}}{L}} k^2\overline  b^L(x,k)\dx+ 
			\int_0^\lambda	\sum_{
				k\in \frac{\pi \mathbb{Z}}{L},k\ne0}\frac{1}{4k^2}\frac{(\overline b^L_{,x}(x,k))^2}{\overline  b^L(x,k)}\dx\,;
		\end{equation}
		\item\label{(ii)discr} $
		(\mu^L,\mu^L_{,x})\res((0,1)\times\R)\stackrel{*}{\rightharpoonup}(\mu,\mu_{,x})
		$;
		\item\label{(iii)discr} $	\limsup_{L\to\infty}\F_\infty^\lambda(\mu^L)\le	\F_\infty\big(\mu\big)$.
	\end{enumerate}
\end{lemma}
\begin{proof}
For each $L>1$ let $\mu_\lambda\in\mathcal{M}_\infty^\lambda$ be the measure given by Lemma \ref{lem:extension} and let $(\nu_{\frac{x}{\lambda}},2x)_{x\in (0,\lambda)}$ be the corresponding disintegration. We then define $\mu^L\in\mathcal{M}_b^+((0,\lambda)\times\R)$ as
	\begin{equation}
		\mu^L:= \sum_{
			k\in \frac{\pi \mathbb{Z}}{L}}\overline  b^L(x,k) \L^1\res{(0,\lambda)}\times\delta_k\,,
	\end{equation}
	where for $(x,k)\in(0,\lambda)\times \frac{\pi \mathbb{Z}}{L}$ we set
	\begin{equation}\label{def:coeff}
	\overline 	b^L(x,k):=\begin{cases}
			0&\text{if }k=0\,,\\[1em]
			2x\nu_{\frac{x}{\lambda}}(I_k^L)&\text{if } k\ne0\,,
		\end{cases}
		\quad\text{and}\quad
		I^L_k:=\begin{cases}
			(k-\frac{\pi}{L},k]&\text{if }k>0\,,\\[1em]
			[k,k+\frac{\pi}{L})&\text{if }k<0\,.
		\end{cases}
	\end{equation}
	Now, for each $k\in\frac{\pi \mathbb{Z}}{L}$, $\overline  b^L(\cdot,k)\in W^{1,1}(0,\lambda)$ with 
	\begin{equation*}
	\overline 	b_{,x}^L(x,k)=\begin{cases}
			0&\text{if }k=0\,,\\[1em]
			\displaystyle2x\int_{I_k^L}\frac{\di\mu_{\lambda,x}}{\di\mu_\lambda}(x,\hat k)\di\nu_{\frac{x}{\lambda}}(\hat k)&\text{if } k\ne0\,.
		\end{cases}
	\end{equation*}
	Indeed if $k=0$ there is nothing to prove. If instead $k\ne0$, for $\phi\in C^\infty_c(0,\lambda)$ from the definition of $\nu_{\frac{x}{\lambda}}$ and recalling Remark \ref{rem:int-byparts} we get
	\begin{equation*}
		\begin{split}
			\int_0^\lambda &\overline  b^L(x,k)\dot\phi(x)\dx=\int_0^\lambda\int_{\R}\mathbbm{1}_{I_k^L}(\hat k)\dot\phi(x)
			\di\nu_\frac{x}{\lambda}(\hat k)\,2x\dx\\
			&=\int_{(0,\lambda)\times\R}\mathbbm{1}_{I_k^L}(\hat k)\dot\phi(x)\di\mu_\lambda
			 =- \int_{(0,\lambda)\times\R}\mathbbm{1}_{I_k^L}(\hat k)\phi(x)\di\mu_{\lambda,x}\\
			&=- \int_{(0,\lambda)\times\R}\mathbbm{1}_{I_k^L}(\hat k)\phi(x)\frac{\di\mu_{\lambda,x}}{\di\mu_\lambda}(x,\hat k)
			\di\mu_{\lambda}
		=-\int_0^\lambda\int_{I_k^L}\frac{\di\mu_{\lambda,x}}{\di\mu_\lambda}(x,\hat k)
			\di\nu_\frac{x}{\tilde\lambda}(\hat k)\,2x\phi(x)\dx\,;
		\end{split}
	\end{equation*}
furthermore by Young's inequality and \eqref{eq:lim_F_lambda}
\begin{equation*}
	\begin{split}
			\int_0^\lambda|\overline b_{,x}^L(x,k)|\dx\le
		\int_{ I_k^L\times(0,\lambda)}\left|\frac{\di\mu_{\lambda,x}}{\di\mu_\lambda}\right|\di\mu_\lambda
		\le \frac12
		\int_{I_k^L\times(0,\lambda)} \left(k^2 +
	\frac1{k^2}\left(\frac{\di\mu_{\lambda,x}}{\di\mu_\lambda}\right)^2\right)\di\mu_\lambda\le C\,.
	\end{split}
\end{equation*}
Thus in particular 
\begin{equation*}
	\int_0^\lambda\sum_{k\in\frac{\pi\Z}{L}}|\overline b_{,x}^L(x,k)|\dx\le\mathcal{F}_\infty^\lambda(\mu_\lambda)\le C\,.
\end{equation*}
	As a consequence we have that $\mu_{,x}^L\in\mathcal{M}_b((0,\lambda)\times\R)$, and
	\begin{equation*}
		\mu_{,x}^L= \sum_{
			k\in \frac{\pi \mathbb{Z}}{L}}\overline  b_{,x}^L(x,k)  \L^1\res{(0,\lambda)}\times\delta_k\,,
	\end{equation*}
	and by Remark \ref{rem:int-byparts} \ref{rem-ii} $\tilde\mu_{,x}^L\ll\tilde\mu^L$.
	Moreover, as $\nu_{\frac{x}{\lambda}}$ is a probability measure, there holds
	\begin{equation*}
		\sum_{
			k\in \frac{\pi \mathbb{Z}}{L}}\overline  b^L(x,k)=2x \bigg(\sum_{
			k\in \frac{\pi \mathbb{Z}}{L},k\ne0} \nu_{\frac{x}{\lambda}}(I_k^L)\bigg)=2x\,,\quad \forall x\in (0,\lambda)\,.
	\end{equation*}
	Note in particular that $\mu^L\in\mathcal{M}_\infty^\lambda$. Moreover \eqref{eq:energy} readily follows and \ref{(i)discr} is proved. We next show \ref{(ii)discr}.
	Take $\varphi\in C^\infty_c((0,1)\times\R)$, so that from \eqref{def:coeff} we obtain
	\begin{equation}\label{eq:show1}
		\begin{split}
			\int_{(0,1)\times\R}\varphi\di\mu^L
			&=\int_0^1
			\sum_{
				k\in \frac{\pi \mathbb{Z}}{L},k\ne0}\varphi(x,k) \nu_{\frac{x}{\lambda}}(I_k^L)\,2x\dx\\
			& = \int_0^1
			\sum_{
				k\in \frac{\pi \mathbb{Z}}{L},k\ne0} \int_{I_k^L}
			\big(	\varphi(x,k)-\varphi(x,\hat k)\big)\di\nu_{\frac{x}{\lambda}}(\hat k)
			\,	2x\dx+\int_{(0,1)\times\R}\varphi\di\mu_\lambda\,.
		\end{split}
	\end{equation}
	Since $\varphi$ is uniformly continuous for every $\varepsilon>0$ there is $L_0>1$ such that for all $L\ge L_0$ 
	\begin{equation*}
		|	\varphi(x,k)-\varphi(x,\hat k)|<\varepsilon\quad\forall x\in(0,\lambda)\,,\  \forall k\in\frac{\pi\mathbb Z}{L},\ \forall\hat{k}\in I_k^L \,,
	\end{equation*}
	from which we readily deduce that 
	\begin{equation}\label{eq:show2}
		\int_0^1
		\sum_{
			k\in \frac{\pi \mathbb{Z}}{L},k\ne0} \int_{I_k^L}
		\big|	\varphi(x,k)-\varphi(x,\hat k)\big|\di\nu_{\frac{x}{\lambda}}(\hat k)
		\,2x\dx\le \mu_\lambda((0,\lambda)\times\R)\varepsilon=\lambda^2 \mu((0,1)\times\R)\varepsilon\,.
	\end{equation}
	From \eqref{eq:show1}, \eqref{eq:show2},\eqref{eq:lim_lambda} and the arbitrariness of $\varepsilon$ we infer $\mu^L\res((0,1)\times\R)\stackrel{*}{\rightharpoonup}\mu$ as $L\to+\infty$.  By analogous arguments we get $\mu_{,x}^L\res((0,1)\times\R)\stackrel{*}{\rightharpoonup}\mu_{,x}$ as $L\to+\infty$. 
	It remains to prove \ref{(iii)discr}. We start by observing that for all $\delta>0$ and  all $\hat k\in I_k^L$ we have
 $$k^2\le \left(\hat k+\frac\pi L\right)^2\le (1+\delta)\hat k^2+ (1+\delta^{-1})\frac{\pi^2}{L^2}\,,$$
	so that
	\begin{equation}\label{eq:lim_first_term}
		\begin{split}
			\int_0^\lambda	\sum_{
				k\in \frac{\pi \mathbb{Z}}{L}} k^2\overline  b^L(x,k)\dx&
			=\int_{0}^\lambda \sum_{
				k\in \frac{\pi \mathbb{Z}}{L},k\ne0} \int_{I_k^L} k^2 \di\nu_{\frac{x}{\lambda}}(\hat k)\, 2x\dx\\
			&= \sum_{
				k\in \frac{\pi \mathbb{Z}}{L},k\ne0} \int_{(0,\lambda)\times I_k^L} k^2\di\mu_\lambda(x,\hat k)\\
			&\le \sum_{
				k\in \frac{\pi \mathbb{Z}}{L},k\ne0} \int_{(0,\lambda)\times I_k^L}  \bigg((1+\delta)\hat k^2  
			+(1+\delta^{-1})\frac{\pi^2}{L^2}\bigg)
			\di\mu_\lambda(x,\hat k)\\
			&=(1+\delta)	\int_{(0,\lambda)\times\R}\hat k^2\di\mu_\lambda
			+\mu_\lambda((0,\lambda)\times\R)(1+\delta^{-1})\frac{\pi^2}{L^2}\,.
		\end{split}
	\end{equation}
	Moreover there holds
	\begin{equation}\label{eq:lim_second_term}
		\begin{split}
			\int_0^\lambda	\sum_{
				k\in \frac{\pi \mathbb{Z}}{L},k\ne0}\frac{1}{4k^2}\frac{(\overline  b^L_{,x}(x,k))^2}{\overline  b^L(x,k)}\dx
			=\int_0^\lambda\sum_{
				k\in \frac{\pi \mathbb{Z}}{L},k\ne0}
			\frac{1}{4 k^2}\bigg(\frac{\overline b_{,x}^L(x,k)}{\overline  b^L(x,k)}\bigg)^2\overline  b^L(x,k)
			\dx\,.
		\end{split}
	\end{equation}
	Since $1/|k|\le 1/|\hat k|$ for $\hat k\in I_k^L$ the following inequality follows
	\begin{equation}\label{claim}
		\frac{1}{2 |k|}
		\frac{\overline b_{,x}^L(x,k)}{\overline  b^L(x,k)}=\frac{1}{2 |k|} \strokedint_{I_k^L}   \frac{\di\mu_{\lambda,x}}{\di\mu_\lambda}(x,\hat k)
		\di\nu_{\frac{x}{\lambda}}(\hat k)
		\le \strokedint_{I_k^L}  \frac{1}{2 |\hat k|} \frac{\di\mu_{\lambda,x}}{\di\mu_\lambda}(x,\hat k)
		\di\nu_{\frac{x}{\lambda}}(\hat k)\,.
	\end{equation}
	Now combining \eqref{eq:lim_second_term} with \eqref{claim} we obtain 
	\begin{equation}\label{eq:lim_second_term_1}
		\begin{split}
			\int_0^\lambda	\sum_{
				k\in \frac{\pi \mathbb{Z}}{L},k\ne0}\frac{1}{4k^2}\frac{(\overline  b^L_{,x}(x,k))^2}{\overline b^L(x,k)}\dx
			&=\int_0^\lambda\sum_{
				k\in \frac{\pi \mathbb{Z}}{L},k\ne0}
			\bigg(	\strokedint_{I_k^L}  \frac{1}{2 |\hat k|} \frac{\di\mu_{\lambda,x}}{\di\mu_\lambda}(x,\hat k)
			\di\nu_{\frac{x}{\lambda}}(\hat k)\bigg)^2
			\nu_{\frac{x}{\lambda}}(I_k^L)\, 2x\dx\\
			&\le 
			\int_0^\lambda\sum_{
				k\in \frac{\pi \mathbb{Z}}{L},k\ne0}
			\int_{I_k^L}  \frac{1}{4 \hat k^2}		\bigg( \frac{\di\mu_{\lambda,x}}{\di\mu_\lambda}(x,\hat k)\bigg)^2
			\di\nu_{\frac{x}{\lambda}}(\hat k)
			\, 2x\dx\\
			&=\int_{(0,\lambda)\times\R}\frac{1}{4 \hat k^2}		\bigg( \frac{\di\mu_{\lambda,x}}{\di\mu_\lambda}(x,\hat k)\bigg)^2\di\mu_\lambda\\
			&
			=\int_{(0,1)\times\R}\frac{1}{4 \hat k^2}		\bigg( \frac{\di\mu_{,x}}{\di\mu}(x,\hat k)\bigg)^2\di\mu
			\,,
		\end{split}
	\end{equation}
	where the inequality follows by Jensen's inequality. Finally gathering together \eqref{eq:lim_first_term} and \eqref{eq:lim_second_term_1} and recalling \eqref{eq:lim_F_lambda} we infer
	\begin{equation*}
		\limsup_{L\to\infty}	\F_\infty^\lambda( \mu^L)\le (1+\delta)
	\limsup_{L\to\infty}	 \F_\infty^\lambda\big(\mu_\lambda\big)
		\le (1+\delta) \F_\infty(\mu)\,.
	\end{equation*}
By arbitrariness of $\delta$, \ref{(iii)discr} follows by letting $\delta\to0$.\\
\end{proof}
\begin{lemma}[Construction of  $u$]\label{lem:moll}
		Let $\mu\in\mathcal{M}_\infty$ be such that $\F_\infty(\mu)<+\infty$. Let  $\eps=\eps(L)>0$  and $n=n(L)\in \mathbb{N}$ be such that 
		$$\lim_{L\to +\infty}\eps(L)=0\,,\quad
		\lim_{L\to +\infty}n(L)=\lim_{L\to +\infty}\frac L{n(L)}=+\infty\,.$$
		Then there exists $\hat u^L\in \mathcal{A}_L^{\rm out}\cap \mathcal{A}_{L_0}^{\rm out}$  with $L_0:=L/n(L)$ that satisfies the following properties: let $$A^L(x):=\frac12\fint_{-L}^L(\hat{u}^L_{,y}(x,\cdot))^2\dy \quad \text{ and }
	\quad	f^L(x):=\sqrt{\frac{x}{A^L(x)}}\,.$$
	Then	 for all $x\in (0,1)$ and $N\in\mathbb{N}$ there holds
		\begin{equation}\label{eq:AL}
			\max\{x,\eps\}\lesssim A^L(x) \lesssim \max\{x,\eps\}\,;
		\end{equation}
		\begin{equation}\label{eq:f(x)}(f^L(x))^2\lesssim \frac x{\max\{x,\eps\}}\,,\quad
	(	f^L(x))^2\le 1+o_L(1)
	\,;
		\end{equation}
		\begin{equation}\label{eq:f(x)-2}
			(	f^L(x))^2\ge1+o_N(1)\ \text{ if }x\in (N\eps,1)\,;
		\end{equation}
		\begin{equation}\label{eq:f'(x)-f''(x)}
			(\dot	f^L(x))^2\lesssim
			\frac{ \max \{e^{-\frac x\eps}, e^{-\frac{1}{\sqrt\eps}}\}}{x\eps}\,, \quad (	\ddot f^L(x))^2\lesssim
			\frac1{x^3\eps}\,;
		\end{equation}
	{	there exists a continuous increasing function $\omega\colon[0,+\infty)\to[0,+\infty)$ with $\omega(0)=0$ such that }
		\begin{equation}\label{eq:est-u}
		\fint_{-L}^L(\hat u^L(x,\cdot))^2\dy \lesssim
	{
		\begin{cases}
			\max	\{x,\eps\}	(\omega(2N\eps)+Ne^{-N})& \text{ if }x\in (0,N\eps)\\
			x&\text{ if }x\in (N\eps,1)
		\end{cases}\quad
	}\,;
	\end{equation}
	\begin{equation}\label{eq:est-u^4}
		\fint_{-L}^L(\hat u^L(x,\cdot))^4\dy \lesssim
		L_0^2(\max\{\eps,x\})^2\,;
	\end{equation}
	\begin{equation}\label{eq:uyy+ux}
		\fint_{-L}^{L}
		\left( (\hat u_{,yy}^L(x,\cdot))^2 + (\hat u_{,x}^L(x,\cdot))^2\right)
		\dy\lesssim\frac1\eps\,.
	\end{equation}
	Moreover 
				\begin{equation}\label{eq:weak-conv}
				(\mu^L(\hat u^L),\mu^L_{,x}(\hat u^L))\stackrel{*}{\rightharpoonup}(\mu,\mu_{,x})\quad \text{in }\mathcal M_b((-1,1)\times\R)^2\,;
			\end{equation}
			\begin{equation}\label{eq:limsup}
				\limsup_{L\to\infty}
				\fint_{-L}^L\int_{-1}^1\big((\hat u^L_{,x})^2+ (\hat u^L_{,yy})^2\big) \dx\dy\le \mathcal{F}_\infty(\mu)\,;
			\end{equation}
		\begin{equation}\label{eq:uxx-uxy-uxyy}
		\fint_{-L}^L\int_{-1}^1\left((\hat u^L_{,xy})^2+(\hat u^L_{,xx})^2+ 	(\hat u_{,yyx}^L)^2\right)\dx\dy
		\lesssim \frac 1{\eps^2}\,;
		\end{equation}
			\begin{equation}\label{eq:ux}
\fint_{-L}^L\int_{-1}^1(\hat u^L_{,x})^4\dx\dy\lesssim \frac{L_0^2}{\eps^2}\,.
			\end{equation}
Finally since $\hat u^L$ and all its derivatives are $2L_0$-periodic in the $y$-variable the above estimates still hold true if we replace the average integral on $[-L,L]$ with the average integral on $[-L_0,L_0]$.
\end{lemma}
\begin{proof}
		Let $\mu\in\mathcal{M}_\infty$, $\eps=\eps(L)$, $n=n(L)$ and $L_0$ be as in the statement.
	We set
	\begin{equation*}
		\lambda=\lambda(L):=(1+\sqrt{\eps})
		\searrow1\quad\text{as }L\to+\infty\,,
	\end{equation*}
hence in particular $\lambda\searrow1$ as $L_0\to+\infty$.
	%
	We  construct  $\hat u^L\in \mathcal A_{L_0}^{\rm out}$ and then we extend it periodically in  $[-1,1]\times[-L,L]$, without relabelling it.  In this way, since $L=n(L)L_0$ with $n(L)\in \mathbb{N}$, we have $\hat u^L\in \mathcal A_{L}^{\rm out}$.
The main idea is that of discretizing the measure $\mu$ in the variable $k$ to get a measure concentrated on lines $\R\times\{k\}$ with $k\in \frac{\pi\mathbb{Z}}{L_0}$ where the weight on each line is a coefficient $b^{L_0}(x,k)$. Afterwards we  define $\hat u^L$ as  in such a way that its Fourier coefficients $a^L_k(x)$ are as close as possible to ${\sqrt{b^{L_0}(x,k)}}/k$  but at the same time have better regularity to ensure $\hat u^L\in \mathcal{A}_L^{\rm out}$.  In order to do that we first dilate $\mu$ with a factor $\lambda$ in the $x$-variable as in Lemma \ref{lem:extension}, then we discretize using Lemma \ref{lem:discretisation}, and finally we mollify $b^{L_0}(\cdot,k)$ at scale $\eps$ after a suitable extension in $\R$. \\

\noindent
To this purpose we let $(\mu^{L_0})\subset \mathcal{M}_\infty^\lambda$ be the sequence of Lemma \ref{lem:discretisation} for the parameter $L_0$, which is of the form
	$$	\mu^{L_0}=\sum_{
		k\in \frac{\pi \mathbb{Z}}{L_0}} \overline b^{L_0}(x,k)  \L^1\res{(0,\lambda)}\times\delta_k\,.$$ 
By the mean value theorem, for each $\lambda$, we can find $\bar\lambda\in(\frac{\lambda+1}2,\lambda)$ such that
	\begin{equation}\label{ub:1}
		\sum_{
			k\in \frac{\pi \mathbb{Z}}{L_0}} \overline b^{L_0}(\bar{\lambda},k)k^2\le \fint_{\frac{\lambda+1}{2}}^\lambda \sum_{
			k\in \frac{\pi \mathbb{Z}}{L_0}} \overline b^{L_0}(x,k)k^2\dx\,.
	\end{equation}
	By truncating $\bar b^{L_0}$ at $x=\bar\lambda$ we can define $ b^{L_0}\colon\R\times\frac{\pi\mathbb{Z}}{L_0}\to \R$ as 
	\begin{equation}\label{def:b}
		b^{L_0}(k,x):=\begin{cases}0 &\text{if }x\le0,\\
			\bar b^{L_0}(x,k)&\text{if }0<x<\bar\lambda,\\
			\bar b^{L_0}(\bar\lambda,k)&\text{if }x\ge\bar\lambda\,.
		\end{cases}
	\end{equation}
Let $\rho_\eps(x):=\frac{1}{2\eps}e^{\frac{-|x|}{\eps}}$ and note that in particular
	\begin{equation}\label{eq:prop-moll}
		|\dot\rho_\eps(x)|= \frac1\eps\rho_\eps(x)\,.
	\end{equation}
	Finally we let   $ a^{L}\colon\R\times\frac{\pi\mathbb{Z}}{L_0}\to \R$ be defined as
	\begin{equation*}
		a^{L}(x,k):= b^{L_0}(\cdot,k)*\rho_\eps(x)\,,
	\end{equation*}
	and   $\hat u^L\in\mathcal{A}_{L_0}^{\rm out}$  be the  function 
	\begin{equation*}
	\hat	u^{L}(x,y):= \sum_ {k\in \frac{\pi \mathbb{Z}}{L_0}, k>0 } \frac{\sqrt{ a^{L}(x,k)}}{k}\sqrt2\sin(ky) + \sum_ {k\in \frac{\pi \mathbb{Z}}{L_0}, k < 0 }  \frac{\sqrt{ a^{L}(x,k)}}{k}\sqrt2\cos(ky)\,.
	\end{equation*}
Eventually we extend $\hat u^L$, without relabelling it, periodically in $[-1,1]\times [-L,L]$. \\

\noindent
\textit{Step 1:} in this step we show \eqref{eq:AL}--\eqref{eq:f'(x)-f''(x)}. By \eqref{plancherel-derivatives} and  \eqref{eq:constraint} we have that 
\begin{equation}\label{est:A}
	\begin{split}
		2A^L(x)&=\fint_{-L}^L(\hat{u}_{,y})^2\dy= 
	\fint_{-L_0}^{L_0}(\hat{u}_{,y})^2\dy
\\&	= 
	\sum_{k\in \frac{\pi \mathbb{Z}}{L_0}}{a}^{L}(x,k)
= \Big(\sum_{k\in \frac{\pi \mathbb{Z}}{L_0}}{b}^{L_0}(\cdot,k)\Big)*\rho_\eps(x)\\
&= 2\big(x\chi_{(0,\bar\lambda)}+\bar\lambda\chi_{(\bar\lambda,+\infty)}\big)*\rho_\eps(x)=2x\chi_{\{x\ge0\}}+\eps\Big(e^{\frac{-|x|}{\eps}}  - e^{\frac{x-\bar\lambda}{\eps}} \Big )\,,
	\end{split}
\end{equation}
for all $x\in [-1,1]$. 
Then for $\eps$ small enough we have 
$$\frac\eps{2e}\le 2 A^L(x)\le {3\eps}\quad \text{ if }x\in (0,\eps)\,,$$
$$ x \le 2x+ \eps\big(e^{\frac{-1}{\eps}}  - e^{\frac{1-\bar\lambda}{\eps}} \big )
\le2 A^L(x)\le {3x}\quad\text{ if } x\in(\eps,1)\,,$$
so that 
\begin{equation*}
 \frac1{2e}\max\{x,\eps\}\le	2A^L(x)=\fint_{-L}^L(\hat{u}_{,y})^2\dy\le 3\max\{x,\eps\}\,.
\end{equation*}
We have that
\begin{equation*}
	(f^L(x))^2=\frac{x}{A^L(x)}\,,
\end{equation*}
and by estimates above
$$(f^L(x))^2\lesssim\frac{ x}{\max\{x,\eps\}}\lesssim\frac x\eps\quad \text{ if }x\in(0,\eps)\,.$$
Observing that $(e^{\frac{-x}{\eps}}-e^{\frac{x-\bar\lambda}{\eps}})\ge 0$ if and only if $x\in(0,\bar\lambda/2)$ we have 
$$(f^L(x))^2=  \frac{x}{x+\frac\eps2 (e^{\frac{-x}{\eps}}-e^{\frac{x-\bar\lambda}{\eps}}) } \le 1 \quad \text{ in }(0,\bar\lambda/2)\,,$$ 
and 
$$ 	(f^L(x))^2\le \frac{x}{x+\frac\eps2 (e^{\frac{-1}{\eps}}-e^{\frac{1-\bar\lambda}{\eps}})} \le 1+ \frac{\frac\eps2 |e^{\frac{-1}{\eps}}-e^{\frac{1-\bar\lambda}{\eps}}|
 }{\frac12+ \frac\eps2 (e^{\frac{-1}{\eps}}-e^{\frac{1-\bar\lambda}{\eps}})}=1+\frac{\eps}{8}\quad \text{ for }x\in(\bar\lambda/2,1)\,.
$$ 
Moreover we have that 
$$ (f^L(x))^2
 \ge  \frac{x}{x+\frac\eps2 e^{{-N}} }= 1-  \frac{\frac\eps2 e^{{-N}} }{x+\frac\eps2 e^{{-N}} } \ge 1- 
\frac{\frac\eps2 e^{{-N}} }{N\eps }= 1+o_N(1)\quad\text{ for }x\in[N\eps,1)\,.
$$ 
A direct computation shows that 
$$(\dot f^L(x))^2=\frac{\left(A^L(x)-x\dot A^L(x)\right)^2}{4x(A^L(x))^3}
\lesssim 
 \frac{(\max\{x,\eps\})^2 \max\{e^{-\frac x\eps}, e^{\frac{1-\bar{\lambda}}{\eps}}\} }{x(\max\{x,\eps\})^3} \lesssim \frac{\max\{e^{-\frac x\eps}, e^{\frac{1}{\sqrt\eps}}\} }{x\max\{x,\eps\}}\lesssim\frac{\max\{e^{-\frac x\eps}, e^{\frac{1}{\sqrt\eps}}\} }{x\eps}\,,$$
where the second inequality follows from $\frac{1-\bar\lambda}{\eps}\le \frac{1-\frac{\lambda+1}{2}}{\eps}= \frac1{2\sqrt{\eps}}$.
Furthermore we have
\begin{equation}
\begin{split}
(\ddot f^L(x))^2&\lesssim \frac{x(\ddot A^L(x))^2}{(A^L(x))^3}+ \frac1{x^3A^L(x)}+ \frac{(\dot A^L(x))^2}{x(A^L(x))^3}+ \frac{x(\dot A^L(x))^4}{(A^L(x))^5}\\
&\lesssim \frac{x\eps^{-2}e^{-\frac x\eps}}{(\max\{x,\eps\})^3}+
\frac1{x^3\max\{x,\eps\}}+ \frac1{x(\max\{x,\eps\})^3}+ \frac{x}{(\max\{x,\eps\})^5}\,.
\end{split}
\end{equation}
Hence 
\begin{equation*}
(\ddot f^L(x))^2\lesssim \frac{1}{\eps^4}+ \frac1{x^3\eps} + \frac{1}{x\eps^3}\lesssim \frac1{x^3\eps}\quad \text{ if }x\in(0,\eps)\,,
\end{equation*}
and 
\begin{equation*}
(\ddot f^L(x))^2\lesssim \frac{e^{-\frac x\eps}}{x^2\eps^2}+ \frac1{x^4}\lesssim \frac1{x^3\eps}\quad \text{ if }x\in(\eps,1)\,.
\end{equation*}

 \noindent
\textit{Step 2:} in this step we show \eqref{eq:est-u}. By \eqref{eq:plancherel-u} it holds
\begin{equation}\label{eq:u^2-1}
	\begin{split}
		\fint_{-L}^{L}(\hat u^L(x,\cdot))^2\dy=
		\fint_{-L_0}^{L_0}(\hat u^L(x,\cdot))^2\dy&=\sum_ {k\in \frac{\pi \mathbb{Z}}{L_0},k\ne 0 }  \frac{{ a^{L}(x,k)}}{k^2} \\&= 
		\bigg(\sum_ {k\in \frac{\pi \mathbb{Z}}{L_0},k\ne 0 }  \frac{{ b^{L_0}(\cdot,k)}}{k^2} \bigg) *\rho_\eps(x) \\
		&
		=\int_0^{+\infty} \bigg(\sum_ {k\in \frac{\pi \mathbb{Z}}{L_0},k\ne 0 }  \frac{{ b^{L_0}(z,k)}}{k^2} \bigg)\rho_\eps(x-z)\dz
		\,.
	\end{split}
\end{equation}
Moreover by the fundamental theorem of calculus and Hölder's inequality we have 
\begin{equation}\label{eq:u^2-2}
	b^{L_0}(z,k)=
	\left(\sqrt{	b^{L_0}(z,k) }\right)^2=
	\left(\int_0^z\frac{b^{L_0}_{,x}(\hat z,k)}{2\sqrt{b^{L_0}(\hat z,k)}}\,{\rm d}\hat z\right)^2
	\le z  \int_0^z \frac{(b^{L_0}_{,x}(\hat z,k))^2}{4{b^{L_0}(\hat z,k)}}\,{\rm d}\hat z
	\,.
\end{equation}
Combining \eqref{eq:u^2-1} with \eqref{eq:u^2-2} we find
\begin{equation*}
	\begin{split}
		\fint_{-L}^{L}(\hat u^L(x,\cdot))^2\dy&\le  \int_0^{+\infty}
		\bigg(
		\sum_ {k\in \frac{\pi \mathbb{Z}}{L_0},k\ne 0 } 
		\int_0^z \frac{(b^{L_0}_{,x}(\hat z,k))^2}{4k^2{b^{L_0}(\hat z,k)}}\,{\rm d}\hat z
		\bigg)z\rho_\eps(x-z)
		\dz
	\end{split}
\end{equation*}
By definition of $b^{L_0}$ it follows that
\begin{equation}\label{eq:u^2-3}
	\begin{split}
		\sum_ {k\in \frac{\pi \mathbb{Z}}{L_0},k\ne 0 } 	\int_0^z \frac{(b^{L_0}_{,x}(\hat z,k))^2}{4k^2{b^{L_0}(\hat z,k)}}\,{\rm d}\hat z &= 	\sum_ {k\in \frac{\pi \mathbb{Z}}{L_0},k\ne 0 } 
		\int_0^{z\wedge\bar\lambda} \frac{(\bar b^{L_0}_{,x}(\hat z,k))^2}{4k^2{\bar b^{L_0}(\hat z,k)}}\,{\rm d}\hat z\\
		&
		\le 
		\int_{(0,z\wedge\bar\lambda)\times\R}\frac{1}{4  k^2}		\bigg( \frac{\di\mu^{L_0}_{,x}}{\di\mu^{L_0}}\bigg)^2\di\mu^{L_0}=:\omega(z)
		\,,
	\end{split}
\end{equation}
where the last inequality can be obtained by arguing exactly as in \eqref{eq:lim_second_term_1}.  
Therefore we deduce that 
\begin{equation*}
	\fint_{-L_0}^{L_0}(\hat u^L(x,\cdot))^2\dy
	\le 
	\int_0^{+\infty}\omega(z) z\rho_\eps(x-z)\dz\,.
\end{equation*}
Note that  $\omega(z)\to 0$ as $z\to0$ and $\omega(z)\le \omega(\bar{\lambda})\le (\lambda^2)\mathcal{F}_\infty(\mu)\le C$.
Let $N\ge 2$ be a natural  number.   Assume $x\in (0,N\eps]$. Since $\omega$ is increasing we have
\begin{equation}
	\begin{split}
		\int_0^{+\infty}\omega(z) z\rho_\eps(x-z)\dz&\le	\omega(2N\eps)\int_0^{2N\eps} z\rho_\eps(x-z)\dz+ \omega(\bar\lambda)	\int_{2N\eps}^{+\infty} z\rho_\eps(x-z)\dz\\
		&\le \omega(2N\eps)\max\{x,\eps\}+CN e^{-N}\eps\\
		&\lesssim\max\{x,\eps\}(\omega(2N\eps)+Ne^{-N})
		\,.
	\end{split}
\end{equation}
If instead $x\in (N\eps,1)$, we get 
\begin{equation}
	\begin{split}
		\int_0^{+\infty}\omega(z) z\rho_\eps(x-z)\dz&\le	
		\omega(\bar\lambda)	\int_{0}^{+\infty} z\rho_\eps(x-z)\dz
		\lesssim x.
	\end{split}
\end{equation}

\noindent
\textit{Step 3:} in this step we show \eqref{eq:est-u^4}.  By the mean value theorem and the fact that $u^L(x,\cdot)$ is $2L_0$-periodic, for fixed $x$, we can find $y_0=y_0(x)\in[-L_0,L_0]$ such that
\begin{equation*}
	\hat	u^L(x,y_0)= \fint_{-L_0}^{L_0}\hat u^L(x,\hat y)\di\hat y=0\,,
\end{equation*}
where the second equality follows by the definition of $\hat u^L$ and the fact that 
\begin{equation*}\begin{split}
&	 \fint_{-L_0}^{L_0}\sin (k\hat y)\di\hat y=\frac k{2L_0}(-\cos(kL_0)+\cos(kL_0))=0\,,\\
& \fint_{-L_0}^{L_0}\cos (k\hat y)\di\hat y=\frac k{2L_0}(\sin(kL_0)-\sin(kL_0))=0\,,
	\end{split}
\end{equation*}
for all $k\in \frac{\pi\Z}{L_0}$.
Thus by the fundamental theorem of calculus, Hölder's inequality and Plancherel it holds
\begin{equation*}\label{eq:abs-u}
	\begin{split}
		|\hat u^L(x,y)|=\left|\int_{y_0}^y\hat u_{,y}(x,y')\dy'\right| &\le \sqrt{L_0}\left(\int_{-L_0}^{L_0}(u^L_{,y})^2\dy'\right)^{\frac12}\\
		&=L_0\sqrt{2A^L(x)}\,.
	\end{split}
\end{equation*}
This  together with steps 1 and 2 yield
\begin{equation*}	\strokedint_{-L}^{L} \int_{-1}^1
	{(\hat u^L)^4} \dx\dy=
	\strokedint_{-L_0}^{L_0} \int_{-1}^1
	{(\hat u^L)^4} \dx\dy\lesssim L_0^2A^L(x) \strokedint_{-L_0}^{L_0} \int_{-1}^1
	{(\hat u^L)^2} \dx\dy\lesssim L_0^2(\max\{x,\eps\})^2\,.
\end{equation*}

\noindent
\textit{Step 4:} in this step we show \eqref{eq:uyy+ux}. By \eqref{plancherel-derivatives} the definition of $a^L$ and \eqref{ub:1} 
\begin{equation}\label{est:uyy}
	\begin{split}
			\fint_{-L}^{L}(u_{,yy}^L)^2\dy&=
		\fint_{-L_0}^{L_0}(u_{,yy}^L)^2\dy= \sum_{k\in\frac{\pi\mathbb{Z}}{L_0}}k^2a^L(x,k)= \sum_{k\in\frac{\pi\mathbb{Z}}{L_0}}k^2b^{L_0}(x,\cdot)*\rho_\eps(x)\\
		& = \sum_ {k\in \frac{\pi \mathbb{Z}}{L_0}}\biggl(k^2\int_0^{\bar \lambda}\bar b^{L_0}(z,k)\rho_\eps(x-z)\dz\biggr)+ \sum_ {k\in \frac{\pi \mathbb{Z}}{L_0}}\bar b^{L_0}(\bar\lambda,k)k^2\int_{\bar\lambda}^{+\infty}\rho_\eps(x-z)\dz\\
		& \le \|\rho_\eps\|_\infty
		\int_{0}^{\bar{\lambda}} \sum_{k\in\frac{\pi\mathbb{Z}}{L_0}}\overline b^{L_0}(x,k)k^2\dx
		+e^{\frac{x-\bar\lambda}{\eps}}
		\fint_{\frac{\lambda+1}{2}}^\lambda \sum_{
			k\in \frac{\pi \mathbb{Z}}{L_0}} \overline b^{L_0}(x,k)k^2\dx\\
		&\lesssim \left(\frac1\eps+ \frac{1}{\lambda-1}e^{\frac{x-\bar\lambda}{\eps}}\right)\int_{0}^{{\lambda}} \sum_{k\in\frac{\pi\mathbb{Z}}{L_0}}\overline b^{L_0}(x,k)k^2\dx\\
		&
		\lesssim
		\left(\frac1\eps+ \frac{1}{\lambda-1}e^{\frac{x-\bar\lambda}{\eps}}\right)\F_\infty^\lambda(\mu^{L_0})\lesssim \frac{1}\eps \,.
	\end{split}
\end{equation}
Since the function $(z_1,z_2)\mapsto{z_1^2}/{z_2}$ is convex by Jensen's inequality we have 
\begin{equation*}
\frac{(a^L_{,x}(x,k))^2}{a^L(x,k)}= \frac{\big(b^{L_0}_{,x}(\cdot,k)*\rho_\eps(x)\big)^2}{b^{L_0}(\cdot,k)*\rho_\eps(x)}
\le 
\frac{(b^{L_0}_{,x}(\cdot,k))^2}{b^{L_0}(\cdot,k)}*\rho_\eps(x)\,.
\end{equation*}
This together with \eqref{plancherel-derivatives} and \eqref{derivatives-a} imply
\begin{equation}\label{est:ux}
\begin{split}
		\fint_{-L}^{L}(u_{,x}^L)^2\dy=
	\fint_{-L_0}^{L_0}(u_{,x}^L)^2\dy&= \sum_{k\in\frac{\pi\mathbb{Z}}{L_0},k\ne0 }\frac{1}{4k^2}\frac{(a^L_{,x}(x,k))^2}{a^L(x,k)}
	 \le 
	 \sum_{k\in\frac{\pi\mathbb{Z}}{L_0},k\ne0 }
	 \frac{1}{4k^2}\frac{(b^{L_0}_{,x}(\cdot,k))^2}{b^{L_0}(\cdot,k)}*\rho_\eps(x)\\
	 &= 
	 	 \sum_{k\in\frac{\pi\mathbb{Z}}{L_0},k\ne0 }
	\bigg( \frac{1}{4k^2}\int_0^{\bar\lambda}
	 \frac{(b^{L_0}_{,x}(z,k))^2}{b^{L_0}(z,k)}\rho_\eps(x-z)\dz\bigg)\\
	 &\le \|\rho_\eps\|_\infty \int_0^{\bar\lambda}	 \sum_{k\in\frac{\pi\mathbb{Z}}{L_0},k\ne0 }
	 \frac{1}{4k^2}
	 \frac{(b^{L_0}_{,x}(x,k))^2}{b^{L_0}(x,k)} \dx\le \frac1\eps \F_\infty^\lambda(\mu^{L_0})\lesssim \frac{1}\eps \,.
\end{split}
\end{equation}
Thus combining \eqref{est:uyy} with \eqref{est:ux} we infer \eqref{eq:uyy+ux}.\\

\noindent
\textit{Step 5:} in this step we show \eqref{eq:weak-conv}. By recalling Definition \ref{def:muL} we have 
\begin{equation*}
\hat\mu^L:=	\mu^L(\hat u^L)= \sum_ {k\in \frac{\pi \mathbb{Z}}{L_0} } a^L(x,k) \mathcal{L}^1\res(-1,1)\times\delta_k\,,
\end{equation*}
and 
\begin{equation*}
\hat\mu^L_{,x}=	\mu^L_{,x}(\hat u^L)= \sum_ {k\in \frac{\pi \mathbb{Z}}{L_0} }a^L_{,x}(x,k) \mathcal{L}^1\res(-1,1)\times\delta_k\,,
\end{equation*}
Let $\varphi \in C^\infty_c((-1,1)\times\R)$. Then 
\begin{equation*}
	\begin{split}
		\int_{(-1,1)\times\R}\varphi\di\hat \mu^L&= 
		\int_{(-1,1)\times\R}\varphi\di\hat\mu^L- 	\int_{(0,1)\times\R}\varphi\di\mu^{L_0}
		+ \int_{(0,1)\times\R}\varphi\di\mu^{L_0}\,.
	\end{split}
\end{equation*}
By Lemma \ref{lem:discretisation} we have that 
\begin{equation*}
	\lim_{L\to\infty}\int_{(0,1)\times\R}\varphi\di\mu^{L_0}=\int_{(0,1)\times\R}\varphi\di\mu\,,
\end{equation*}
hence it suffices to show that 
\begin{equation*}
\lim_{L\to +\infty}	\biggl(\int_{(-1,1)\times\R}\varphi\di\hat\mu^L- 	\int_{(0,1)\times\R}\varphi\di\mu^{L_0} \biggr)=0\,.
\end{equation*}
Indeed by \eqref{est:A} and \eqref{eq:constraint} we have 
\begin{equation*}
	\begin{split}
		\biggl|\int_{(-1,1)\times\R}\varphi\di\hat\mu^L- &	\int_{(0,1)\times\R}\varphi\di\mu^{L_0} \biggr|\le
\|\varphi\|_\infty \left|
	\int_{(-1,1)\times\R}\di\hat\mu^L- 	\int_{(0,1)\times\R}\di\mu^{L_0}\right|\\&=\|\varphi\|_\infty
	\Bigg| \int_{-1}^1 \sum_ {k\in \frac{\pi \mathbb{Z}}{L_0} } a^L(x,k)\dx- \int_0^1 \sum_ {k\in \frac{\pi \mathbb{Z}}{L_0} }\bar b^L(x,k)\dx\Bigg|\\
	&=\|\varphi\|_\infty   \left|\int_{-1}^1\eps \Big(e^{-\frac{|x|}{\eps}}-e^{\frac{x-\bar\lambda}{\eps}}\Big)\dx\right|\le C\eps^2\to0\quad \text{as }L\to+\infty\,.
	\end{split}
\end{equation*}
Moreover we have 
\begin{equation*}
		\int_{(-1,1)\times\R}\varphi\di\hat \mu^L_{,x}= -	\int_{(-1,1)\times\R}\varphi_{,x}\di\hat \mu^L\to 
	-	\int_{(-1,1)\times\R}\varphi_{,x}\di \mu= 
			\int_{(-1,1)\times\R}\varphi\di \mu_{,x}\,.
\end{equation*}

\noindent
\textit{Step 6:} in this step we show \eqref{eq:limsup}.
By  \eqref{plancherel-derivatives}, \eqref{def:b} we have 
\begin{equation}\label{ub:est-lt1}
	\begin{split}
		&	\strokedint_{-L_0}^{L_0} \int_{-1}^1 (\hat u^L_{,yy})^2\dx \dy= \int_{-1}^1\sum_ {k\in \frac{\pi \mathbb{Z}}{L_0}}a^L(x,k)k^2 \dx=
		\int_{-1}^1\sum_ {k\in \frac{\pi \mathbb{Z}}{L_0}}b^{L_0}(\cdot,k)*\rho_\eps(x)k^2 \dx\\
		& =\int_{-1}^1\sum_ {k\in \frac{\pi \mathbb{Z}}{L_0}}\biggl(k^2\int_0^{\bar \lambda}\bar b^{L_0}(z,k)\rho_\eps(x-z)\dz\biggr)\dx+ \sum_ {k\in \frac{\pi \mathbb{Z}}{L_0}}\bar b^{L_0}(\bar\lambda,k)k^2\int_{-1}^1\int_{\bar\lambda}^{+\infty}\rho_\eps(x-z)\dz\,.
	\end{split}
\end{equation}
Fubini's theorem yields
\begin{equation}\label{ub:est-lt2}
	\int_{-1}^1\sum_ {k\in \frac{\pi \mathbb{Z}}{L_0}}\biggl(k^2\int_0^{\bar \lambda}\bar b^{L_0}(z,k)\rho_\eps(x-z)\dz\biggr)\dx\le \int_0^{\bar{\lambda}}\sum_ {k\in \frac{\pi \mathbb{Z}}{L_0}} \bar b^{L_0}(x,k)k^2\dx\,.
\end{equation}
while from \eqref{ub:1} we deduce 
\begin{equation}\label{ub:est-lt3}
	\sum_ {k\in \frac{\pi \mathbb{Z}}{L_0}}\bar b^{L_0}(\bar\lambda,k)k^2\int_{-1}^1\int_{\bar\lambda}^{+\infty}\rho_\eps(x-z)\dz\dx\le
	\frac{\eps}{2}\left(e^{\frac{1-\bar\lambda}{\eps}}- e^{\frac{-1-\bar\lambda}{\eps}}
	\right)
	\fint_{\frac{\lambda+1}{2}}^\lambda \sum_{
		k\in \frac{\pi \mathbb{Z}}{L_0}} \overline b^{L_0}(x,k)k^2\dx \,.
\end{equation}
Analogously from \eqref{plancherel-derivatives}, \eqref{derivatives-a} and Jensen's inequality it holds
\begin{equation}\label{ub:est-lt4}
	\begin{split}
		\strokedint_{-L_0}^{L_0} \int_{-1}^1  (\hat u^L_{,x})^2\dx \dy&= \int_{-1}^1\sum_ {k\in \frac{\pi \mathbb{Z}}{L_0},k\ne0}\frac{1}{4k^2}
		\frac{(a^L_{,x}(x,k))^2}{ a^L(x,k)} \dx\\
		&\le 
		\int_{-1}^1\sum_ {k\in \frac{\pi \mathbb{Z}}{L_0},k\ne0}\frac{1}{4k^2}
		\frac{(b^{L_0}_{,x}(\cdot,k))^2}{ b^{L_0}(\cdot,k)} *\rho_\eps(x)\dx\\
		&\le \int_{0}^{\bar\lambda}
		\sum_ {k\in \frac{\pi \mathbb{Z}}{L_0},k\ne0}\frac{1}{4k^2}
		\frac{(\bar b^{L_0}_{,x}(x,k))^2}{\bar b^{L_0}(x,k)} \dx\,.
	\end{split}
\end{equation}
Gathering together \eqref{ub:est-lt1}--\eqref{ub:est-lt4} we obtain 
\begin{equation}\label{eq:step1}
	\fint_{-L}^L\int_{-1}^1\big((\hat u^L_{,x})^2+ (\hat u^L_{,yy})^2\big) \dx\dy\le \biggl(1+ \frac{C\eps(e^{\frac{1-\bar\lambda}{\eps}-e^\frac{-1-\bar\lambda}{\eps}})}{1-\lambda}\biggr)\mathcal F^\lambda_\infty (\mu^{L_0}) \,,
\end{equation}
and hence by letting $L\to+\infty$ and recalling Lemma \ref{lem:discretisation} \ref{(iii)discr} we deduce \eqref{eq:limsup}.\\

\noindent
 \textit{Step 7:} in this step we show \eqref{eq:uxx-uxy-uxyy}. By \eqref{plancherel-derivatives}, \eqref{derivatives-a}
 \begin{equation}\label{eq:uxy}
 		\fint_{-L}^{L}\int_{-1}^1(\hat u^L_{,xy})^2\dx\dy= 
 	\fint_{-L_0}^{L_0}\int_{-1}^1(\hat u^L_{,xy})^2\dx\dy= 
\int_{-1}^1 	\sum_ {k\in \frac{\pi \mathbb{Z}}{L_0},k\ne 0 } \frac14 \frac{(a^L_{,x}(x,k))^2}{a^L(x,k)}\dx\,.
 \end{equation}
 Next we observe that \eqref{eq:prop-moll} yields
 \begin{equation}\label{eq:ax}
 	{(a^L_{,x}(x,k))^2}={(b^L(\cdot,k)*\dot\rho_\eps(x))^2}
 	\le	 {(b^L(\cdot,k)*|\dot\rho_\eps|(x))^2} = \frac1{\eps^2}{(b^L(\cdot,k)*\rho_\eps(x))^2}
 	\le \frac1{\eps^2}(a^L(x,k))^2\,,
 \end{equation}
 so that combining \eqref{eq:uxy} with \eqref{eq:ax} and recalling \eqref{est:A} we obtain
 \begin{equation}\label{eq:uxy-bis}
	\fint_{-L}^{L}\int_{-1}^1(\hat u^L_{,xy})^2\dx\dy\le\frac1{4\eps^2} \int_{-1}^1\sum_ {k\in \frac{\pi \mathbb{Z}}{L_0} } 	(a^L(x,k))^2\dx\lesssim \frac{1}{\eps^2}\,.
 \end{equation}
In a similar way \eqref{plancherel-derivatives} and \eqref{derivatives-a} give 
 \begin{equation}\label{eq:uxx}
 	\begin{split}
 		 		\fint_{-L}^{L}
 		\int_{-1}^1(\hat u_{,xx}^L)^2\dx\dy&=
 		\fint_{-L_0}^{L_0}
 	\int_{-1}^1(\hat u_{,xx}^L)^2\dx\dy
 		=
 		\int_{-1}^1\sum_{k\in\frac{\pi\mathbb{Z}}{L_0},k\ne0}\frac1{k^2}\Big[\Big(\sqrt{a^L(x,k)}\Big)_{,xx}\Big]^2\dx
 		\\
 		&\lesssim \int_{-1}^{1}\sum_{k\in\frac{\pi\mathbb{Z}}{L_0},k\ne0}\frac1{k^2}\frac{(a^L_{,xx}(x,k))^2}{a^L(x,k)}\dx+
 		\int_{-1}^{1}\sum_{k\in\frac{\pi\mathbb{Z}}{L_0},k\ne0}\frac1{k^2}\frac{(a^L_{,x}(x,k))^4}{(a^L(x,k))^3}\dx\\
 		&\lesssim \frac{1}{\eps^2}\int_{-1}^{1}\sum_{k\in\frac{\pi\mathbb{Z}}{L_0},k\ne0}\frac1{4k^2}\frac{(a^L_{,x}(x,k))^2}{a^L(x,k)}\dx\\
 		&\lesssim \frac{1}{\eps^2} \int_{-1}^1\fint_{-L_0}^{L_0}(\hat u^L_{,x})^2\dy\dx
 		\lesssim\frac{1}{\eps^2}\mathcal F^\lambda_\infty (\mu^{L_0})	\lesssim\frac{1}{\eps^2}
 		\,,
 	\end{split}
 \end{equation}
where the second inequality follows from 
\begin{equation*}
	{(a^L_{,x}(x,k))^4}= (a^L_{,x}(x,k))^2{(b^L(\cdot,k)*\dot\rho_\eps(x))^2}
	\le \frac1{\eps^2} (a^L_{,x}(x,k))^2 (a^L(x,k))^2\,,
\end{equation*}
and 
\begin{equation*}
	\begin{split}
		{(a^L_{,xx}(x,k))^2}= {(b^L_{,x}(\cdot,k)*\dot\rho_\eps(x))^2}\le 
		\frac1{\eps^2}(a^L_{,x}(x,k))^2\,.
	\end{split}
\end{equation*}
Moreover appealing again to \eqref{eq:ax} we find
\begin{equation}\label{eq:uyyx}
	\begin{split}\fint_{-L}^{L}	\int_{-1}^1(\hat u_{,yyx}^L)^2\dx\dy&=
\fint_{-L_0}^{L_0}	\int_{-1}^1(\hat u_{,yyx}^L)^2\dx\dy
=
\int_{-1}^1\sum_{k\in\frac{\pi\mathbb{Z}}{L_0},k\ne0}{k^2}\Big[\Big(\sqrt{a^L(x,k)}\Big)_{,x}\Big]^2\dx\\
&= \int_{-1}^1\sum_{k\in\frac{\pi\mathbb{Z}}{L_0},k\ne0 }\frac{k^2}4\frac{(a^L_{,x}(x,k))^2}{a^L(x,k)}\dx \lesssim\frac1{\eps^2}\int_{-1}^1\sum_{k\in\frac{\pi\mathbb{Z}}{L_0},k\ne0 }{k^2}{a^L(x,k)}\dx\\
&
\lesssim \frac1{\eps^2}\int_{-1}^1\fint_{-L_0}^{L_0}(\hat u^L_{,yy})^2\dy\dx
\lesssim\frac{1}{\eps^2}\mathcal F^\lambda_\infty (\mu^{L_0})	\lesssim\frac{1}{\eps^2}\,.
	\end{split}
\end{equation}
Eventually gathering together \eqref{eq:uxy-bis}--\eqref{eq:uyyx} we deduce \eqref{eq:uxx-uxy-uxyy}.\\

\noindent 
\textit{Step 8:} in this step we show \eqref{eq:ux}.  Analogously to step 3 we can find $y_0\in[-L_0,L_0]$ such that
\begin{equation*}
\hat	u^L_{,x}(x,y_0)= \fint_{-L_0}^{L_0}\hat u^L_{,x}(x,\hat y)\di\hat y=0\,,
\end{equation*}
so that by the fundamental theorem of calculus, Hölder's inequality it holds
\begin{equation}\label{eq:abs-u_x}
	\begin{split}
		|\hat u_{,x}^L(x,y)|=\left|\int_{y_0}^y\hat u_{,xy}(x,y')\dy'\right| &\le \sqrt{L_0}\left(\int_{-L_0}^{L_0}(u^L_{,xy})^2\dy'\right)^{\frac12}\\
		&=\sqrt 2L_0\Biggl(	\sum_ {k\in \frac{\pi \mathbb{Z}}{L},k\ne0}\frac{1}{4}
		\frac{ (a^L_{,x}(x,k))^2}{a^L(x,k)}\Biggr)^{\frac12}\\
		& \lesssim L_0\Bigg(\frac1{4\eps^2} \sum_ {k\in \frac{\pi \mathbb{Z}}{L},k\ne0} a^L(x,k)\Bigg)^{\frac12}\lesssim \frac{L_0}{\eps}
		\,,
	\end{split}
\end{equation}
where the last two inequalities follow from \eqref{eq:ax} and \eqref{est:A}. 
Therefore \eqref{eq:abs-u_x} gives
\begin{equation*}\label{ub:est4}
	\begin{split}
		\strokedint_{-L_0}^{L_0} \int_{-1}^1
		{(\hat u^L_{,x})^4} \dx\dy&\lesssim  \frac{ L_0^2}{\eps^2 }
		\strokedint_{-L_0}^{L_0} \int_{-1}^1
		{(\hat u^L_{,x})^2} \dx\dy
		\lesssim \frac{ L_0^2}{\eps^2}\,,
	\end{split}
\end{equation*}
and the proof is concluded.

\end{proof}
We are now in a position to prove Proposition \ref{prop:upb}.
\begin{proof}[Proof of Proposition \ref{prop:upb}]
Let $\mu\in\mathcal{M}_\infty$ be as in the statement.  Let $\eps=\eps(L)>0$ and $n=n(L)\in\mathbb{N}$ be such that
\begin{equation}\label{cond_parameters1}
	\lim_{L\to +\infty}\eps(L)=0\,,\quad
	\lim_{L\to +\infty}n(L)=\lim_{L\to +\infty}\frac L{n(L)}=+\infty\,,
\end{equation}
to be chosen later. Let
 $\hat u^L\in \mathcal{A}_L^{\rm out}\cap \mathcal{A}_{L_0}^{\rm out}$  with $L_0:=L/n(L)$ be the function given by Lemma \ref{lem:moll}. Recall that
 $$	A^L(x)=\frac12\fint_{-L}^L(\hat{u}^L_{,y})^2\dy\ \text{ and }\
 f^L(x)=\sqrt{\frac{x}{A^L(x)}}\ \text{ for }\ x\ge0\,.$$
  Furthermore we let $ M=M(L)\in\mathbb{N}$, $M\ge 2$   to be chosen later  such that   setting $\delta=\delta(L):=\frac{\eps(L)}{M(L)}<\eps(L)$ we have
  \begin{equation}\label{cond_parameters2}
\lim_{L\to +\infty}M(L)=+\infty\,,\quad \text{and}\quad \lim_{L\to+\infty}\delta(L)=\lim_{L\to +\infty}M^2(L)\delta(L)=0\,.
  \end{equation}
 %
We consider $\psi_\delta\in C^\infty(\R)$ such that 
\begin{equation*}
\psi_\delta\equiv0\ \text{ in } (-\infty,\delta]\,, \quad \psi_\delta\equiv1\ \text{ in } [2\delta,+\infty)\,,\quad |\dot\psi_\delta(x)|\le C\delta^{-1}\,,\quad |\ddot\psi_\delta(x)|\le C\delta^{-2}\,.
\end{equation*}
Note that $\dot\psi_\delta=\ddot\psi_\delta=0$ in $(\delta,2\delta)^c$.
We next define  $(w^L,u^L)= \big((w^{L}_1,w^{L}_2),u^L\big) $ as follows:
\begin{equation*}
	\begin{split}
	&	u^L(x,y):= \psi_\delta(x)f^L(x)\hat u^L(x,y) \,,\\
	& w_2^{L}(x,y):=\psi_\delta^2(x)xy+ B^{L}(x)- \frac12\int_0^y(u^{L}_{,y})^2 \dy' 	\,,\\
	&w^{L}_1(x,y):=x-\frac{1}{L^2}\int_0^y(w^L_{2,x}+u^L_{,x}u^L_{,y})\dy'\,,
	\end{split}
\end{equation*}
where 
\begin{equation*}
B^L(x):=\frac12 \fint_{-L_0}^{L_0}\int_0^y(u^L_{,y})^2\dy'\dy-\fint_{-L_0}^{L_0}\int_0^x u^L_{,x}u^L_{,y}\dx'\dy\,.
\end{equation*}
Clearly $u^L\in \mathcal{A}_L^{\rm out}\cap \mathcal{A}_{L_0}^{\rm out}$.  We show that  $w^L\in \mathcal{A}_{L}^{\rm in}\cap\mathcal{A}_{L_0}^{\rm in}$. Precisely to see that $w^L(x,\cdot)$ is $2L_0$ periodic we use the following fact: 

\smallskip
\textit{A differentiable function $h$ is $T$-periodic if $h'$ is $T$-periodic and $h(t)=h(t+T)$ for some $t$. }

\smallskip

\noindent
The function $w^L_{2,y}(x,\cdot)$ is $2L_0$-periodic, since $u^L_{,y}$ is, and from \eqref{plancherel-derivatives} satisfies
\begin{equation*}
	\begin{split}
		w_2^L(x,L_0)-w_2^L(x,-L_0)&=2L_0\psi_\delta^2(x)x-\int_{-L_0}^{L_0}( u^L_{,y})^2\dy \\
	&=
	2L_0\psi_\delta^2(x)(x-(f^L(x))^2A^L(x))
	=
	0\,,
	\end{split}
\end{equation*}
from which we deduce $w^L_2(x,\cdot)$ is $2L_0$-periodic. Using this periodicity, and in particular also of $w^L_{2,x}$, we see that $w^L_{1,y}(x,\cdot)$ is $2L_0$-periodic. Moreover  we have
\begin{equation*}
	\begin{split}
	w_1^L(x,L_0)&-w_1^L(x,-L_0)=-\frac{1}{L^2}\int_{-L_0}^{L_0}(w^L_{2,x}+u^L_{,x}u^L_{,y})\dy\\
	&=
	-\frac{1}{L^2}\left(2L_0\dot B^L(x)-\int_{-L_0}^{L_0}\int_0^yu_{,y}^Lu^L_{,xy}\dy'\dy
	+ \int_{-L_0}^{L_0}u_{,x}^Lu_{,y}^L\dy
	\right)=0\,,
	\end{split}
\end{equation*}
where the second and the third equalities follow from
\begin{equation*}
w^L_{2,x}= (\psi_\delta^2(x)x)'y+\dot B^L(x)-\int_0^yu^L_{,y}u^L_{,xy}\dy'\,,
\end{equation*}
and
\begin{equation*}
2L_0\dot B^L(x)=2L_0\left(\fint_{-L_0}^{L_0}u^L_{,y}u^L_{,xy}\dy- \fint_{-L_0}^{L_0}u^L_{,x}u^L_{,y}\dy\right)\,.
\end{equation*}
Thus we deduce that $w^L_{1,y}$ is $2L_0$-periodic.  For the reader convenience we divide the rest of the proof into several  steps. 
We will repeatedly use that the averaged integral over $(-L,L)$ of a $2L_0$-periodic function is equal to  the averaged integral over $(-L_0,L_0)$ of the same function.\\

\noindent
\textit{Step 1:} we show that $(w^L,u^L)$ converges to $\mu$ in the sense of Definition \ref{def:convergence}. 
We have that 
	\begin{equation*}
	\hat	u^{L}(x,y)= \sum_ {k\in \frac{\pi \mathbb{Z}}{L_0}, k>0 }  a^L_k(x) \sin(ky) + \sum_ {k\in \frac{\pi \mathbb{Z}}{L_0}, k < 0 }  a^L_k(x)\cos(ky)\,,
\end{equation*}
so that 
\begin{equation*}
	u^L(x,y)= \sum_ {k\in \frac{\pi \mathbb{Z}}{L_0}, k>0 }\psi_\delta(x)f^L(x)  a^L_k(x) \sin(ky) + \sum_ {k\in \frac{\pi \mathbb{Z}}{L_0}, k < 0 } \psi_\delta(x)f^L(x) a^L_k(x)\cos(ky)\,.
\end{equation*}
Therefore we get
\begin{equation*}
	\begin{split}
		\mu^L:=	\mu^L(u^L)&=\sum_ {k\in \frac{\pi \mathbb{Z}}{L_0} }\psi^2_\delta(x)(f^L(x) )^2 (a^L_k(x))^2k^2 \mathcal{L}^1\res(-1,1)\times  \delta_k\\
		&=\sum_ {k\in \frac{\pi \mathbb{Z}}{L_0} }\psi_\delta^2(x)(f^L(x))^2  (a^L_k(x))^2k^2  \mathcal{L}^1\res(\delta,1)\times \delta_k \,.
	\end{split}
\end{equation*}
We show $\mu^L\stackrel{*}{\rightharpoonup}\mu$. We fix $\varphi\in C_c^\infty((-1,1)\times\R)$ and we write
\begin{equation*}
	\int_{(-1,1)\times\R}\varphi\di\mu^L= 	\left(\int_{(-1,1)\times\R}\varphi\di\mu^L- 	\int_{(-1,1)\times\R}\varphi\di\mu^L(\hat u^L)\right)+\int_{(-1,1)\times\R}\varphi\di\mu^L(\hat u^L)\,.
\end{equation*}
By Lemma \ref{lem:moll} it holds 
\begin{equation*}
	\lim_{L\to+\infty}\int_{(-1,1)\times\R}\varphi\di\mu^L(\hat u^L)=\int_{(-1,1)\times\R}\varphi\di\mu\,.
\end{equation*}
Therefore it is sufficient to show that 
\begin{equation*}
	\lim_{L\to+\infty}\left(\int_{(-1,1)\times\R}\varphi\di\mu^L- 	\int_{(-1,1)\times\R}\varphi\di\mu^L(\hat u^L)\right)=0\,.
\end{equation*}
Recalling that $\psi_\delta=0$ in $(-1,\delta)$, $\psi_\delta=1$ in $(2\delta,1)$ and $M\eps=M\delta^2\ge 2\delta$, we have 
\begin{equation*}
	\begin{split}
	 \bigg|
		\int_{(-1,1)\times\R}\varphi\di\mu^L- 	\int_{(0,1)\times\R}\varphi\di\mu^{L}(\hat u^L)\bigg|
	&	=
		\bigg| \int_{0}^1 \sum_ {k\in \frac{\pi \mathbb{Z}}{L_0} }\varphi(x,k) (a^L_k(x))^2k^2 \Big(\psi^2_\delta(x)(f^L(x))^2-1\Big)\dx
		\bigg|\\
		&
		\le \|\varphi\|_{\infty}
		 \bigg|  \int_0^{M\eps} \Big(\psi^2_\delta(x)(f^L(x))^2-1\Big)\sum_ {k\in \frac{\pi \mathbb{Z}}{L_0} } (a^L_k(x))^2k^2 \dx \bigg|\\
		 &+ 
		  \|\varphi\|_{\infty}
		 \bigg|  \int_{M\eps}^1\Big((f^L(x))^2-1\Big) \sum_ {k\in \frac{\pi \mathbb{Z}}{L_0} } (a^L_k(x))^2k^2\dx \bigg|\,.
	\end{split}
\end{equation*}
Since $$\sum_ {k\in \frac{\pi \mathbb{Z}}{L_0} } (a^L_k(x))^2k^2= A^L(x)=\frac12\fint_{-L}^L(\hat{u}_{,y}(x,\cdot))^2\dy\,,$$ by \eqref{eq:AL} and \eqref{eq:f(x)} we have
\begin{equation*}
 \Big(\psi^2_\delta(x)(f^L(x))^2-1\Big)\sum_ {k\in \frac{\pi \mathbb{Z}}{L_0} }
  (a^L_k(x))^2k^2\le\Big(\psi^2_\delta(x)(1+o_L(1))-1\Big) \max\{x,\eps\}
\end{equation*}
which together with \eqref{cond_parameters2} imply
\begin{equation*}
\|\varphi\|_{\infty}
\bigg|  \int_0^{M\eps} \sum_ {k\in \frac{\pi \mathbb{Z}}{L_0} } (a^L_k(x))^2k^2(\psi^2_\delta(x)(f^L(x))^2-1)\dx \bigg|\le C M\eps =CM^2\delta\to0\quad \text{ as } L\to+\infty\,.
\end{equation*}
Whereas \eqref{eq:f(x)}, \eqref{eq:f(x)-2} with $N=M$ and the fact that $M=M(L)\to+\infty$ imply 
$$ (f^L(x))^2=1+o_L(1)\quad \text{ in }\quad[M\eps,1)\,,$$
so that 
\begin{equation*}
  \|\varphi\|_{\infty}
\bigg|  \int_{M\eps}^1 \sum_ {k\in \frac{\pi \mathbb{Z}}{L_0} } a^L_k(x)(f^L(x)-1)\dx \bigg|\le o_L(1)\to0\quad \text{ as } L\to+\infty\,.
\end{equation*}
Eventually by duality we have 
\begin{equation*}
	\int_{\R\times(-1,1)}\varphi\di \mu^L_{,x}= -	\int_{\R\times(-1,1)}\varphi_{,x}\di \mu^L\to 
	-\int_{\R\times(-1,1)}\varphi_{,x}\di \mu= 
	\int_{\R\times(-1,1)}\varphi\di\mu_{,x}\,,
\end{equation*}
which in turn implies  $\mu^L_{,x}\stackrel{*}{\rightharpoonup}\mu_{,x}$.
\medskip

\noindent
\textit{Step 2:} we show that 
\begin{equation}\label{ub:step2}
	\limsup_{L\to+\infty}
	\fint_{-L}^L\int_{-1}^1((u^L_{,x})^2  + (u^L_{,yy})^2  )\dx\dy
	\le \mathcal{F}_\infty(\mu)+C\lim_{L\to +\infty} \omega(2M^2\delta)\log M\,.
\end{equation}
To this purpose we note that 
\begin{equation}\label{der:ux}
	u^L_{,x}(x,y)= \psi_\delta(x)f^L(x)\hat u^L_{,x}(x,y)+ \dot\psi_\delta(x)f^L(x)\hat u^L(x,y)+
	\psi_\delta(x)\dot f^L(x)\hat u^L(x,y)\,.
\end{equation}
Therefore by Young's inequality
\begin{equation}\label{ub:step2-0}
\begin{split}
		\fint_{-L}^L\int_{-1}^1(u^L_{,x})^2  \dx\dy
	&	\le (1+\alpha) 
			\fint_{-L}^L\int_{\delta}^1(f^L(x))^2(\hat u^L_{,x})^2  \dx\dy\\
			&+ 2(1+\alpha^{-1}) \frac1{\delta^2}	\fint_{-L}^L\int_{\delta}^{2\delta}(f^L(x))^2(\hat u^L)^2  \dx\dy\\
				&+ 2(1+\alpha^{-1})	\fint_{-L}^L\int_{\delta}^{1}(\dot f^L(x))^2(\hat u^L)^2  \dx\dy\,,
\end{split}
\end{equation}
for any $\alpha>0$. In this way by recalling \eqref{eq:f(x)}  \eqref{eq:est-u} and \eqref{eq:f'(x)-f''(x)}  we have 
\begin{equation}\label{ub:step2-1}
	\fint_{-L}^L\int_{-1}^1(f^L(x))^2(\hat u^L_{,x})^2  \dx\dy\le 
	 (1+o_L(1))	\fint_{-L}^L\int_{-1}^1(\hat u^L_{,x})^2  \dx\dy\,,
\end{equation}
\begin{equation}\label{ub:step2-2}
	\frac1{\delta^2}	\fint_{-L}^L\int_{\delta}^{2\delta}(f^L(x))^2(\hat u^L)^2  \dx\dy\lesssim \frac1{\delta^2}{\delta}(\omega(2M\eps)+Me^{-M})\int_\delta^{2\delta}\dx\lesssim (\omega(2M\eps)+Me^{-M})\,,
\end{equation}
and
\begin{equation}\label{ub:step2-3}
	\begin{split}
\fint_{-L}^L\int_{\delta}^{1}(\dot f^L(x))^2(\hat u^L)^2  \dx\dy&\lesssim  
\int_\delta^\eps \frac1{x\eps}\eps (\omega(2M\eps)+Me^{-M})\dx \\
&+ \int_{\eps}^{\sqrt{\eps}} \frac{e^{-\frac x\eps}}{x\eps}x(\omega(2M\eps)+Me^{-M})\dx
+ \int_{\sqrt\eps}^{1}\frac{e^{-\frac1{\sqrt\eps}}}{x\eps}x\dx\\
& \lesssim( \log\frac{\eps}{\delta}+1)(\omega(2M\eps)+Me^{-M})+ \frac{e^{-\frac1{\sqrt\eps}}}{\eps}\,.
	\end{split}
\end{equation}
Gathering together \eqref{ub:step2-0}--\eqref{ub:step2-3} and recalling that $\delta=\eps/M$ we deduce 
\begin{equation}\label{lead-1}
	\fint_{-L}^L\int_{-1}^1(u^L_{,x})^2  \dx\dy\le (1+\bar \alpha) 	\fint_{-L}^L\int_{-1}^1(\hat u^L_{,x})^2  \dx\dy+ 
	C(\log M+1)(\omega(2M^2\delta)+Me^{-M})+o_L(1)\,,
\end{equation}
with $\bar{\alpha}:=\alpha+\alpha o_L(1)+o_L(1)$.
Moreover as $u^L_{,yy}=\psi_\delta(x)f^L(x)\hat u^L_{,yy}$ we get
\begin{equation}\label{lead-2}
	\fint_{-L}^L\int_{-1}^1(u^L_{,yy})^2  \dx\dy\le 	\fint_{-L}^L\int_{-1}^1(f^L(x))^2(\hat u^L_{,x})^2  \dx\dy
	\le 	\fint_{-L}^L\int_{-1}^1(\hat u^L_{,x})^2  \dx\dy\,.
\end{equation}
By \eqref{lead-1}, \eqref{lead-2}, \eqref{eq:limsup} and the fact that $M^2\delta\to0$ we finally deduce
\begin{equation*}
	\limsup_{L\to+\infty}
\fint_{-L}^L\int_{-1}^1((u^L_{,x})^2  + (u^L_{,yy})^2  )\dx\dy\le (1+\alpha)\mathcal{F}_\infty(\mu)+ C\lim_{L\to +\infty}\omega(2M^2\delta)\log M \,.
\end{equation*}
Eventually by the arbitrariness of $\alpha$ we infer the desired estimate.
\medskip

\noindent
\textit{Step 3:} we show that
\begin{equation}\label{ub:step3}
L^2\strokedint_{-L}^{L} \int_{-1}^1
\biggl( w^L_{1,x}+\frac{(u^L_{,x})^2}{2L^2}-1 \biggr)^2 \dx\dy\lesssim 
\frac {L_0^4}{L^2} \frac1{\delta^2\eps}\lesssim \frac {L_0^4}{L^2}\frac{1}{\delta^3M}\,.
\end{equation}
By Young's inequality we have
\begin{equation}\label{ub:est1}
	\begin{split}
L^2 \strokedint_{-L}^{L}
 \int_{-1}^1
\biggl( w^L_{1,x}+\frac{(u^L_{,x})^2}{2L^2}-1 \biggr)^2 \dx\dy
\lesssim  \fint_{-L}^{L} \int_{-1}^1
\frac{(u^L_{,x})^4}{L^2} \dx\dy+L^2
\strokedint_{-L}^{L}\int_{-1}^1
\bigl( w^L_{1,x}-1 \bigr)^2 \dx\dy\,.
	\end{split}
\end{equation}
We estimate the first term on the right hand-side of \eqref{ub:est1}. 
By \eqref{der:ux} it follows
\begin{equation}\label{eq:est-term1}
	\begin{split}
	\fint_{-L}^{L} \int_{-1}^1&
	{(u^L_{,x})^4} \dx\dy \lesssim 	\fint_{-L}^{L} \int_{-1}^1(f^L(x))^4
	{(\hat u^L_{,x})^4} \dx\dy\\&+ \frac{1}{\delta^4}
	\fint_{-L}^{L} \int_{\delta}^{2\delta}(f^L(x))^4
	{(\hat u^L)^4} \dx\dy
	+ 	\fint_{-L}^{L} \int_{\delta}^1(\dot f^L(x))^4	{(\hat u^L)^4} \dx\dy\,.
	\end{split}
\end{equation}
By \eqref{eq:f(x)} and \eqref{eq:ux} we have 
\begin{equation}
	\fint_{-L}^{L} \int_{-1}^1(f^L(x))^4
{(\hat u^L_{,x})^4} \dx\dy\lesssim\frac{L_0^2}{\eps^2}\,,
\end{equation}
whereas from \eqref{eq:f(x)}, \eqref{eq:est-u^4}, and the fact that $x\in (\delta,2\delta)$ we get
\begin{equation}
	\begin{split}
\frac{1}{\delta^4}
\fint_{-L}^{L} \int_{\delta}^{2\delta}(f^L(x))^4
{(\hat u^L)^4} \dx\dy& \lesssim \frac{1}{\delta^4} \frac{\delta^2}{\eps^2} L_0^2\int_{\delta}^{2\delta}(\max\{x,\eps\})^2
\dx
\lesssim \frac{L_0^2}{\delta}\,.
	\end{split}
\end{equation}
Finally by  \eqref{eq:f'(x)-f''(x)}   and \eqref{eq:est-u^4} 
\begin{equation}\label{eq:est-term1-3}
\begin{split}
	\fint_{-L}^{L} \int_{\delta}^1(\dot f^L(x))^4	{(\hat u^L)^4} &\dx\dy\lesssim
	L_0^2 \int_{\delta}^1 \frac{1}{x^2\eps^2}(\max\{x,\eps\})^2
	\dx\lesssim L_0^2\left(\frac1\delta+\frac1{\eps^2}\right)
	\,.
\end{split}
\end{equation}
Thus gathering together \eqref{eq:est-term1}--\eqref{eq:est-term1-3} we infer
\begin{equation}\label{eq:est-ux^4}
\fint_{-L}^{L} \int_{-1}^1
\frac{(u^L_{,x})^4}{L^2} \dx\dy\lesssim \frac{L_0^2}{L^2}\left(\frac1\delta+\frac1{\eps^2}\right)\,.
\end{equation}
 We now pass to estimate the second term on the right hand side of \eqref{ub:est1}. To this aim we observe that integrating by parts it holds
\begin{equation}\label{ub:est5}
	\begin{split}
	w_{2,x}^L+u^L_{,x}u^L_{,y}&=
	(x\psi_\delta^2(x))'y
	+\dot B^L(x)-\int_0^yu^L_{,y}u^L_{,xy}\dy'+u^L_{,x}u^L_{,y}(x,y)\\
	&= 	(x\psi_\delta^2(x))'y
	+\dot B^L(x)+ \int_0^yu_{,yy}u_{,x}\dy' +u^L_{,x}u^L_{,y}|_{y=0}\\
	&=  	(x\psi_\delta^2(x))'y+ \int_0^yu_{,yy}u_{,x}\dy'+C^L(x)\,,
	\end{split}
\end{equation}
where 
$$ C^L(x):=- \fint_{-L}^{L} \int_0^yu_{,yy}u_{,x}\dy'\dy=- \fint_{-L_0}^{L_0} \int_0^yu_{,yy}u_{,x}\dy'\dy\,,$$
and  the last equality is a consequence of the following identity
\begin{equation*}
\dot B^L(x)+  u^L_{,x}u^L_{,y}|_{y=0}=
\fint_{-L}^{L}\biggl(
\int_0^yu^L_{,y}u^L_{,xy}\dy'-
u^L_{,x}u^L_{,y}+  u^L_{,x}u^L_{,y}|_{y=0}\biggr) \dy=C^L(x)\,.
\end{equation*}
Using \eqref{ub:est5} and the definition of $\omega$ we get
\begin{equation*}
\begin{split}
w^L_{1,x}-1& = -\frac{1}{L^2}\int_0^y\left( 	(x\psi_\delta^2(x))''y' + \int_0^{y'}(u^L_{,yy}u^L_{,xx}+ u^L_{,yyx}u^L_{,x})\dy''+\dot C^L(x)
\right)\dy'\\
& =
 -\frac{1}{L^2}\left( (x\psi_\delta^2(x))''\frac{y^2}{2} + \int_0^y\int_0^{y'}(u^L_{,yy}u^L_{,xx}+ u^L_{,yyx}u^L_{,x})\dy''\dy'+\dot C^L(x)y
\right)
\,.
\end{split}
\end{equation*}
This together with Young's inequality give 
\begin{equation}\label{ub:est6}
	\begin{split}
	L^2	\strokedint_{-L}^{L}\int_{-1}^1
		\bigl( w^L_{1,x}-1 \bigr)^2 \dx\dy&= L^2	\strokedint_{-L_0}^{L_0}\int_{-1}^1
		\bigl( w^L_{1,x}-1 \bigr)^2 \dx\dy\\ &
		\lesssim \frac 1{L^2}
			\strokedint_{-L_0}^{L_0}\int_{-1}^1	((x\psi_\delta^2(x))'')^2y^4\dx\dy\\
			& 
		+	\frac 1{L^2}
			\strokedint_{-L_0}^{L_0}\int_{-1}^1\biggl(\int_0^y \int_0^{y'} (u^L_{,yy}u^L_{,xx}+ u^L_{,yyx}u^L_{,x})\dy''
			\dy'\biggr)^2\dx\dy \\
			&+ 
			\frac 1{L^2}
			\strokedint_{-L_0}^{L_0}\int_{-1}^1(\dot C^L(x))^2y^2\dx\dy\,.
	\end{split}
\end{equation}
As  $(x\psi_\delta^2(x))''= 4\psi_\delta(x)\dot\psi_\delta(x)+2x\psi_\delta(x)\ddot\psi_\delta(x)+2x(\dot\psi_\delta(x))^2$ in $(\delta,2\delta)$ and $(x\psi_\delta^2(x))''=0$ otherwise, we have
\begin{equation}\label{ub:est7}
	\begin{split}
\frac 1 {L^2}	\strokedint_{-L_0}^{L_0}\int_{-1}^1((x\psi_\delta^2(x))'')^2
y^4\dx\dy \lesssim 
\frac{L_0^4}{L^2}	\int_{\delta}^{2\delta}((x\psi_\delta^2(x))'')^2
\dx 
\lesssim \frac{L_0^4}{L^2} \frac1{\delta}
 \,.
	\end{split}
\end{equation}
We now estimate the second term on the right hand-side of \eqref{ub:est6}. 
We first observe that if $a,b,c,d$ are $2L_0$-periodic then by applying in order   H\"older, Young and Jensen inequalities we have
\begin{equation*}
	\begin{split}
		\bigg[
		\int_0^y\int_0^{y'}&(ab+cd)\dy''\dy' 
		\bigg]^2
		\le \left[
		\int_0^y  \|a\|_{L^2(0,y')}\|b\|_{L^2(0,y')}
		+ \|c\|_{L^2(0,y')}\|d\|_{L^2(0,y')}
		\dy'
		\right]^2\\
		& \lesssim 
		\left[
		\int_0^y  \|a\|_{L^2(0,y')}\|b\|_{L^2(0,y')}
		\dy'
		\right]^2
		+ \left[
		\int_0^y 
		 \|c\|_{L^2(0,y')}\|d\|_{L^2(0,y')}
		\dy'
		\right]^2\\
		& \lesssim L_0^2 \fint_{-L_0}^{L_0} \|a\|^2_{L^2(0,y')}\|b\|^2_{L^2(0,y')}
		\dy'+ 
		L_0^2 \fint_{-L_0}^{L_0} \|c\|^2_{L^2(0,y')}\|d\|^2_{L^2(0,y')}
		\dy'\\
	&\lesssim
	L_0^4\left[
	\bigg(\fint_{-L_0}^{L_0}a^2\dy''\bigg)
	\bigg(\fint_{-L_0}^{L_0}b^2\dy''\bigg)+
	\bigg(\fint_{-L_0}^{L_0}c^2\dy''\bigg)
	\bigg(\fint_{-L_0}^{L_0}d^2\dy''\bigg)
	\right]
		\,.	\end{split}
\end{equation*}
Therefore it follows that 
\begin{equation}\label{ub:est0}
	\begin{split}
		\frac 1{L^2}
		\strokedint_{-L_0}^{L_0}\int_{-1}^1\biggl(\int_0^y &\int_0^{y'} (u^L_{,yy}u^L_{,xx}+ u^L_{,yyx}u^L_{,x})\dy''
		\dy'\biggr)^2\dx\dy \\
		&=	\frac 1{L^2}
		\strokedint_{-L_0}^{L_0}\int_{\delta}^1\biggl(\int_0^y \int_0^{y'} (u^L_{,yy}u^L_{,xx}+ u^L_{,yyx}u^L_{,x})\dy''
		\dy'\biggr)^2\dx\dy\\
		&\lesssim  \frac {L_0^4}{L^2}
	\int_{\delta}^1 \bigg(\fint_{-L_0}^{L_0}(u^L_{,yy} )^2\dy''\bigg)\bigg(\fint_{-L_0}^{L_0}(u^L_{,xx} )^2\dy''\bigg)
		\dx \\
		&+  \frac {L_0^4}{L^2}
		\int_{\delta}^1
	\bigg(\fint_{-L_0}^{L_0}(u^L_{,yyx} )^2\dy''\bigg)\bigg(\fint_{-L_0}^{L_0}(u^L_{,x} )^2\dy''\bigg)
		\dx
		 \,.
	\end{split}
\end{equation}
Next we  show separately that:
\begin{equation}\label{uyy+ux}
	\fint_{-L_0}^{L_0}(u^L_{,yy})^2\dy''\lesssim \frac1\eps\,,\qquad \qquad
	\fint_{-L_0}^{L_0}(u^L_{,x})^2\dy''\lesssim\frac1{x\eps}\max\{x,\eps\}\,,
\end{equation}
\begin{equation}\label{uxx+uyyx}
\int_{\delta}^{1}\fint_{-L_0}^{L_0}(u^L_{,xx})^2\dy''\dx\lesssim \frac1{\delta^2}\,,
\qquad\qquad
\int_{\delta}^{1}\fint_{-L_0}^{L_0}(u^L_{,yyx})^2\dy''\dx\lesssim\frac1{\delta\eps} \,.
\end{equation}
 %
Since $u^L_{,yy}=\psi_\delta(x)f^L(x)\hat u^L_{,yy}$   by \eqref{eq:f(x)} and \eqref{eq:uyy+ux} we have
 \begin{equation*}
 	\label{ub:est-yy}
 \fint_{-L_0}^{L_0}(\hat u_{,yy}^L)^2\dy \lesssim \frac{1}{\eps}\,.
 \end{equation*}
By \eqref{der:ux}, \eqref{eq:f(x)}, \eqref{eq:uyy+ux}, \eqref{eq:est-u} and \eqref{eq:f'(x)-f''(x)} we have
\begin{equation*}\label{ub:est-x}
	\begin{split}
	\fint_{-L_0}^{L_0}(u^L_{,x})^2\dy
		\lesssim  \fint_{-L_0}^{L_0}(f^L(x))^2(\hat u^L_{,x})^2\dy
	&	+ \frac1{\delta^2}\fint_{-L_0}^{L_0}\chi_{(\delta,2\delta)}(f^L(x))^2(\hat u^L)^2\dy
		+ \fint_{-L_0}^{L_0}(\dot f^L(x))^2(\hat u^L)^2\dy\\
		&\lesssim \frac1\eps+ \frac1{\delta^2}\frac\delta\eps\max\{x,\eps\}\chi_{(\delta,2\delta)}+ \frac1{x\eps}\max\{x,\eps\}\\
	&	\lesssim \frac1\eps +\frac{1}{\delta}\chi_{(\delta,2\delta)}+ \frac1{x\eps}\max\{x,\eps\}\lesssim \frac1{x\eps}\max\{x,\eps\}
		\,.
	\end{split}
\end{equation*}
From  \eqref{der:ux} it follows 
\begin{equation*}
	\begin{split}
		u^L_{,xx}(x,y)&=\psi_\delta(x)f^L(x)\hat{u}^L_{,xx}(x,y)+\psi_\delta(x)\ddot f^L(x)\hat{u}^L(x,y)+	\ddot\psi_\delta(x)f^L(x)\hat{u}^L(x,y) \\
		&+ 2\dot\psi_\delta(x)f^L(x)\hat{u}^L_{,x}(x,y)+ 2\dot\psi_\delta(x)\dot f^L(x)\hat{u}^L(x,y)+ 2\psi_\delta(x)\dot f^L(x)\hat{u}^L_{,x}(x,y)\,.
	\end{split}
\end{equation*}
Hence by \eqref{eq:f(x)}, \eqref{eq:uxx-uxy-uxyy}, \eqref{eq:f'(x)-f''(x)}, \eqref{eq:est-u} and \eqref{eq:limsup} we have
\begin{equation}\label{ub:est-xx}
	\begin{split}
	\int_{\delta}^1\fint_{-L_0}^{L_0}(u_{,xx}^L)^2\dy\dx&
		\lesssim \int_{\delta}^1\fint_{-L_0}^{L_0}(f^L(x))^2(\hat u_{,xx}^L)^2\dy\dx+ \int_\delta^1\fint_{-L_0}^{L_0}(\ddot f^L(x))^2(\hat u^L)^2\dy\dx\\
		& 	+ \int_{\delta}^{1}\fint_{-L_0}^{L_0}(\dot f^L(x))^2(\hat u_{,x}^L)^2\dy\dx+
	 \frac{1}{\delta^4}\int_\delta^{2\delta}\fint_{-L_0}^{L_0}(f^L(x))^2(\hat u^L)^2\dy\dx\\
		& 
		+ \frac{1}{\delta^2}\int_{\delta}^{2\delta}\fint_{-L_0}^{L_0}(\dot f^L(x))^2(\hat u^L)^2\dy\dx
		+\frac{1}{\delta^2}\int_{\delta}^{2\delta}\fint_{-L_0}^{L_0}(f^L(x))^2(\hat u_{,x}^L)^2\dy\dx\\
		& \lesssim \frac1{\eps^2} +  \int_\delta^1\frac{1}{x^3\eps}\max\{x,\eps\}\dx+ \frac{1}{\delta\eps}
+ \frac1{\delta^4}\int_\delta^{2\delta}\frac x\eps\max\{x,\eps\}\dx\\
&
+ \frac{1}{\delta^2}\int_\delta^{2\delta}\frac1{x\eps}\max\{x,\eps\}\dx+\frac1{\delta^2}\lesssim 
\frac1{\eps^2}+ \frac{1}{\delta^2} + \frac1{\delta\eps}\lesssim \frac1{\delta^2}
	\,.
	\end{split}
\end{equation}
Analogously by \eqref{der:ux} it follows
\begin{equation*}\label{der:uxyy}
	u^L_{,yyx}(x,y)= \psi_\delta(x)f^L(x)\hat u^L_{,yyx}(x,y)+ \dot\psi_\delta(x)f^L(x)\hat u^L_{,yy}(x,y)+
	\psi_\delta(x)\dot f^L(x)\hat u^L_{,yy}(x,y)\,,
\end{equation*}
from which together with \eqref{eq:uxx-uxy-uxyy}, \eqref{eq:f(x)}, \eqref{eq:uyy+ux} and \eqref{eq:f'(x)-f''(x)}
\begin{equation*}\label{ub:est-yyx}
	\begin{split}
\int_\delta^1\fint_{-L_0}^{L_0}(u_{,yyx}^L)^2\dy\dx&\lesssim
\int_\delta^1 \fint_{-L_0}^{L_0}(f^L(x))^2(\hat u_{,yyx}^L)^2\dy\dx
+ \frac{1}{\delta^2} \int_\delta^{2\delta}\fint_{-L_0}^{L_0}(f^L(x))^2(\hat u_{,yy}^L)^2\dy\dx\\&
+ \int_\delta^1\fint_{-L_0}^{L_0}(\dot f^L(x))^2(\hat u_{,yy}^L)^2\dy\dx
\lesssim\frac1{\eps^2} + \frac1{\delta^2}\frac\delta\eps\frac{\delta}{\eps}+ \frac1{\delta\eps} \lesssim\frac1{\delta\eps}\,.
	\end{split} 
\end{equation*}
Now \eqref{uyy+ux} and \eqref{uxx+uyyx} yield
\begin{equation}\label{ub:est-I}
\begin{split}
	\frac {L_0^4}{L^2}
\int_{\delta}^1\bigg(\fint_{-L_0}^{L_0}(u^L_{,yy} )^2\dy''\bigg)\bigg(\fint_{-L_0}^{L_0}(u^L_{,xx} )^2\dy''\bigg)
	\dx
\lesssim 	\frac {L_0^4}{L^2}\frac{1}{\delta^2\eps}
\,,
\end{split}
\end{equation}
and
\begin{equation}\label{ub:est-II}
\begin{split}
 \frac {L_0^4}{L^2}
\int_{\delta}^1\bigg(\fint_{-L_0}^{L_0}(u^L_{,yyx} )^2\dy''\bigg)\bigg(\fint_{-L_0}^{L_0}(u^L_{,x} )^2\dy''\bigg)
\dx  \lesssim  \frac {L_0^4}{L^2} \left(\frac{1}{\delta}+\frac1\eps\right)\frac1{\delta\eps}\lesssim \frac {L_0^4}{L^2} \frac1{\delta^2\eps}
\,.
\end{split}
\end{equation}
Gathering together  \eqref{ub:est0}, \eqref{ub:est-I} and \eqref{ub:est-II} we find
\begin{equation}\label{ub:estIII}
\frac 1{L^2}
\strokedint_{-L_0}^{L_0}\int_{-1}^1\biggl(\int_0^y \int_0^{y'} (u^L_{,yy}u^L_{,xx}+ u^L_{,yyx}u^L_{,x})\dy''
\dy'\biggr)^2\dx\dy \lesssim \frac {L_0^4}{L^2} \frac1{\delta^2\eps}\,.
\end{equation}
It remains to estimate the third term on the right hand-side of \eqref{ub:est6}.
In a similar way, see in particular \eqref{ub:est-I} and \eqref{ub:est-II}, we have
\begin{equation}\label{ub:estIV}
\begin{split}
	\frac 1{L^2}
\strokedint_{-L_0}^{L_0}\int_{-1}^1(\dot C^L(x))^2y^2\dx\dy&\lesssim \frac{L_0^2}{L^2}\int_{-1}^1 \bigg(\fint_{-L_0}^{L_0}\int_0^y
(u^L_{,yy}u^L_{,xx}+ u^L_{,yyx}u^L_{,x})
\dy'\dy
\bigg)^2\dx\\
&\lesssim \frac{L_0^4}{L^2}
\int_{-1}^1\fint_{-L_0}^{L_0}  \bigg(\fint_{-L_0}^{L_0}(u^L_{,yy} )^2\dy''\bigg)\bigg(\fint_{-L_0}^{L_0}(u^L_{,xx} )^2\dy''\bigg) \dy\dx\\
&+ \frac{L_0^4}{L^2}
\int_{-1}^1\fint_{-L_0}^{L_0}  \bigg(\fint_{-L_0}^{L_0}(u^L_{,yyx} )^2\dy'' \bigg)\bigg(\fint_{-L_0}^{L_0}(u^L_{,x} )^2\dy''\bigg) \dy\dx\\
&\lesssim \frac {L_0^4}{L^2} \frac1{\delta^2\eps}\,.
\end{split}
\end{equation}
Gathering together \eqref{ub:est6}, \eqref{ub:est7}, \eqref{ub:estIII} and \eqref{ub:estIV} we infer
\begin{equation*}
	L^2	\strokedint_{-L}^{L}\int_{-1}^1
\bigl( w^L_{1,x}-1 \bigr)^2 \dx\dy\lesssim  \frac {L_0^4}{L^2} \frac1{\delta^2\eps}\,,
\end{equation*}
which together with \eqref{eq:est-ux^4} and \eqref{ub:est1} implies 
\begin{equation*}
	\begin{split}
		L^2 \strokedint_{-L}^{L}
		\int_{-1}^1
		\biggl( w^L_{1,x}+\frac{(u^L_{,x})^2}{2L^2}-1 \biggr)^2 \dx\dy
		\lesssim \frac{L_0^2}{L^2}\left(\frac1\delta+\frac1{\eps^2}\right)+
		 \frac {L_0^4}{L^2} \frac1{\delta^2\eps}\lesssim  \frac {L_0^4}{L^2} \frac1{\delta^2\eps}
	\,.
	\end{split}
\end{equation*}

\noindent
\textit{Step 4:} we show that
\begin{equation}\label{ub:step4}
L^2\strokedint_{-L}^{L} \int_{-1}^1 \Big ( w^L_{2,y}+\frac{(u_{,y}^L)^2}{2} -x \Big)^2 \dx\dy \le L^2 \left(\frac13+C\delta^3\right)\,.
\end{equation}
Recalling the definition of $w_2^L$ it holds
$
	w^L_{2,y}+\frac{(u_{,y}^L)^2}{2} -x = \psi_\delta^2(x)x-x 
$, so that 
\begin{equation*}
\begin{split}
\strokedint_{-L}^{L} \int_{-1}^1 \Big ( w^L_{2,y}+\frac{(u_{,y}^L)^2}{2} -x \Big)^2 \dx\dy= \int_{-1}^{1}(
 \psi_\delta^2(x)x-x  )^2\dx\le \int_{-1}^{2\delta} x^2\dx =\frac83\delta^3+
 \frac13\,.
\end{split}
\end{equation*}

\noindent
\textit{Step 5:} we show that
\begin{equation}\label{ub:step5}
\strokedint_{-L}^{L} \int_{-1}^1  \Big( L^2 w_{1,y}^L+w_{2,x}^L+u_{,x}^Lu_{,y}^L \Big)^2  \dx\dy=0\,.
\end{equation}
This is a trivial consequence of the following identity
\begin{equation*}
w^L_{1,y}=-\frac1{L^2}(w_{2,x}^L+u_{,x}^Lu_{,y}^L)\,,
\end{equation*}
which follows from the definitions of $w_1^L$ and $w_2^L$.
\medskip

\noindent
\textit{Step 6:} we show that
\begin{equation}\label{ub:step6}
\frac{1}{L^2}
\strokedint_{-L}^{L} \int_{-1}^1  \bigg(2(u^L_{,xy})^2+\frac{1}{L^2}(u^L_{,xx})^2\bigg)
\dx\dy\lesssim \left(\frac1{L^2}+\frac1{L^4}\right)\frac1{\delta^2}\,.
\end{equation}
We have 
\begin{equation*}
	u^L_{,xy}(x,y)= \psi_\delta(x)f^L(x)\hat u^L_{,xy}(x,y)+ \dot\psi_\delta(x)f^L(x)\hat u^L_{,y}(x,y)+
	\psi_\delta(x)\dot f^L(x)\hat u^L_{,y}(x,y)\,.
\end{equation*}
Thus by \eqref{eq:AL}, \eqref{eq:f(x)}, \eqref{eq:uxx-uxy-uxyy}, \eqref{eq:limsup} and \eqref{eq:f'(x)-f''(x)} 
\begin{equation*}\label{eq:est-xy}
	\begin{split}
		\strokedint_{-L}^{L} \int_{-1}^1  (u^L_{,xy})^2 
	\dx\dy&\lesssim
	\strokedint_{-L}^{L} \int_{\delta}^1 (f^L(x))^2 (\hat u^L_{,xy})^2 
	\dx\dy\\
	&+ \frac1{\delta^2}	\strokedint_{-L}^{L} \int_{\delta}^{2\delta} (f^L(x))^2 (\hat u^L_{,y})^2 
	\dx\dy\\&+ \strokedint_{-L}^{L} \int_{\delta}^{1} (\dot f^L(x))^2 (\hat u^L_{,y})^2 
	\dx\dy
	 \lesssim \frac1{\eps^2}+ 1+ \frac1{\eps}+\frac\eps\delta\lesssim\frac1{\delta\eps}\,.
	\end{split}
\end{equation*}
This together with \eqref{ub:est-xx} imply \eqref{ub:step6}.\\

\noindent
\textit{Conclusions.}  By Step 1 we have that $(\omega^L,u^L)$ converges to $\mu$ in the sense of Definition \ref{def:convergence}.
Moreover  by collecting the estimates showed in Steps 2--6, i.e.,
 \eqref{ub:step2}, \eqref{ub:step3}, \eqref{ub:step4}, \eqref{ub:step5} and \eqref{ub:step6} we find
\begin{equation}\label{the-end0}
	\begin{split}
\limsup_{L\to+\infty}L^2(R_L(w^L,u^L)-\mathcal E_0)&\le \mathcal{F}_\infty(\mu)+ C\lim_{L\to +\infty}
\omega(2M^2\delta)\log M\\
&+C\lim_{L\to +\infty} \Big(\frac{L_0^4}{L^2}\frac1{\delta^3M}+ L^2\delta^3+ \frac{1}{L^2\delta^2}
\Big)\,.
	\end{split}
\end{equation}
We now proceed with the choice of the parameters. We start by noticing  that for every $\overline M\in\mathbb N$ there exists $L_{\overline M}\ge\overline M^{1/2}$ such that
\begin{equation}\label{cond:M}
 \omega(2\overline M^2L^{-2/3})\le {\overline M}^{-1}\quad\forall L\ge L_{\overline M}\,.
\end{equation}
Since $L_{\overline M+1}\ge L_{\overline M}$ we set
$$ M=M(L):=\overline M\quad \text{ if } L\in[L_{\overline M},L_{\overline M+1})$$
Next we define
\begin{equation*}
	\delta:=\frac{\eps}{M}= L^{-2/3}M^{-1/8}\,\iff\, \eps= L^{-2/3}M^{7/8}\,,
\end{equation*}
and we choose 
\begin{equation*}
	L_0:=\frac Ln\in [M^{1/8},2M^{1/8})\,\iff\, n\in \Big[\frac{L}{2M^{1/8}},  \frac L{M^{1/8}} \Big)\,.
\end{equation*}
These choices ensures the validity of \eqref{cond_parameters1} and \eqref{cond_parameters2}. Indeed we have 
\begin{equation*}
	\eps= \frac1{L^{2/3}M^{-7/8}}= \frac1{L^{2/3}\overline M^{-7/8}}\le 
	\frac1{\overline M^{3}\overline M^{-7/8}}
	\quad\text{ if }L\in[L_{\overline M},L_{\overline M+1})\,,
\end{equation*}
where the last inequality follows from the fact that $L\ge L_{\overline M}\ge \overline M^{1/2}$. Hence $\eps\to0$ and $\delta\to0$ as $L\to+\infty$. In a similar way we have $M^2\delta\to0$ and $L_0,n\to+\infty$ as $L\to+\infty$.\\
Recalling that $\omega$ is monotone we find
$$ \omega(2M^2\delta)=\omega(2M^2L^{-2/3}M^{-1/8})\le \omega(2M^2L^{-2/3})= \omega(2\overline M^2L^{-2/3}) \quad\text{ if }L\in[L_{\overline M},L_{\overline M+1})\,,$$ 
which together with \eqref{cond:M} imply 
\begin{equation}\label{the-end1}	
	\log M\omega(2M^2\delta)\le \log( \overline M)\overline M^{-1}\quad \text{ if } L\in[L_{\overline M},L_{\overline M+1})\,.
\end{equation}
Moreover  if  $L\in[L_{\overline M},L_{\overline M+1})$ it holds
\begin{equation}\label{the-end2}
L^2\delta^3= L^2L^{-2}M^{-3/8}=\overline M^{-3/8}\,,
\end{equation}
 \begin{equation}\label{the-end3}
\frac1{L^2\delta^2}=\frac{L^{4/3}M^{1/4}}{L^2}=\frac{ M^{1/4}}{L^{2/3}} =
\frac{\overline M^{1/4}}{L^{2/3}} \le
\frac{\overline M^{1/4}}{(\overline M^{1/2})^{2/3}} =\overline M^{-1/12}\,,
 \end{equation}
and
\begin{equation}\label{the-end4}
 \frac{L_0^4}{L^2}\frac1{\delta^3M}  \le \frac{16M^{1/2}}{M M^{-3/8}}=16M^{-1/8}= 16\overline M^{-1/8} \,,
\end{equation}
 Eventually collecting \eqref{the-end0}--\eqref{the-end4} we infer 

\begin{equation*}
	\begin{split}
		\limsup_{L\to+\infty}L^2(R_L(w^L,u^L)-\mathcal E_0)\le \mathcal{F}_\infty(\mu)\,.
	\end{split}
\end{equation*}
\end{proof}
\section{Existence and regularity of minimizers of $\mathcal{F}_\infty$}\label{sec:regularity}
In this section we address the existence of minimizers of the limiting functional $\mathcal{F}_\infty$ and we discuss some properties such as equipartition of the energy.
In order to do that we need to introduce the definition of disintegration of measures in the $k$-variable, which is slightly different from the disintegration in the $x$-variable introduced in Section \ref{sec:preliminaries}.\\

\noindent
In the following for a given interval $I\subset\R$ we denote by $L^0(I)$ the space of functions $g\colon I\to \R$ that are Lebesgue measurable. Moreover the map $\pi_2\colon I\times\R\to \R$ denotes the canonical projection, and for any $\mu\in\mathcal{M}_b(I\times\R)$ we indicate by $(\pi_2)_\sharp\mu\in \mathcal{M}_b^+(\R)$ its push-forward with respect to the map $\pi_2$.

\begin{definition}[Disintegration of measures in the $k$-variable]\label{def:disintegration-in-k}
 Let $I\subset\R$ be an interval and let $\mu\in\mathcal{M}_b(I\times\R)$. We say that the family
\begin{equation*}
	{\big(\lambda\,,\,(g_k)_{k\in\R}\big)}\quad\text{ with } \quad\lambda\in \mathcal{M}_b(\R)\quad\text{ and } \quad g_k\in L^0(I)\quad \forall k\in \R\,,
\end{equation*}
is a disintegration of $\mu$ $($in the $k$-variable$)$ if $k\mapsto g_k$ is $\lambda$-measurable, $\int_0^1g_k\dx=1$ for $\lambda$-a.e. $k\in\R$ and
\begin{equation}\label{eq:disint-in-k}
\int_{I\times\R}f(x,k)\di\mu=\int_\R\int_If(x,k)g_k(x)\dx\di\lambda(k)\,,
\end{equation}
for every $f\in L^1(I\times\R;|\mu|)$.
\end{definition} 
With this definition at hand we can state the main result of this section.
\begin{theorem}[Minimizers of $\mathcal{F}_\infty$]\label{thm:minimizers}
	Let $\mathcal{M}_\infty$ and $\mathcal{F}_\infty$ be as in \eqref{def:limit_measures} and \eqref{def:F_infty} respectively. Then there exists $\hat\mu\in\mathcal{M}_\infty$ such that
	\begin{equation*}
	\mathcal{F}_\infty(\hat\mu)=\inf_{\mu\in\mathcal{M}_\infty}\mathcal{F}_\infty(\mu)\,.
	\end{equation*}
Moreover, every minimizer $\hat\mu$ satisfies the following properties:
 there exist a constant $C>0$ and a $(\pi_2)_\sharp\hat\mu$-measurable map $k\mapsto g_k$  with $g_k\in BV(0,1)$ for $(\pi_2)_\sharp\hat\mu$ a.e. $k\in\R$, such that 
 $${\big((\pi_2)_\sharp\hat\mu\,,\, (g_k)_{k\in\R}\big)}\quad \text{ is a disintegration of }\ \hat \mu\,,$$
\begin{equation}\label{equipart-section}
	\int_0^1k^2g_k(x)\dx=\int_0^1\frac1{4k^2}\left(\frac{\di\hat\mu_{,x}}{\di\hat\mu}\right)^2g_k(x)\dx\quad \text{for } \ (\pi_2)_\sharp\hat\mu\ \text{ a.e. $k\in\R$}\,,
\end{equation}
and 
\begin{equation}\label{small-k}
(\pi_2)_\sharp\hat\mu\Big(\big\{|k|<C\big\}\Big)=0\,.
\end{equation}

\end{theorem}
As a direct consequence minimizers of $\mathcal{F}_\infty$ satisfy equipartition of the energy.
\begin{cor}[Equipartition of the energy]
	Let $\hat\mu\in\mathcal{M}_\infty$ be a minimizers of $\mathcal{F}_\infty$. Then it holds
	\begin{equation*}
	\int_{(0,1)\times\R}k^2\di\hat\mu= 	\int_{(0,1)\times\R} \frac1{4k^2}\Big(\frac{\di\hat\mu_{,x}}{\di\hat\mu}\Big)^2  \di\hat\mu\,.
	\end{equation*}
\end{cor}
\begin{proof}
	By Theorem \ref{thm:minimizers} it holds
	\begin{equation*}
\begin{split}
	\int_{(0,1)\times\R}k^2\di\hat\mu&=  \int_\R 	\int_0^1k^2g_k(x)\dx\di(\pi_2)_\sharp\hat\mu\\
	&=\int_\R \int_0^1\frac1{4k^2}\left(\frac{\di\hat\mu_{,x}}{\di\hat\mu}\right)^2g_k(x)\dx \di(\pi_2)_\sharp\hat\mu
	=\int_{(0,1)\times\R} \frac1{4k^2}\Big(\frac{\di\hat\mu_{,x}}{\di\hat\mu}\Big)^2  \di\hat\mu\,.
\end{split}
	\end{equation*}
\end{proof}
We divide the proof of Theorem \ref{thm:minimizers} into several steps. Precisely we need to show that the functional $\mathcal F_\infty$ is convex and lower semi-continuous and that the class of measures $\mathcal M_\infty$ admits a disintegration in the $k$-variable of the form ${\big((\pi_2)_\sharp\hat\mu\,,\, (g_k)_{k\in\R}\big)}$.
First of all we recall that by Remark \ref{rem:limit} \ref{rem-ii} we have
	\begin{equation*}
	\mathcal F_\infty(\mu)= \int_{(0,1)\times\R} 
	k^2 \di\mu+   \int_{(0,1)\times\R} 
	\frac{1}{4k^2}\Bigl(\frac{\di \mu}{\di|\tilde\mu|}\Bigr)^{-1}\Bigl(\frac{\di \mu_{,x}}{\di|\tilde\mu|}\Bigr)^2
	\di|\tilde \mu|\,,
\end{equation*}
with $\tilde\mu=(\mu,\mu_{,x})$ and $|\tilde\mu|$ its total variation. This alternative formulation turns out to be more convenient, in particular the term 
\begin{equation*}
\int_{(0,1)\times\R} 
\frac{1}{4k^2}\Bigl(\frac{\di \mu}{\di|\tilde\mu|}\Bigr)^{-1}\Bigl(\frac{\di \mu_{,x}}{\di|\tilde\mu|}\Bigr)^2
\di|\tilde \mu|
\end{equation*}
is reminiscent of the Benamou-Brenier functional used in optimal transport which enjoys nice properties such as lower semicontinuity and convexity. Here  we consider a specific case and we refer to \cite[Section 5.3.1]{Sa15} for a general treatment of this topic.  

\medskip

\noindent
For any $\rho,E\in\mathcal{M}_b((0,1)\times\R)$ the Benamou-Brenier functional is defined as
\begin{equation}\label{def:BB}
\mathscr{B}_2(\rho,E):=\sup\left\{\int_{(0,1)\times\R}a(x,k)\di\rho+\int_{(0,1)\times\R}b(x,k)\di E\colon (a,b)\in C_b((0,1)\times\R;K_2)
\right\}\,,
\end{equation}
where 
\begin{equation*}
K_2:=\left\{(z_1,z_2)\in\R^2\colon z_1+\frac12z_2^2\le0\right\}\,.
\end{equation*}
We next recall some properties of $\mathscr{B}_2$ which follow from \cite[Proposition 5.18]{Sa15}. Then we state and prove two intermediate Lemmas (cf. Lemma \ref{lem:prop-F} and Lemma \ref{lem-disint-k}) which will be used to show the validity of Theorem \ref{thm:minimizers}.
\begin{prop}[Properties of $\mathscr{B}_2$]\label{prop:BB}
	The functional $\mathscr{B}_2$ is convex and lower semi-continuous on the space $(\mathcal{M}_b((0,1)\times\R))^2$. Moreover, the following property hold: if both $\rho$ and $E$ are absolutely continuous w.r.t. a same positive measure $\lambda$ on $(0,1)\times\R$, then
	\begin{equation*}
		\mathscr{B}_2(\rho,E)=\int_{(0,1)\times\R}\frac12\Big(\frac{\di\rho}{\di\lambda}\Big)^{-1} \Big(\frac{\di E}{\di\lambda}\Big)^2\di\lambda\,.
	\end{equation*}
\end{prop}
\begin{lemma}[Properties of $\mathcal F_\infty$]\label{lem:prop-F}
	The functional $\mathcal{F}_\infty$ is 1-homogeneous, convex and lower semi-continuous on the space $\mathcal{M}_b((0,1)\times\R)$. Moreover let $(\mu_j)\subset\mathcal M_\infty$ be a minimizing sequence, i.e.,
	\begin{equation*}
		\mathcal{F}_\infty(\mu_j)\to \inf_{\mu \in \mathcal M_\infty} 	\mathcal{F}_\infty(\mu)\,.
	\end{equation*}
Then $(\mu_j)$ is pre-compact in $\mathcal M_\infty$, i.e., there exists $\hat\mu\in\mathcal{M}_\infty$ such that, up to subsequence, $\mu_j\stackrel{*}{\rightharpoonup}\hat\mu$. Thus, in particular, 
\begin{equation*}
\mathcal{F}_\infty(\hat\mu)=\inf_{\mu \in \mathcal M_\infty} 	\mathcal{F}_\infty(\mu)\,.
\end{equation*}
\end{lemma}
\begin{proof}\textit{1-homogeneity.} Let $\alpha>0$ and let $\mu\in\mathcal{M}_\infty$. Then a direct computation shows that 
	\begin{equation*}
		\frac{\di(\alpha\mu_{,x})}{\di(\alpha\mu)}=\frac{\di\mu_{,x}}{\di\mu}\,,
	\end{equation*}
	from which we readily deduce $\mathcal{F}_\infty(\alpha\mu)=\alpha\mathcal{F}_\infty(\mu)$.
	
	\noindent
\textit{Convexity.} 
Let $\mu_1,\mu_2\in\mathcal M_\infty$, $t\in(0,1)$. Assume that $\mathcal{F}_\infty(\mu_1),\mathcal{F}_\infty(\mu_1)<+\infty$, otherwise there is nothing to prove.
Clearly $\mu_3=t\mu_1+(1-t)\mu_2\in\mathcal M_\infty$ and
\begin{equation*}
	\mathcal F_\infty(\mu_3)= \int_{(0,1)\times\R} 
	k^2 \di\mu_3+   \int_{(0,1)\times\R} 
	\frac{1}{4k^2}\Bigl(\frac{\di \mu_3}{\di|\tilde\mu_3|}\Bigr)^{-1}\Bigl(\frac{\di \mu_{3,x}}{\di|\tilde\mu_3|}\Bigr)^2
	\di|\tilde \mu_3|\,.
\end{equation*}
We have
\begin{equation}\label{conv1}
\int_{(0,1)\times\R} 
k^2 \di\mu_3= t\int_{(0,1)\times\R} 
k^2 \di\mu_1+ (1-t)\int_{(0,1)\times\R} 
k^2 \di\mu_2\,.
\end{equation}
Next for $i=1,2$ set $\rho_i:=\mu_i$, $E_i:=\frac1{\sqrt2k}{\mu_{1,x}}$,  $\lambda_i=|\tilde\mu_i|$, and note that they belong to $\mathcal{M}_b((0,1)\times\R)$. Indeed  $\rho_i,\lambda_i$ are bounded by definition, whereas by Young inequality, it holds 
\begin{equation*}
	\begin{split}
	E_i((0,1)\times\R)&=\int_{(0,1)\times\R}\frac1{\sqrt2k}\di\mu_{i,x}=\int_{(0,1)\times\R}\frac1{\sqrt2k}\Big(\frac{\di\mu_{i,x}}{\di\mu_i}\Big)\di\mu_{i}\\
	&	\le \int_{(0,1)\times\R}\frac1{4k^2}\Big(\frac{\di\mu_{i,x}}{\di\mu_i}\Big)^2\di\mu_{i}
+ \frac12\int_{(0,1)\times\R}\di\mu_i\le \mathcal{F}_\infty(\mu_i)+\mu_i((0,1)\times\R)<+\infty\,.
	\end{split}
\end{equation*}
Thus we can invoke Proposition \ref{prop:BB} and get
\begin{equation}\label{conv2}
\begin{split}
	 \int_{(0,1)\times\R} &
	\frac{1}{4k^2}\Bigl(\frac{\di \mu_3}{\di|\tilde\mu_3|}\Bigr)^{-1}\Bigl(\frac{\di \mu_{3,x}}{\di|\tilde\mu_3|}\Bigr)^2
	\di|\tilde \mu_3|
	= \mathscr{B}_2\left(\mu_3,\frac1{\sqrt2k}{\mu_{3,x}} \right)\\
	&\le t\mathscr{B}_2\left(\mu_1,\frac1{\sqrt2k}{\mu_{1,x}} \right)+
	(1-t)\mathscr{B}_2\left(\mu_2,\frac1{\sqrt2k}{\mu_{2,x}} \right)\\
	&\le t \int_{(0,1)\times\R} 
	\frac{1}{4k^2}\Bigl(\frac{\di \mu_1}{\di|\tilde\mu_1|}\Bigr)^{-1}\Bigl(\frac{\di \mu_{1,x}}{\di|\tilde\mu_1|}\Bigr)^2
	\di|\tilde \mu_1|+(1-t)\int_{(0,1)\times\R} 
	\frac{1}{4k^2}\Bigl(\frac{\di \mu_2}{\di|\tilde\mu_2|}\Bigr)^{-1}\Bigl(\frac{\di \mu_{2,x}}{\di|\tilde\mu_2|}\Bigr)^2
	\di|\tilde \mu_2|\,.
\end{split}
\end{equation}
Finally combining \eqref{conv1} with \eqref{conv2} we get
\begin{equation*}
\mathcal{F}_\infty(\mu_3)\le t \mathcal{F}_\infty(\mu_1)+(1-t)\mathcal{F}_\infty(\mu_2)\,.
\end{equation*}

\noindent
\textit{Lower semi-continuity.} Let $(\mu_j)\subset\mathcal{M}_\infty$ be such that $\mu_j\stackrel{*}{\rightharpoonup}\mu$ for some $\mu\in\mathcal{M}_\infty$. 
Let $\varphi\in C^\infty_c((0,1)\times\R)$, then by duality we have
\begin{equation*}
\lim_{j\to+\infty} \int_{(0,1)\times\R}\varphi \di\mu_{j,x}=-\lim_{j\to+\infty} \int_{(0,1)\times\R}\varphi_{,x} \di\mu_{j}=
- \int_{(0,1)\times\R}\varphi_{,x} \di\mu=- \int_{(0,1)\times\R}\varphi \di\mu_{,x}\,,
\end{equation*}
	so that $\mu_{j,x}\stackrel{*}{\rightharpoonup}\mu_{,x}$. Then by applying Reshetnyak Theorem \cite[Theorem 2.38]{AmFuPa:00} exactly as in \eqref{eq:Resh1} and \eqref{eq:Resh2} we deduce
	\begin{equation*}
		\liminf_{j\to+\infty}\mathcal{F}_\infty(\mu_j)\ge  \mathcal{F}_\infty(\mu)\,.
	\end{equation*}

	\noindent
	\textit{Compactness.} Let $(\mu_j)\subset\mathcal M_\infty$ be a minimizing sequence for $\mathcal{F}_\infty$. Then by Corollary \ref{cor:disint=pushfwd} for every $j$ there exists $x\nu_{j,x}$ with $\nu_{j,x}$ probability measure on $\R$ such that
	\begin{equation*}
|\mu_j|((0,1)\times\R)=\int_{(0,1)\times\R}\di\mu_j=\int_0^1\int_\R\di\nu_{j,x}2x\dx=\int_0^12x\dx=1\,.
	\end{equation*}
	Thus, up to subsequence, $\mu_j\stackrel{*}{\rightharpoonup}\hat\mu$. Moreover
	 by Young's inequality and \cite[Proposition 1.23]{AmFuPa:00} we have
	\begin{equation*}
C\ge\mathcal{F}_\infty(\mu_j)\ge\int_{(0,1)\times \R}\left|\frac{\di\mu_{j,x}}{\di\mu_j}\right|\di\mu_j=|\mu_{j,x}|((0,1)\times\R).
	\end{equation*}
Thus, up to subsequence and 
arguing as in the proof of Proposition \ref{prop:compactness}, we can deduce that $\mu_{j,x}\stackrel{*}{\rightharpoonup}\hat\mu_{,x}$ and $\hat\mu\in \mathcal{M}_\infty$.

\noindent \textit{Minimality of $\hat \mu$.} By lower semi-continuity and compactness we have
\begin{equation*}
\inf_{\mu\in\mathcal{M}_\infty}\mathcal{F}_\infty(\mu)= \lim_{j\to+\infty}\mathcal{F}_\infty(\mu_j)\ge \mathcal{F}_\infty(\hat\mu)\ge \inf_{\mu\in\mathcal{M}_\infty}\mathcal{F}_\infty(\mu)\,,
\end{equation*}
	so that 
	\begin{equation*}
		\mathcal{F}_\infty(\hat\mu)=\inf_{\mu \in \mathcal M_\infty} 	\mathcal{F}_\infty(\mu)\,.
	\end{equation*}
	
\end{proof}
\begin{lemma}[Disintegration of $\mu\in\mathcal{M}_\infty$ in the $k$-variable]\label{lem-disint-k}
	Let $\mu\in\mathcal{M}_\infty$ with $\mathcal{F}_\infty(\mu)<+\infty$. Then there exists $k\mapsto g_k$ $(\pi_2)_\sharp\mu$-measurable such that ${\big((\pi_2)_\sharp\hat\mu\,,\, (g_k)_{k\in\R}\big)}$ is a disintegration of $\mu$ $($in the $k$-variable$)$.
	Moreover  for $(\pi_2)_\sharp\mu$ a.e. $k\in\R$
	\begin{equation*}
		g_k\in W^{1,1}(0,1) \quad\text{ with }\quad
		\dot g_k=\frac{\di\mu_{,x}}{\di\mu}(\cdot,k)g_k\mathcal L^1\,.
	\end{equation*}
\end{lemma}
\begin{proof}
	By the Disintegration Theorem (cf. \cite[Theorem2.28]{AmFuPa:00}) there exists $k\mapsto\nu_k\in \mathcal{M}^+_b(0,1)$ $(\pi_2)_\sharp\mu$-measurable with $\nu_k(0,1)=1$ such that
	\begin{equation}\label{eq:disin}
\int_{(0,1)\times\R}f(x,k)\di\mu(x,k)=\int_\R\int_0^1 f(x,k)\di\nu_k(x)\di(\pi_2)_\sharp\mu(k)\,,
	\end{equation}
	for all $f\in L^1((0,1)\times\R;\mu)$.
Let $\varphi(x,k)=\phi(x) \mathbbm{1}_A(k)$ with $\phi\in C^\infty_c(0,1)$ and $A\subset\R$ bounded and measurable. Then, being $\varphi_{,x}(x,k)=\dot\phi(x) \mathbbm{1}_A(k)$, from Remark \ref{rem:int-byparts} \ref{rem-i} 
 we have
\begin{equation*}
\begin{split}
\int_{(0,1)\times\R}\varphi\di\mu_{,x}=-\int_{(0,1)\times\R}\varphi_{,x}\di\mu
=-\int_\R\int_0^1 \varphi_{,x}\di\nu_k(x)\di(\pi_2)_\sharp\mu(k)\,.
\end{split}
\end{equation*}
Moreover, since $\mu_{,x}\ll\mu$,
\begin{equation*}
	\begin{split}
\int_{(0,1)\times\R}\varphi\di\mu_{,x}=
\int_{(0,1)\times\R}\varphi \frac{\di\mu_{,x}}{\di\mu} \di\mu= 
\int_\R\int_0^1 \varphi \frac{\di\mu_{,x}}{\di\mu} 
\di\nu_k(x)\di(\pi_2)_\sharp\mu(k)\,.
	\end{split}
\end{equation*}
	Therefore we deduce
	\begin{equation*}
	-\int_A \int_0^1 \dot\phi(x)\di\nu_k(x) \di(\pi_2)_\sharp\mu(k)= 
	\int_A\int_0^1 \phi(x) \frac{\di\mu_{,x}}{\di\mu} 
	\di\nu_k(x)\di(\pi_2)_\sharp\mu(k)\,.
	\end{equation*}
	By the arbitrariness of $A$ this implies
	\begin{equation*}
	- \int_0^1 \dot\phi(x)\di\nu_k(x)=\int_0^1 \phi(x) \frac{\di\mu_{,x}}{\di\mu}(\cdot,k) 
	\di\nu_k(x)\quad \text{ for }\ (\pi_2)_\sharp\mu\ \text{ a.e. } k\in\R\,,
	\end{equation*}
	from which, in turn, it follows that $(\nu_k)_{,x}\ll\nu_k$ and 
	\begin{equation}\label{eq:density}
	(\nu_k)_{,x}= \frac{\di\mu_{,x}}{\di\mu}(\cdot,k)\, \nu_k \quad \text{ for }\ (\pi_2)_\sharp\mu \ \text{ a.e. } k\in\R\,.
	\end{equation}
	We next claim that \eqref{eq:density} implies  that $\nu_k\ll\mathcal{L}^1\res(0,1)$ and there exists $g_k\in W^{1,1}(0,1)$ such that $\nu_k=g_k\mathcal{L}^1\res(0,1)$ for  $(\pi_2)_\sharp\mu $ a.e. $k\in\R$. This is enough to conclude the proof since \eqref{eq:disin} becomes
		\begin{equation*}
		\int_{(0,1)\times\R}f(x,k)\di\mu(x,k)=\int_\R\int_0^1 f(x,k)g_k(x)\dx\di(\pi_2)_\sharp\mu(k)\,.
	\end{equation*}
It remains to show the claim. Let $\delta>0$ and let $\rho_\delta(x)=\frac1\delta\rho(\frac x\delta)$ be a smooth mollifier at scale $\delta$. From \eqref{eq:density} we have
\begin{equation*}
	(\nu_k*\rho_\delta)_{,x}= \Big(\frac{\di\mu_{,x}}{\di\mu}(\cdot,k)\, \nu_k \Big)*\rho_\delta\quad \text{ for }\ (\pi_2)_\sharp\mu \ \text{ a.e. } k\in\R\,,
\end{equation*}
where here is implicitly assumed $\nu_k$ to be extended to 0 in $(0,1)^c$. Thus 
\begin{equation*}
|\nu_k*\rho_\delta(x)|=\left|\int_{-\infty}^x\Big(\frac{\di\mu_{,x}}{\di\mu}(\cdot,k)\, \nu_k \Big)*\rho_\delta\dt\right|\le  \int_0^1\Big| \frac{\di\mu_{,x}}{\di\mu}(\cdot,k)
\Big|\di\nu_k<+\infty\,,
\end{equation*}
so that, in particular $\nu_k*\rho_\delta\in L^\infty(0,1)$. This together with  $\nu_k*\rho_\delta\stackrel{*}{\rightharpoonup}\nu_k$, this imply $\nu_k= g_k\mathcal{L}^1\res(0,1)$ for some $g_k\in L^\infty(0,1)$. In addition, given $\psi\in C_c^\infty(0,1)$, 
\begin{equation*}
\int_0^1\dot\psi g_k\dx=\int_0^1\dot\psi\di\nu_k=-\int_0^1\psi\frac{\di\mu_{,x}}{\di\mu}(\cdot,k)\di\nu_k\,,
\end{equation*}
from which we infer $\dot g_k=\frac{\di\mu_{,x}}{\di\mu}(\cdot,k)\nu_k=\frac{\di\mu_{,x}}{\di\mu}(\cdot,k)g_k\mathcal L^1$. Eventually by Young's inequality
\begin{equation*}
	\begin{split}
	\int_0^1|\dot g_k(x)|\dx&= \int_0^1\frac{|\dot g_k(x)|}{\sqrt2 k\sqrt{g_k(x)}}\sqrt2k\sqrt{g_k(x)}\dx\\&\le \int_0^1 \left[k^2g_k(x)+ \frac1{4k^2} \frac{(\dot g_k(x))^2}{g_k(x)}\right]
	\dx\\
	&\le 
	\int_0^1\left[k^2+\frac1{4k^2}\Big(\frac{\di\mu_{,x}}{\di\mu}(\cdot,k)\Big)^2
	\right]g_k(x)\dx<+\infty
	\end{split}
\end{equation*}
for $(\pi_2)_\sharp\mu$ a.e. $k\in\R$, and thus in particular $g_k\in W^{1,1}(0,1)$.

\end{proof}
W are now ready to prove the main result of this section.
\begin{proof}[Proof of Theorem \ref{thm:minimizers}]
	By Lemma \ref{lem:prop-F} we know there exists $\hat\mu\in\mathcal{M}_\infty$ minimizer of $\mathcal{F}_\infty$. Moreover by Lemma \ref{lem-disint-k} there exists $k\mapsto g_k$ $(\pi_2)_\sharp\hat\mu$ measurable with $g_k\in W^{1,1}(0,1)$ for $(\pi_2)_\sharp\hat\mu$ a.e. $k\in\R$, such that ${\big((\pi_2)_\sharp\hat\mu\,,\, (g_k)_{k\in\R}\big)}$ is a disintegration of $\hat\mu$. Therefore, in particular we can rewrite 
	\begin{equation*}
	\mathcal F_\infty(\hat\mu)= \int_\R \int_0^1\left[k^2+\frac1{4k^2}\Big(\frac{\di\hat\mu_{,x}}{\di\hat\mu}\Big)^2\right]g_k(x)\dx\di(\pi_2)_\sharp\hat\mu\,.
	\end{equation*}
\textit{Step 1: we show \eqref{equipart-section}.}
Assume by contradiction that \eqref{equipart-section} does not hold true. Then the set 
\begin{equation*}
E:=\left\{k\in\R \colon 
\int_0^1k^2g_k(x)\dx\ne \int_0^1 \frac1{4k^2}\Big(\frac{\di\hat\mu_{,x}}{\di\hat\mu}\Big)^2g_k(x)\dx
\right\}
\end{equation*}
is such that $(\pi_2)_\sharp\hat\mu(E)>0$. Assume, without loss of generality, that the subset
\begin{equation*}
	E_1^+:=\left\{k\in\R^+ \colon 
	\int_0^1k^2g_k(x)\dx< \int_0^1 \frac1{4k^2}\Big(\frac{\di\hat\mu_{,x}}{\di\hat\mu}\Big)^2g_k(x)\dx
	\right\}\subset E
\end{equation*}
satisfies $(\pi_2)_\sharp\hat\mu(E_1^+)>0$ (the other cases can be treated in a similar way). 
Since $\int_0^1g_k(x)\dx=1$ we can rewrite
\begin{equation*}
	E_1^+=\left\{k\in\R^+\colon
	1<\frac1{4k^4}\int_0^1 \Big(\frac{\di\hat\mu_{,x}}{\di\hat\mu}\Big)^2g_k(x)\dx
	\right\}\,.
\end{equation*}
Then there exists $\sigma>0$ such that 
\begin{equation*}
	E_\sigma:=\left\{k\in\R^+ \colon 
1+\sigma
\le\frac1{4k^2} \int_0^1 \Big(\frac{\di\hat\mu_{,x}}{\di\hat\mu}\Big)^2g_k(x)\dx
	\right\}\subset E_1^+
\end{equation*}
with $(\pi_2)_\sharp\hat\mu(E_\sigma)>0$. Next fix $\delta>0$ such  that $(1+\delta)^2<1+\sigma$ and let  $\tilde\mu\in\mathcal{M}_b((0,1)\times\R)$ be defined as follows
\begin{equation}
\tilde\mu:=\hat\mu\res(0,1)\times(\R\setminus E_\sigma)+ \mu_\delta\,,
\end{equation}
where $\mu_\delta:=(\tau_\delta)_\sharp\mu\res((0,1)\times E_\sigma)$ is the push-forward of $\mu\res((0,1)\times E_\sigma)$ with respect to the map $\tau_\delta\colon (0,1)\times \R\to (0,1)\times\R$, $\tau_\delta(x,k):=(x,k(1+\delta))$. 
Note that $\tilde\mu$ is a positive measure. Setting $ E_\sigma^\delta:=(1+\delta)E_\sigma$, then the support of $\mu_\delta$ is contained in $\overline{(0,1)\times E_\sigma^\delta}$, and
\begin{equation*}
\int_{(0,1)\times E_\sigma^\delta}f(x,k)\di\mu_\delta= \int_{(0,1)\times E_\sigma}f(x,k(1+\delta))\di\hat\mu
\end{equation*}
for every $f$ summable with respect to $\mu_\delta$. Hence, by duality and using that $\hat\mu_{,x}\ll\hat\mu$ we have 
\begin{equation*}
	\begin{split}
\int_{(0,1)\times E_\sigma^\delta}\varphi(x,k)\di(\mu_\delta)_{,x}= -  \int_{(0,1)\times E_\sigma^\delta}\varphi_{,x}(x,k)\di\mu_\delta &=-   \int_{(0,1)\times E_\sigma}\varphi_{,x}(x,k(1+\delta))\di\hat\mu\\= \int_{(0,1)\times E_\sigma}\varphi(x,k(1+\delta))\di\hat\mu_{,x}
& =\int_{(0,1)\times E_\sigma}\varphi(x,k(1+\delta))\frac{\di\hat\mu_{,x}}{\di\hat\mu}(x,k)\di\hat\mu\\ &
= \int_{(0,1)\times E_\sigma^\delta}\varphi(x,k)\frac{\di\hat\mu_{,x}}{\di\hat\mu}\Big(x,\frac k{1+\delta}\Big)\di\mu_\delta\,,
	\end{split}
\end{equation*}
for all $\varphi\in C_c^\infty((0,1)\times E_\sigma^\delta)$. As a consequence we readily deduce that $(\mu_\delta)_{,x}\ll\mu_\delta$ with 
\begin{equation}\label{eq:density_delta}
\frac{\di (\mu_\delta)_{,x}}{\di\mu_\delta}(x,k)=\frac{\di\hat\mu_{,x}}{\di\hat\mu}\Big(x,\frac k{1+\delta}\Big)\,,
\end{equation}
so that, in particular, $\tilde\mu_{,x}\ll\tilde\mu$. 
Moreover for every $\phi\in C_c^\infty(0,1)$ we have
\begin{equation*}
	\begin{split}
	\int_0^12x\phi(x)\dx=\int_{(0,1)\times\R}\phi(x)\di\hat\mu&= \int_{(0,1)\times E^c_\sigma}\phi(x)\di\hat\mu +  \int_{(0,1)\times E_\sigma}\phi(x)\di\hat\mu\\
	&= \int_{(0,1)\times E^c_\sigma}\phi(x)\di\hat\mu +  \int_{(0,1)\times E^\delta_\sigma}\phi(x)\di\mu_\delta= \int_{(0,1)\times\R}\phi(x)\di\tilde\mu\,,
	\end{split}
\end{equation*}
and thus $\tilde\mu\in\mathcal{M}_\infty$. We next show that 
\begin{equation*}
\mathcal{F}_\infty(\tilde\mu)<\mathcal{F}_\infty(\hat\mu)\,,
\end{equation*}
which contradicts the fact that $\hat\mu$ is a minimizer. 
To this purpose it is convenient to define the localized functional 
\begin{equation*}
\mathcal{F}_\infty(\mu, A):=\int_{(0,1)\times A}\left[k^2+\frac{1}{4k^2}\Big(\frac{\di\mu_{,x}}{\di\mu}\Big)^2\right]\di\mu\,,
\end{equation*}
for any bounded measure $\mu$ with $\mu_{,x}\ll\mu$ and any $A\subset\R$ measurable. Observing that $\tilde\mu=\hat\mu$ on $ (0,1)\times(\R\setminus(E_\sigma\cup E_\sigma^\delta))$ and $\tilde \mu=\mu_\delta$ on $(0,1)\times(E_\sigma^\delta\cap E_\sigma)$ we have
\begin{equation}\label{eq:equipart}
\mathcal{F}_\infty(\tilde\mu)= \mathcal{F}_\infty(\hat\mu, \R\setminus(E_\sigma\cup E_\sigma^\delta))+ \mathcal{F}_\infty(\tilde\mu, E_\sigma^\delta\setminus E_\sigma)+ \mathcal{F}_\infty(\mu_\delta, E_\sigma^\delta\cap E_\sigma)\,.
\end{equation}
By Lemma \ref{lem:prop-F} we know that $\mathcal{F}_\infty$ is convex and 1-homogeneous, which together with $$\tilde\mu=\frac{2\hat\mu+2\mu_\delta}{2}\quad\text{ on }\quad (0,1)\times(E_\sigma^\delta\setminus E_\sigma)\,,$$
yield
\begin{equation*}
	\begin{split}
\mathcal{F}_\infty(\tilde\mu, E_\sigma^\delta\setminus E_\sigma)&= \mathcal{F}_\infty(\tilde\mu\res ((0,1)\times E_\sigma^\delta\setminus E_\sigma))\\&\le 
\mathcal{F}_\infty(\hat\mu\res ((0,1)\times E_\sigma^\delta\setminus E_\sigma))+
\mathcal{F}_\infty(\mu_\delta\res ((0,1)\times E_\sigma^\delta\setminus E_\sigma))\\
& = \mathcal{F}_\infty(\hat\mu,( E_\sigma^\delta\setminus E_\sigma))+
\mathcal{F}_\infty(\mu_\delta,( E_\sigma^\delta\setminus E_\sigma))\,.
	\end{split}
\end{equation*}
Combining this together with \eqref{eq:equipart} we get
\begin{equation*}
	\begin{split}
\mathcal{F}_\infty(\tilde\mu)&\le\mathcal{F}_\infty(\hat\mu,\R\setminus E_\sigma)
+ \mathcal{F}_\infty(\mu_\delta, E_\sigma^\delta)\\
& = \mathcal{F}_\infty(\hat\mu)+ \mathcal{F}_\infty(\mu_\delta, E_\sigma^\delta)
- \mathcal{F}_\infty(\hat\mu, E_\sigma)
\,. 
	\end{split}
\end{equation*}
Therefore we would conclude the proof if we show that 
\begin{equation}\label{strict_inequality}
\mathcal{F}_\infty(\mu_\delta, E_\sigma^\delta)
- \mathcal{F}_\infty(\hat\mu, E_\sigma)<0\,.
\end{equation}
By the  change of variable $k=\hat k(1+\delta)$ and recalling \eqref{eq:density_delta} it holds
\begin{equation*}
\begin{split}
\mathcal{F}_\infty(\mu_\delta, E_\sigma^\delta)      & =   \int_{(0,1)\times E_\sigma^\delta}\left[k^2+\frac{1}{4k^2}\Big(\frac{\di(\mu_\delta)_{,x}}{\di\mu_\delta}\Big)^2\right]\di\mu_\delta\\
&= 
 \int_{(0,1)\times E_\sigma}\left[k^2(1+\delta)^2+\frac{1}{4k^2(1+\delta)^2}\Big(\frac{\di\hat\mu_{,x}}{\di\hat\mu}\Big)^2\right]\di\hat\mu\,,
\end{split}
\end{equation*}
from which it follows
\begin{equation*}
	\begin{split}
	\mathcal{F}_\infty(\mu_\delta, E_\sigma^\delta)
- \mathcal{F}_\infty(\hat\mu, E_\sigma)&= \int_{(0,1)\times E_\sigma}\left[k^2((1+\delta)^2-1)+\frac{1-(1+\delta)^2}{4k^2(1+\delta)^2}\Big(\frac{\di\hat\mu_{,x}}{\di\hat\mu}\Big)^2\right]\di\hat\mu\\ 
& =  
\frac{(1+\delta)^2-1}{(1+\delta)^2}
\int_{(0,1)\times E_\sigma}\left[k^2(1+\delta)^2-\frac{1}{4k^2}\Big(\frac{\di\hat\mu_{,x}}{\di\hat\mu}\Big)^2\right]\di\hat\mu
\,.
	\end{split}
\end{equation*}
Being $\delta>0$ we have that $((1+\delta)^2-1)/(1+\delta)^2>0$.  Moreover, by disintegration we can rewrite the integral as
\begin{equation*}
	\begin{split}
		\int_{(0,1)\times E_\sigma}&\left[k^2(1+\delta)^2-\frac{1}{4k^2}\Big(\frac{\di\hat\mu_{,x}}{\di\hat\mu}\Big)^2\right]\di\hat\mu
		\\&=  \int_{E_\sigma} \int_0^1 \left[k^2(1+\delta)^2-\frac{1}{4k^2}\Big(\frac{\di\hat\mu_{,x}}{\di\hat\mu}\Big)^2\right]g_k(x)\dx\di(\pi_2)_\sharp\hat\mu\,.
	\end{split}
\end{equation*}
The above quantity is strictly negative if
\begin{equation*}
	 \int_0^1 \left[k^2(1+\delta)^2-\frac{1}{4k^2}\Big(\frac{\di\hat\mu_{,x}}{\di\hat\mu}\Big)^2\right]g_k(x)\dx<0\quad \text{ for }\ (\pi_2)_\sharp\hat\mu \ \text{ a.e. }\ k\in E_\sigma\,.
\end{equation*}
From $\int_0^1g_k(x)\dx=1$, this is equivalent to 
\begin{equation*}
(1+\delta)^2<\frac1{4k^4}\int_0^1 \Big(\frac{\di\hat\mu_{,x}}{\di\hat\mu}\Big)^2g_k(x)\dx\quad \text{ for }\ (\pi_2)_\sharp\hat\mu \ \text{ a.e. }\ k\in E_\sigma\,,
\end{equation*}
which holds thanks to the choice of $\delta$ and the definition of $E_\sigma$, and thus we infer \eqref{strict_inequality}.\\

\noindent
\textit{Step 2: we show \eqref{small-k}.} By Lemma \ref{lem-disint-k} we have $\dot g_k=\frac{\di\hat\mu_{,x}}{\di\hat\mu}(\cdot,k)g_k\mathcal{L}^1$, so that from \eqref{equipart-section} we have
\begin{equation}\label{equality}
		k^2\int_0^1(\sqrt{g_k(x)})^2\dx= 
	k^2\int_0^1g_k(x)\dx= \frac1{4k^2}\int_0^1\frac{(\dot g_k(x))^2}{g_k(x)}\dx= \frac1{k^2}\int_0^1\Big(\frac{\di}{\dx}\sqrt{g_k(x)}\Big)^2\dx\,.
\end{equation}
Since $\hat\mu\in\mathcal M_\infty$, then in particular $\hat\mu(\{0\}\times\R)=0$ from which it follows $g_k(0)=0$ for $(\pi_2)_\sharp\hat\mu$ a.e. $k\in\R$.  Thus we apply Poincaré's inequality to get
\begin{equation*}
\int_0^1(\sqrt{g_k(x)})^2\dx\le C \int_0^1\Big(\frac{\di}{\dx}\sqrt{g_k(x)}\Big)^2\dx\,,
\end{equation*}
for some constant  $C>0$.  Now combining the above inequality with \eqref{equality} we find that 
\begin{equation*}
\int_0^1(\sqrt{g_k(x)})^2\dx\le Ck^4 \int_0^1(\sqrt{g_k(x)})^2\ \iff \ k^4\ge C^{-1}\,.
\end{equation*}
Hence \eqref{equipart-section} holds true if $k^4\ge C^{-1}$ which in turn implies $(\pi_2)_\sharp\hat\mu(|k|\le C^{-1/4})=0$.

\end{proof}

\section*{Acknowledgement}

This work was done while the second author was postdoc at TU Dortmund. Both authors were partially supported by the German Science Foundation DFG in context of the Emmy
Noether Junior Research Group BE 5922/1-1. RM has been partially funded by the European Union - NextGenerationEU under the Italian Ministry of University and Research (MUR) National Centre for HPC, Big Data and Quantum Computing CN\_00000013 - CUP: E13C22001000006.\\
PB also thanks Carlos Rom\'an and Alaa Elshorbagy for discussions at the early stages of the project. 


\providecommand{\bysame}{\leavevmode\hbox to3em{\hrulefill}\thinspace}
\providecommand{\MR}{\relax\ifhmode\unskip\space\fi MR }
\providecommand{\MRhref}[2]{%
  \href{http://www.ams.org/mathscinet-getitem?mr=#1}{#2}
}
\providecommand{\href}[2]{#2}

%
%

\end{document}